\documentclass{amsart}
\usepackage[margin=1.0in]{geometry}

\usepackage{inputenc}
\usepackage{amsmath}
\usepackage{amssymb, amsfonts}
\usepackage[mathscr]{eucal}
\usepackage{graphicx,epstopdf,bm}
\usepackage{amsaddr}
\usepackage[caption=false]{subfig}
\usepackage{dcolumn}%
\usepackage[title]{appendix}
\usepackage[colorlinks,bookmarks,urlcolor=blue,citecolor=blue,linkcolor=blue]{hyperref}
\usepackage[capitalize]{cleveref}

\usepackage{bbm}
\usepackage{mathtools}
\usepackage{multirow}
\usepackage{xcolor}

\usepackage{algorithm}
\usepackage{algpseudocode}

\usepackage{amsthm}
\numberwithin{equation}{section}
\theoremstyle{plain}
\newtheorem{theorem}{Theorem}[section]
\newtheorem{definition}{Definition}[section]
\newtheorem{proposition}[theorem]{Proposition}
\newtheorem{lemma}[theorem]{Lemma}
\newtheorem{assumption}[theorem]{Assumption}
\newtheorem{corollary}[theorem]{Corollary}
\newtheorem{example}[theorem]{Example}
\newtheorem{remark}[theorem]{Remark}
\allowdisplaybreaks

\usepackage{bbm}

\renewcommand{\vec}[1]{\boldsymbol{#1}}
\newcommand{\vecs}[1]{\boldsymbol{#1}}
\newcommand{\paren}[1]{\left(#1\right)}
\newcommand{\brac}[1]{\left[#1\right]}
\newcommand{\E}{\mathbb{E}}
\newcommand{\avg}[1]{\E[#1]}

\newcommand{\D}[2]{\frac{d#1}{d#2}}
\newcommand{\PD}[2]{\frac{\partial#1}{\partial#2}}

\newcommand{\lap}[1]{\Delta#1}

\newcommand{\abs}[1]{\left|#1\right|}

\DeclareMathOperator{\prob}{Pr}

\newcommand{\DA}{D^{\textrm{A}}}
\newcommand{\DB}{D^{\textrm{B}}}
\newcommand{\DC}{D^{\textrm{C}}}

\newcommand{\veta}{\vecs{\eta}}

\newcommand{\vP}{\vec{P}}
\newcommand{\vG}{\vec{G}}

\newcommand{\vx}{\vec{x}}
\newcommand{\vy}{\vec{y}}
\newcommand{\vz}{\vec{z}}

\newcommand{\vw}{\vec{w}}
\newcommand{\vxi}{\vec{\xi}}

\newcommand{\vq}{\vec{q}}

\newcommand{\vqa}{\vec{q}^{a}}
\newcommand{\vqb}{\vec{q}^{b}}
\newcommand{\vqc}{\vec{q}^{c}}
\newcommand{\vQ}{\vec{Q}}
\newcommand{\vQa}{\vec{Q}^{a}}
\newcommand{\vQb}{\vec{Q}^{b}}
\newcommand{\vQc}{\vec{Q}^{c}}

\def\N{\mathbb{N}}
\def\R{\mathbb{R}}

\renewcommand{\epsilon}{\varepsilon}

\newcommand{\Rp}{\mathcal{R}^{+}}
\newcommand{\Rm}{\mathcal{R}^{-}}
\newcommand{\diffop}{\mathcal{L}}

\newcommand{\pabc}{p^{(a,b,c)}}

\newcommand{\gabc}{g^{(a,b,c)}}

\newcommand{\Kwm}{K_{\textrm{wm}}}

\newcommand{\Def}{\overset{\text{def}}{=}}

\newcommand{\BP}{\mathbb{P}}
\newcommand{\BE}{\mathbb{E}}
\newcommand{\filt}{\mathscr{F}}

\newcommand{\genA}{\mathcal{A}}
\newcommand{\genL}{\mathcal{L}}

\newcommand{\lb}{\left\{}
\newcommand{\rb}{\right\}}
\newcommand{\la}{\left \langle}
\newcommand{\ra}{\right\rangle}

\newcommand{\RN}[1]{
  \textup{\uppercase\expandafter{\romannumeral#1}}%
}

\makeatletter
\def\mathcolor#1#{\@mathcolor{#1}}
\def\@mathcolor#1#2#3{%
  \protect\leavevmode
  \begingroup
    \color#1{#2}#3%
  \endgroup
}
\makeatother

\begin{document}
\title{Mean Field Limits of Particle-Based Stochastic Reaction-Diffusion Models}
\author{S. Isaacson, J. Ma and K. Spiliopoulos}
\footnote{Boston University, Department of Mathematics and Statistics, 111 Cummington Mall, Boston, 02215.
E-mails: isaacson@math.bu.edu, majw@bu.edu, kspiliop@math.bu.edu}
\thanks{K.S. was partially supported by the National Science Foundation DMS-1550918, DMS-2107856, SIMONS Foundation Award 672441 and ARO W911NF-20-1-0244. J. M. and S. A. I. were partially supported by National Science Foundation DMS-1902854,  DMS-0920886  and ARO W911NF-20-1-0244.}

\begin{abstract}
Particle-based stochastic reaction-diffusion (PBSRD) models are a popular approach for studying biological systems involving both noise in the reaction process and diffusive transport. In this work we derive coarse-grained deterministic partial integro-differential equation (PIDE) models that provide a mean field approximation to the volume reactivity PBSRD model, a model commonly used for studying cellular processes.  We formulate a weak measure-valued stochastic process (MVSP) representation for the volume reactivity PBSRD model, demonstrating for a simplified but representative system that it is consistent with the commonly used Doi Fock Space representation of the corresponding forward equation. We then prove the convergence of the general volume reactivity model MVSP to the mean field PIDEs in the large-population (i.e. thermodynamic) limit.
\end{abstract}
\date{\today}
\maketitle

\section{Introduction}

The dynamics of many biological processes rely on an interplay between spatial transport and chemical reaction. At the scale of a single cell, experiments have demonstrated that many such processes have stochastic dynamics. Particle-based stochastic reaction-diffusion (PBSRD) models are a widely used approach for studying such processes, explicitly modeling the diffusion of, and reactions between, individual molecules. PBSRD models are appropriate for studying chemical systems in cells containing up to hundreds of thousands to millions of molecules, over timescales of days.  They are more macroscopic descriptions than millisecond-timescale quantum mechanical or molecular dynamics models of a few molecules~\cite{ShawAntonMS2009}, but more microscopic descriptions than deterministic 3D reaction-diffusion PDEs for the average concentration of each species of molecule. One PBSRD model that has been widely used to study biological processes is the volume reactivity (VR) model of Doi~\cite{TeramotoDoiModel1967,DoiSecondQuantA,DoiSecondQuantB}. In this model positions of \emph{individual molecules} are typically represented as points undergoing Brownian motion.  Bimolecular reactions between two reactant molecules occur with a probability per unit time based on their current positions~\cite{DoiSecondQuantA,DoiSecondQuantB}. Unimolecular reactions are typically assumed to represent internal processes, and as such are modeled as occurring with exponentially distributed times based on a specified reaction-rate constant.

Due to their mathematical complexity and high dimensionality, PBSRD models are almost entirely studied by Monte Carlo simulation  approximating the underlying stochastic process of molecules diffusing and reacting. The computational expense of such methods can greatly limit the size of chemical systems (in each of number of molecules, number of reactions, or physical domain size) that can be studied. One approach to overcoming this challenge is to use more coarse-grained mathematical models that accurately capture the dynamics of the underlying PBSRD model in appropriate physical regimes. Deterministic and stochastic partial differential equation (PDE/SPDE) models are often postulated as coarse-grainings of PBSRD models in certain large-population or thermodynamic limits where the population size becomes unbounded but species concentrations are held fixed. However, for the PBSRD models commonly used in biological modeling, e.g. the VR model and the contact reactivity model, there is limited \emph{rigorous} work identifying and proving the existence of such deterministic coarse-grained limits (i.e. law of large numbers).

To facilitate the development of rigorous coarse-grainings of PBSRD models, our work begins with formulating the dynamics of the diffusing and reacting molecules as measure-valued stochastic processes (MVSPs). These processes describe the evolution of the concentration fields of each chemical species as a sum of $\delta$ functions in each molecule's position. A weak formulation of the dynamics of these processes is then derived, giving the action of the processes on an arbitrary test function. The subsequent equations for the time-evolution of the pairing between a test function and the MVSP then involve both continuous noise processes that account for the diffusion of individual molecules, and state-dependent Poisson-random measures that encode the timing and occurrence of chemical reactions between molecules. We establish in a simplified, but representative, case that the MVSP is equivalent to the commonly used Doi Fock Space representation for the forward equation of the VR model.

We then investigate the large population limit of the MVSP dynamics in which the initial number of molecules of each chemical species becomes unbounded, but the concentrations of each species are held fixed. The latter can be achieved by considering molar concentrations, and treating Avogadro's number and/or the domain volume as a large ``system size'' parameter. As we work in free space, here we consider the limit where Avogadro's number can be considered a large parameter. Such limits are typically considered one of the primary physical regimes in which PDE or SPDE models for biological systems arise as physical approximations to the underlying process of molecules diffusing and reacting~\cite{Arnold1980dl,Nolen2019,Schutte2019}.

To rigorously determine the limit of the VR model, we will generalize the martingale problem approach for studying solutions to stochastic differential equations developed by Stroock and Varadhan~\cite{EthierKurtz, StroockVaradhan} to our weak MVSP representation. Adaptations of this method have been successfully used to study large-population limits in stochastic models for population dynamics, evolutionary dynamics, interacting particle systems, and financial models \cite{TypicalDefaults,LargePortfolio,DaiPra2,DaiPra3,Delarue,Inglis,Moynot,Sompolinsky}. We first identify a macroscopic system of partial integro-differential equations (PIDEs) whose solution corresponds to the large population limit of the MVSP, and then rigorously prove the convergence (in a weak sense) of the MVSP to this solution.

Our approach is unique in using a bottom-up hierarchy to rigorously derive from spatial PBSRD models \emph{new} macroscopic PIDEs that correspond to the true large-population limit, and which correctly account for chemical interactions between particles. This is in contrast to the standard macroscopic reaction-diffusion PDE models of chemical reaction systems used at the cellular scale~\cite{IyengarCellShape08,MunozGarcia:2009hd}. The latter are typically obtained by formally modifying reaction-rate equation ODE models for \emph{non-spatial} chemical reaction systems by simply adding Laplacian terms to model molecular diffusion.

To illustrate our main result, consider the special case of the reversible reaction $\textrm{A} + \textrm{B} \leftrightarrows \textrm{C}$ reaction. Let $\gamma$ denote a system size parameter (i.e. Avogadro's number, or in bounded domains the product of Avogadro's number and the domain volume). We assume all molecules move by Brownian Motion in $\R^d$, with species-dependent diffusivities, $\DA$, $\DB$ and $\DC$ respectively. Let $K_1^{\gamma}(x,y) = K_{1}(x,y)/\gamma$ denote the probability per time an individual \textrm{A} molecule at $x$ and \textrm{B} molecule at $y$ can react, with $m_{1}(z|x,y)$ giving the probability density that when the \textrm{A} and \textrm{B} molecules react they produce a \textrm{C} molecule at $z$. We define $K_2^{\gamma}(z) = K_2(z)$ and $m_2(x,y|z)$ similarly for the reverse reaction. Finally, denote by $A(t)$ the stochastic process for the number of species \textrm{A} molecules at time $t$, and label the position of $i$th molecule of species \textrm{A} at time $t$ by the stochastic process $\vQ_i^{A(t)} \subset \R^d$. The random measure
\begin{equation*}
A^{\gamma}(x,t) = \frac{1}{\gamma} \sum_{i=1}^{A(t)} \delta\paren{x - \vQ_i^{A(t)}}
\end{equation*}
corresponds to the stochastic process for the the molar concentration of species \textrm{A} at $x$ at time $t$. We can similarly define $B^{\gamma}(x,t)$ and $C^{\gamma}(x,t)$. In this work we study the large population (thermodynamic) limit where $\gamma \to \infty$ and $A^{\gamma}(x,0)$ converges to a well defined limiting molar concentration field (with similar limits for the molar concentrations of species \textrm{B} and \textrm{C}). We prove, in a weak sense, that as $\gamma \to \infty$,
\begin{align*}
 \paren{A^{\gamma}(x,t),B^{\gamma}(x,t),C^{\gamma}(x,t)} \to \paren{\bar{A}(x,t), \bar{B}(x,t), \bar{C}(x,t)},
\end{align*}
where
\begin{align*}
  \partial_t \bar{A}(x,t) &= \DA \lap \bar{A}(x,t) - \int_{\R^d} K_1(x,y) \bar{A}(x,t) \bar{B}(y,t) \, dy
  + \int_{\R^{2d}} K_2(z) m_2(x,y|z) C(z,t) \,dy \, dz.\\
  \partial_t \bar{B}(x,t) &= \DB \lap \bar{B}(x,t) - \int_{\R^d} K_1(x,y) \bar{A}(y,t) \bar{B}(x,t) \, dy
  + \int_{\R^{2d}} K_2(z) m_2(x,y|z) C(z,t) \,dy \, dz.\\
  \partial_t \bar{C}(x,t) &= \DC \lap \bar{C}(x,t) - K_2(z) C(z,t) + \int_{\R^{2d}} K_1(x,y) m_1(z|x,y) \bar{A}(x,t) \bar{B}(y,t) \, dx \, dy.
\end{align*}

Our main result, Theorem~\ref{thm:convergence}, establishes this rigorous limit for the VR PBSRD model of general chemical reaction systems involving first and second order reactions. To simplify the (already detailed) exposition, we impose one constraint, assuming that the reaction network structure is such that the total concentration of molecules in the system has a strict upper bound. Theorem~\ref{thm:convergence} therefore does not cover reaction systems containing reactions that can lead to unbounded population growth, ruling out creation reactions like $\varnothing \to \textrm{A}$ and $\textrm{A} \to 2\textrm{A}$. We note, however, that we expect our basic approach should be adaptable to such systems, but would require the introduction \and analysis of a stopping time for when the total population of molecules reaches some threshold (see the discussion of Remark~\ref{R:StoppingTime}). This is similar to one approach used for proving the classical large-population limit of \emph{non-spatial} stochastic chemical kinetic systems~\cite{DK:2015}.

Upon completion of this work we became aware of the recent publication \cite{Nolen2019}. In~\cite{Nolen2019} the authors study the rigorous large-population limit for a subset of the reactions we allow in our VR PBSRD model, restricting to reactions of the form $\textrm{A} + \textrm{B} \rightarrow \textrm{C} + \textrm{D}$. This ensures that the total number of particles is preserved for all time in their system, allowing~\cite{Nolen2019} to formulate a strong-form pathwise representation for the evolution of the stochastic processes for particle positions and types. Particle diffusion and particle-particle reactions are then represented through separate but coupled equations. This formulation enabled the authors to use relative entropy methods to prove propagation of chaos. In contrast, we consider general reaction networks as are needed to model many biological processes, in which the number of particles in the system changes over time. We therefore work with a weak formulation, studying the empirical distribution of particle position and type directly. As such, while~\cite{Nolen2019} could leverage relative entropy methods and large deviations estimates to establish their results and even get quantitative estimates, we will work with the general martingale problem formulation.

The large-population limit was also studied in~\cite{Oelschlager89a} for PBSRD systems involving only birth reactions ($\textrm{A} \to 2\textrm{A}$), death reactions ($\textrm{A} \to \varnothing$), and first-order conversion reactions ($\textrm{A} \to \textrm{B}$). For such linear reaction systems the resulting mean-field, large-population limit is a system of linear, local reaction-diffusion PDEs. Our high level formulation for the system dynamics more closely follows~\cite{Oelschlager89a}, which also studied a weak MVSP representation for the reaction and diffusion of particles. We stress, however, that a key difference in our work is in allowing for general second order reactions, i.e. reactions of the form $\textrm{A} + \textrm{B} \to \cdots$, where the reaction dynamics critically depend on spatial interactions between two individual reactant particles. Such reactions are prevalent in most cellular signaling processes, and common in many chemical network models for biological systems. By including bimolecular reactions, formulation of the underlying equations describing the particle dynamics is complicated by the need to model the two-body interactions between particles, and to model the placement of reaction products in space given the positioning of reactants. The resulting large-population limit becomes a system of nonlinear and non-local partial integral differential equations. Allowing for both changing particle numbers and bimolecular reactions results in the mathematical formulation of the problem, as presented in Sections \ref{S:GeneralSetup} and \ref{S:GeneratorPathLevelDescription}, being more involved than in~\cite{Oelschlager89a}, necessitating a number of technical estimates (given in Appendix \ref{A:appendixA}).

An interesting future research direction is to obtain quantitative convergence results for general reaction networks with uneven inputs and outputs, i.e. for reaction networks (like the ones studied in this paper) where the total number of particles are not conserved in time. To do so, generalizations of the propagation of chaos techniques used in \cite{Mischlet2015} and \cite{Nolen2019} seem likely to be necessary, as such techniques largely rely on conservation of the total number of particles
in time, see also the related discussion in \cite{Nolen2019}. We leave this question for future work.

The paper is organized as follows. In Section~\ref{S:GeneralSetup} we describe the problem in more mathematical terms, introducing basic notation for specifying chemical reaction systems, for describing system state as a MVSP for the (number) density of particles in the system, and for representing reactant (product) configuration spaces that encode possible positions of individual reactant (product) particles involved in a reaction. In Section~\ref{S:GeneratorPathLevelDescription} we define reaction kernels specifying the probability per time a reaction involving specific reactants can occur. For each reaction type we also specify a placement density. These give the probability density that product particles of a reaction between one or more reactants are placed at specific positions. We then introduce the stochastic equation describing the evolution of the empirical measure (MVSP) of the chemical species in path space. In Section~\ref{S:assumptions} we summarize the basic assumptions we make about the form of the reaction rate functions and product placement densities. In Section~\ref{S:MainResult} we present our main result, Theorem~\ref{thm:convergence}, describing for general reaction networks the evolution equation satisfied in the large-population limit by the empirical measures for the molar concentration of each species. We also present a number of illustrative examples showing the derived large-population limit for specific chemical systems. In Section~\ref{S:EquivalentDerivations} we prove that the MVSP formulation we study is equivalent to the more commonly used Fock-space (i.e. Kolmogorov forward equation) representation popularized by Doi~\cite{DoiSecondQuantA,DoiSecondQuantB}, focusing on the simplified case of the reversible $\textrm{A} + \textrm{B} \leftrightarrows \textrm{C}$ reaction. Finally, in Section~\ref{S:ProofMainResult} we give the proof of Theorem~\ref{thm:convergence}. The appendix includes proofs of a number of technical estimates as well as the existence, uniqueness and regularity statement for the forward Kolmogorov equation of the $\textrm{A} + \textrm{B} \leftrightarrows \textrm{C}$ reaction system studied in Section~\ref{S:EquivalentDerivations}.

\section{Notation and preliminary definitions}\label{S:GeneralSetup}

We consider a collection of particles with $J$ possible different types. Note, in the following we will interchangeably use particle or molecule and type or species. Let $\mathscr{S} = \{ S_{1},\cdots, S_{J}\}$ denote the set of different possible particle types, with $p_i \in \mathscr{S} $ the value of the type of the $i$-th particle. In the remainder, we also assume an underlying probability triple, $(\Omega,\filt,\BP)$, on which all random variables are defined.

The goal of this paper is to study the process that molecules diffuse in space $\R^d$ freely and undergo at most $L$ possible different type of reactions, denoted as $\mathscr{R}_1, \cdots , \mathscr{R}_L$. We describe the
$\mathscr{R}_\ell$th reaction,  $\ell \in \{1,\dots,L\}$, by
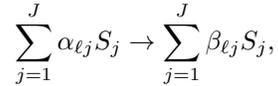
\begin{equation*}
\sum_{j = 1}^{J} \alpha_{\ell j}S_j \rightarrow \sum_{j = 1}^{J} \beta_{\ell j}S_j,
\end{equation*}
where we assume the stoichiometric coefficients $\{\alpha_{\ell j}\}_{j=1}^{J}$ and $\{\beta_{\ell j}\}_{j=1}^J$ are non-negative integers. Let $\vec{\alpha}^{(\ell)} = (\alpha_{\ell 1}, \alpha_{\ell 2}, \cdots, \alpha_{\ell J})$ and $\vec{\beta}^{(\ell)} = (\beta_{\ell 1}, \beta_{\ell 2}, \cdots, \beta_{\ell J})$ be multi-index vectors collecting the coefficients of the $\ell$th reaction. We denote the reactant and product orders of the reaction by $|\vec{\alpha}^{(\ell)}|\doteq\sum_{i = 1}^{J} \alpha_{\ell i} \leq 2$ and $|\vec{\beta}^{(\ell)}|\doteq\sum_{j = 1}^{J} \beta_{\ell j}\leq 2$, assuming that at most two reactants and two products participate in any reaction. We therefore implicitly assume all reactions are at most second order. This is motivated by the observation that the probability three reactants in a dilute system simultaneously have the proper configuration and energy levels to react is very small, so that trimolecular reactions are very rare~\cite{DD:2003}. In biological models, such reactions are often considered to be approximations to sequences of bimolecular reactions. For subsequent notational purposes, we order the reactions such that the first $\tilde{L}$ reactions correspond to those that have no products, i.e. annihilation reactions of the form
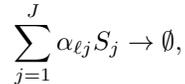
\begin{equation*}
\sum_{j = 1}^{J} \alpha_{\ell j}S_j \rightarrow \emptyset,
\end{equation*}
for $\ell \in \{1,\dots,\tilde{L}\}$. We assume the remaining $L - \tilde{L}$ reactions have one or more product particles.

Let $D^i$ label the diffusion coefficient for the $i$th molecule, taking values in $\{D_1, \dots, D_J\}$, where $D_j$ is the diffusion coefficient for species $S_j$, $ j = 1,\cdots, J$. We denote by $Q^{i}_{t}\in \R^d$ the position of the $i$th molecule, $i \in \N_+$, at time $t$. A particle's state can be represented as a vector in $\hat{P}= \mathbb{R}^{d}\times \mathscr{S}$, the combined space encoding particle position and type. This state vector is subsequently denoted by $\hat Q^i_t \Def (Q^{i}_{t},p_i)$.

We now formulate our representation for the (number) concentration, equivalently number density, fields of each species. Let $E$ be a complete metric space and $M(E)$ the collection of measures on $E$. For  $f: E\mapsto \mathbb{R}$ and $\mu\in M(E)$, define
\begin{equation*}
\la f,\mu\ra_{E} = \int_{x \in E}f(x)\mu(d x).
\end{equation*}
We will frequently have $E=\mathbb{R}^{d}$. In this case we omit the subscript $E$ and simply write $\la f,\mu\ra$. For each $t\ge 0$, we define the concentration of particles in the system at time $t$ by
the distribution
\begin{equation} \label{eq:densitymeasdef}
\nu_t = \sum_{i=1}^{N(t)}\delta_{\hat Q^i_t} = \sum_{i=1}^{N(t)}\delta_{Q^i_t} \delta_{p_i},
 \end{equation}
where, borrowing notation from \cite{BM:2015}, $N(t) = \la 1, \nu_t\ra_{\hat{P}}$ represents the stochastic process for the total number of particles at time $t$. To investigate the behavior of different type of particles, we denote the marginal distribution on the $j$th type, i.e. the concentration field for species $j$, by
\begin{equation*}
\nu^{j}_{t}(\cdot)= \nu^{}_{t}(\cdot\times \{S_j\}),
\end{equation*}
a distribution on $\R^d$. $N_j(t) = \la 1, \nu^j_t\ra$ will similarly label the total number of particles of type $S_j$ at time $t$. For $\nu$ any fixed particle distribution of the form~\eqref{eq:densitymeasdef}, we will also use an alternative representation in terms of the marginal distributions $\nu^j\in M(\R^d)$ for particles of type $j$,
\begin{equation} \label{eq:densitymeasdefmargrep}
\nu = \sum_{j = 1}^{J} \nu^j\delta_{S_j}  \in M(\hat{P}).
\end{equation}

In addition to having notations for representing particle concentration fields, we will also often make use of  state vectors for \emph{all} particles in the system. With some abuse of notation, for $\nu_{t}$ given by~\eqref{eq:densitymeasdef} denote by
\begin{align}
H(\nu_t)=\left((Q^{\sigma_1(1)}_{t},S_{1}),\cdots, (Q^{\sigma_1(N_1(t))}_{t},S_{1}), \cdots \cdots, (Q^{\sigma_J(1)}_{t},S_{r}),\cdots, (Q^{\sigma_J(N_J(t))}_{t},S_{J}),0,0,\cdots\right)\label{Eq:H_definition}
\end{align}
a state vector of the full particle system. Here, for each type $j=1,\dots,J$, the particle index maps $\{\sigma_j(k)\}_{k=1}^{N_j(t)}$ encode a fixed ordering for particles of species $j$, $Q^{\sigma_j(1)}\preceq\cdots\preceq Q^{\sigma_j(N_j(t))}$, arising from an (assumed) fixed underlying ordering on $\mathbb{R}^{d}$. In $H(\nu_{t})$, we order all particles of type 1 by the ordering on $\R^{d}$ first, followed by particles of type 2, then type 3, etc. As particles of the same type are assumed indistinguishable, there is no ambiguity in the value of $H(\nu_{t})$ in the case that two particles of the same type have the same position. $H^i(\nu_t)\in \hat{P}$ will label the $i$th entry of the vector $H(\nu_t)$.
We denote by
\begin{align}
H_Q(\nu^{j}_t)=\left(Q^{\sigma_j(1)}_{t},\cdots, Q^{\sigma_j(N_j(t))}_{t}, 0,0,\cdots\right)
\end{align}
an analogous \emph{position}-only state vector for type $j$ particles, using the same ordering on $\mathbb{R}^{d}$, with $H^i_Q(\nu^j_t)\in \R^d$ labeling the $i$th entry in $H_Q(\nu^j_t)$.

With the preceding definitions, we last introduce a system of notation to encode reactant and particle positions and configurations that are needed to later specify reaction processes.

\begin{definition}\label{def:reacIndSpace}
For reaction $\mathscr{R}_{\ell}$, define the reactant index space
\begin{equation*}
  \mathbb{I}^{(\ell)} = \left(\N\setminus\{0\}\right)^{|\vec{\alpha}^{(\ell)}|},
\end{equation*}
with the convention that if $\abs{\vec{\alpha}^{(\ell)}}=0$ then $\mathbb{I}^{(\ell)}=\varnothing$ is the empty set. In describing the dynamics of $\nu_t$, we will sample vectors containing the indices of the specific reactant particles participating in a single $\ell$-type reaction from $\mathbb{I}^{(\ell)}$. For $\vec{i}\in\mathbb{I}^{(\ell)}$ a particular sampled set of reactant indices, we write
\begin{equation*}
\vec{i}=(i^{(1)}_1, \cdots, i^{(1)}_{\alpha_{\ell 1}}, \cdots, i^{(J)}_1, \cdots, i^{(J)}_{\alpha_{\ell J}}),
\end{equation*}
where $i^{(j)}_r \in \vec{i}$ labels the $r$th sampled index of species type $j$. Here we use the convention that if $\alpha_{\ell j}=0$ no indices are included in $\vec{i}$ for particles of type $j$ (as they do not participate in the $\ell$th reaction as a reactant). Note, as we assumed that $|\vec{\alpha}^{(\ell)}|\doteq\sum_{i = 1}^{J} \alpha_{\ell i} \leq 2$, in practice $\alpha_{\ell i}\in \{0,1,2\}$ and $\vec{i}$ correspondingly identifies zero, one or two reactant particles.
\end{definition}

\begin{definition}\label{def:reacPosSpace}
For reaction $\mathscr{R}_{\ell}$, analogous to our definition of $\mathbb{I}^{(\ell)}$, we define the reactant position space
\begin{equation*}
\mathbb{X}^{(\ell)} = \{ \vec{x} = (x^{(1)}_1, \cdots, x^{(1)}_{\alpha_{\ell 1}}, \cdots, x^{(J)}_1, \cdots, x^{(J)}_{\alpha_{\ell J}}) \, | \,  x_r^{(j)} \in \R^d, \text{ for all } 1\leq j\leq J, 1\leq r\leq \alpha_{\ell j} \}  = \left(\R^d\right)^{|\vec{\alpha}^{(\ell)}|}.
\end{equation*}
Similar to the last definition, when $\alpha_{\ell j} = 0$ particles of species/type $j$ are not involved in reaction $\ell$, and hence not included within possible reactant position vectors. For $\vec{x}\in \mathbb{X}^{(\ell)}$ a sampled reactant position configuration for one individual $\mathscr{R}_{\ell}$ reaction, $x_r^{(j)}$ then labels the sampled position for the $r$th reactant particle of species $j$ involved in the reaction. Let $d\vec{x} = \left( \bigwedge_{j = 1}^J (\bigwedge_{r = 1}^{ \alpha_{\ell j}} d x_r^{(j)}) \right) $ be the corresponding volume form on $\mathbb{X}^{(\ell)}$, which also naturally defines an associated Lebesgue measure. \end{definition}

\begin{definition}\label{def:prodPosSpace}
For reaction $\mathscr{R}_{\ell}$ with $\tilde{L} + 1\leq \ell\leq L$, i.e. having at least one product particle, define the product position space
\begin{equation*}
\mathbb{Y}^{(\ell)} = \{ \vec{y} = (y^{(1)}_1, \cdots, y^{(1)}_{\beta_{\ell 1}}, \cdots, y^{(J)}_1, \cdots, y^{(J)}_{\beta_{\ell J}}) \, | \,  y_r^{(j)} \in \R^d, \text{ for all } 1\leq j\leq J, 1\leq r\leq \beta_{\ell j} \} = \left(\R^d\right)^{|\vec{\beta}^{(\ell)}|}.
\end{equation*}
Analogous to Definition~\ref{def:reacIndSpace}, when $\beta_{\ell j} = 0$ species $j$ is not a product for the $\ell$th reaction, and hence there will be no indices for particles of species $j$ within the product position space. For $\vec{y}\in \mathbb{Y}^{(\ell)}$ a sampled product position configuration for one individual $\mathscr{R}_{\ell}$ reaction, $y_r^{(j)}$ then labels the sampled position for the $r$th product particle of species $j$ involved in the reaction. Let $d\vec{y} = \left( \bigwedge_{j = 1}^J (\bigwedge_{r = 1}^{ \beta_{\ell j}} d y_r^{(j)}) \right) $ be the corresponding volume form on $\mathbb{Y}^{(\ell)}$, which also naturally defines an associated Lebesgue measure.\end{definition}

\begin{definition}\label{def:projMapping}
Consider a fixed reaction $\mathscr{R}_{\ell}$, with $\vec{i}\in \mathbb{I}^{(\ell)}$ and $\nu$ corresponding to a fixed particle distribution given by~\eqref{eq:densitymeasdef} with representation~\eqref{eq:densitymeasdefmargrep}. We define the $\ell$th projection mapping $\mathcal{P}^{(\ell)} :   M(\hat{P})\times  \mathbb{I}^{(\ell)} \rightarrow  \mathbb{X}^{(\ell)}$ as
\begin{equation*}
\mathcal{P}^{(\ell)}(\nu,  \vec{i}) = (H_Q^{i^{(1)}_1}(\nu^1), \cdots, H_Q^{i^{(1)}_{\alpha_{\ell 1}}}(\nu^1) , \cdots , H_Q^{i^{(J)}_1}(\nu^{J}), \cdots, H_Q^{i^{(J)}_{\alpha_{\ell J}}}(\nu^{J})).
\end{equation*}
When reactants with indices $\vec{i}$ in particle distribution $\nu$ are chosen to undergo a reaction of type $\mathscr{\ell}$, $\mathcal{P}^{(\ell)}(\nu,\vec{i})$ then gives the vector of the corresponding reactant particles' positions. For simplicity of notation, in the remainder we will sometimes evaluate $\mathcal{P}^{(\ell)}$ with inconsistent particle distributions and index vectors. In all of these cases the inconsistency will occur in terms that are zero, and hence not matter in any practical way.
\end{definition}

\begin{definition}\label{def:effSampllingSpace}
Consider a fixed reaction $\mathscr{R}_{\ell}$, with $\nu$ a fixed particle distribution given by~\eqref{eq:densitymeasdef} with representation~\eqref{eq:densitymeasdefmargrep}. Using the notation of Def.~\ref{def:reacIndSpace}, we define the allowable reactant index sampling space $\Omega^{(\ell)}(\nu)\subset \mathbb{I}^{(\ell)}$ as
\begin{equation*}
\Omega^{(\ell)}(\nu) = \begin{cases}
\varnothing, & \abs{\vec{\alpha}^{(\ell)}}=0,\\
\{ \vec{i} = i_{1}^{(j)} \in \mathbb{I}^{(\ell)} \, | \, i_1^{(j)} \leq \la 1, \nu^j \ra \}, & \abs{\vec{\alpha}^{(\ell)}}=\alpha_{\ell j} = 1,\\
\{ \vec{i} = (i_{1}^{(j)},i_{2}^{(j)}) \in \mathbb{I}^{(\ell)} \, | \,   i_1^{(j)} < i_{2}^{(j)} \leq \la 1, \nu^j \ra \}, & \abs{\vec{\alpha}^{(\ell)}}=\alpha_{\ell j} = 2,\\
\{ \vec{i} = (i_{1}^{(j)},i_{1}^{(k)})\in \mathbb{I}^{(\ell)} \, | \, i_1^{(j)} \leq \la 1, \nu^j \ra, i_1^{(k)} \leq \la 1, \nu^k \ra \}, & \abs{\vec{\alpha}^{(\ell)}}=2, \quad \alpha_{\ell j} = \alpha_{\ell k} = 1, \quad j < k.\\

\end{cases}
\end{equation*}
Note that in the calculations that follow $\Omega^{(\ell)}(\nu)$ will change over time due to the fact that $\nu=\nu_{t}$ changes over time, but this will not be explicitly denoted for notational convenience.
\end{definition}

\begin{definition}\label{def:lambda}
Consider a fixed reaction $\mathscr{R}_{\ell}$, with $\nu$ any element of $M(\hat{P})$ having the representation~\eqref{eq:densitymeasdefmargrep}. We define the $\ell$th reactant measure mapping $\lambda^{(\ell)} [\, \cdot \,] : M(\hat{P}) \to M( \mathbb{X}^{(\ell)})$ evaluated at $\vx \in \mathbb{X}^{(\ell)}$ via $\lambda^{(\ell)}[\nu](d\vx) = \otimes_{j = 1}^J(\otimes_{ r = 1}^{\alpha_{\ell j }} \nu^j(dx_r^{(j)})) $.
\end{definition}

\begin{definition}\label{def:offDiagSpace}
For blue reaction $\mathscr{R}_{\ell}$, define a subspace $\tilde{\mathbb{X}}^{(\ell)}\subset\mathbb{X}^{(\ell)}$ by removing all particle reactant position vectors in $\mathbb{X}^{(\ell)}$ for which two particles of the same species have the same position. That is
\begin{equation*}
\tilde{\mathbb{X}}^{(\ell)} = \mathbb{X}^{(\ell)}\setminus\{ \vec{x} \in\mathbb{X}^{(\ell)} \, | \,  x_r^{(j)}  =  x_{k}^{(j)} \text{ for some } 1\leq j\leq J, 1\leq k \neq r \leq \alpha_{\ell j} \}.
\end{equation*}
\end{definition}

\section{Generator and process level description}\label{S:GeneratorPathLevelDescription}

Let us consider the time evolution of the process $\nu^{\vec{{\vec{\zeta}}}}_t = \sum_{i=1}^{N^{\vec{\zeta}}(t)}\delta_{\hat Q^i_t}$ which gives the spatial distribution of all particles (i.e. number density or concentration). Here $N^{\vec{\zeta}}(t) = \la 1,\nu^{{\vec{\zeta}}}_{t}\ra_{\hat{P}}$ denotes the total number of particles at time $t$ and ${\vec{\zeta}}= (\frac{1}{\gamma}, \eta)$ is a two-vector consisting of a scaling parameter, $\gamma$, and a displacement range parameter, $\eta$.  In the large population limit we consider $\gamma$ plays the role of a system size, and is considered to be large (e.g. Avogadro's number, or in bounded domains the product of Avogadro's number and the domain volume)~\cite{DK:2015}. On the other hand, $\eta$ is a regularizing parameter allowing us to be able to consider and rigorously handle delta-function placement densities for reaction products (a common choice in many PBSRD simulation methods). We will further clarify these parameters later on,  focusing on the (large-population) limit that $\gamma\to \infty$ and $\eta\to 0$ jointly, denoted as $\vec{\zeta}\to 0$.

To formulate the process-level model, it is necessary to specify more concretely the reaction process between individual particles. For reaction $\mathscr{R}_{\ell}$, denote by $K_\ell^{\gamma}(\vec{x})$ the rate (i.e. probability per time) that reactant particles with positions $\vec{x}\in \mathbb{X}^{(\ell)}$ react. As described in the next section, we assume this rate function has a specific scaling dependence on $\gamma$. Let $m^\eta_\ell( \vec{y} \, | \,  \vec{x})$ be the placement density when the reactants at positions $\vec{x}\in \mathbb{X}^{(\ell)}$ react and generate products at positions $\vec{y}\in \mathbb{Y}^{(\ell)}$. We assume this placement density depends on the displacement range parameter $\eta$, and that for each $\vec{x}$ and fixed $\eta>0$, $m^\eta_\ell( \cdot \, | \,  \vec{x})$ is bounded.

Stochastic particle dynamics involve both diffusive motion and chemical reactions. In describing particle motion we will make use of $\{W^n_t\}_{n\in \N_+}$, a countable collection of standard independent Brownian motions in $\R^d$. To describe a reaction $\mathscr{R}_\ell$ with no products,  i.e. $1\leq \ell \leq \tilde{L}$,  we associate with it a Poisson point measure $dN_\ell(s, \vec{i}, \theta)$  on $\R_+\times \mathbb{I}^{(\ell)} \times\R_+$. Here $\vec{i}\in \mathbb{I}^{(\ell)}$ gives the sampled reactant configuration, with $i^{(j)}_r$ labeling the $r$th sampled index of species $j$. The corresponding intensity measure of $dN_{\ell}$ is given by $d\bar{N}_\ell(s, \vec{i}, \theta)= ds \, \left( \bigwedge_{ j = 1}^J \left( \bigwedge_{r = 1}^{\alpha_{\ell j}}\left(\sum_{k\geq 0} \delta_k( i_{r}^{(j)})\right) \right)\right)\, d\theta$. Analogously, for each reaction $\mathscr{R}_\ell$ with products,  i.e. $\tilde{L} + 1 \leq \ell \leq L$,  we associate with it a Poisson point measure $dN_\ell(s, \vec{i}, \vec{y}, \theta_1, \theta_2)$ on $\R_+\times \mathbb{I}^{(\ell)} \times \mathbb{Y}^{(\ell)} \times\R_+\times\R_+$. Here $\vec{i}\in \mathbb{I}^{(\ell)}$ gives the sampled reactant configuration, with $i^{(j)}_r$ labeling the $r$th sampled index of species $j$. $\vec{y}\in \mathbb{Y}^{(\ell)}$ gives the sampled product configuration, with $y_r^{(j)}$ labeling the sampled position for the $r$th newly created particle of species $j$. The corresponding intensity measure is given by $d\bar{N}_\ell(s, \vec{i}, \vec{y}, \theta_1, \theta_2)= ds \, \left( \bigwedge_{ j = 1}^J \left( \bigwedge_{r = 1}^{\alpha_{\ell j}}\left(\sum_{k\geq 0} \delta_k( i_{r}^{(j)})\right) \right)\right)\,d\vec{y}\, d\theta_1 \, d\theta_2$.

The existence of the Poisson point measure follows as the intensity measure is $\sigma$-finite (see Chapter I - Theorem 8.1 in \cite{NW:2014} or Corollary 9.7 in \cite{K:2001}). Let $d\tilde{N}_{\ell}(s, \vec{i}, \vec{y}, \theta_1, \theta_2) = dN_{\ell}(s, \vec{i}, \vec{y}, \theta_1, \theta_2) - d\bar{N}_{\ell}(s, \vec{i}, \vec{y}, \theta_1, \theta_2)$ be the compensated Poisson measure, for $\tilde{L}+1 \leq \ell \leq L$. For any measurable set $A\in  \mathbb{I}^{(\ell)} \times \mathbb{Y}^{(\ell)} \times\R_+\times\R_+$, $N_\ell(\, \cdot \, , A)$ is a Poisson process and $\tilde{N}_\ell(\, \cdot \,, A)$ is a martingale (see Proposition 9.18 in \cite{K:2001}). Similarly, we can define $d\tilde{N}_{\ell}(s, \vec{i}, \theta) = dN_{\ell}(s, \vec{i}, \theta) - d\bar{N}_{\ell}(s, \vec{i}, \theta)$, for $1\leq \ell \leq \tilde{L}$. In this case, given any measurable set $A\in  \mathbb{I}^{(\ell)} \times\R_+$, we then have that $N_\ell(\, \cdot \, , A)$ is a Poisson process and $\tilde{N}_\ell(\, \cdot \,, A)$ is a martingale.

With the preceding definitions, we now define the dynamics of $\nu_t^{\vec{\zeta}}$ via a weak representation. We consider test functions denoted by $f\in C_b^2(\hat P)$, which we define to mean $f(\cdot,S_j)\in C_{b}^{2}(\mathbb{R}^{d})$ for each $j$. The time evolution for the process $\la f,\nu^{{\vec{\zeta}}}_{t}\ra_{\hat{P}}$ can then be represented by
\begin{align}
\la f,\nu^{{\vec{\zeta}}}_{t}\ra_{\hat{P}}&=\la f,\nu^{{\vec{\zeta}}}_{0}\ra_{\hat{P}}+\sum_{i\geq 1}\int_{0}^{t}1_{\{i\leq \la 1, \nu_{s-}^{\vec{\zeta}}\ra_{\hat{P}}\}}\sqrt{2D^{i}}\frac{\partial f}{\partial Q}(H^i(\nu_{s-}^{\vec{\zeta}}))dW^{i}_{s}+\int_{0}^{t}\sum_{i=1}^{\la 1, \nu_{s-}^{\vec{\zeta}}\ra_{\hat{P}}} D^{i}\frac{\partial^{2} f}{\partial Q^{2}}(H^i(\nu_{s-}^{\vec{\zeta}}))ds\notag\\
&
\quad+\sum_{\ell = 1}^{\tilde{L}} \int_{0}^{t}\int_{\mathbb{I}^{(\ell)}} \int_{\mathbb{R}_{+}}\left(\la f, \nu^{{\vec{\zeta}}}_{s-}- \sum_{j = 1}^J\sum_{r = 1}^{\alpha_{\ell j}} \delta_{(H_Q^{i_r^{(j)}}(\nu^{{\vec{\zeta}}, j}_{s-}),  S_j)} \ra_{\hat{P}} - \la f,\nu^{{\vec{\zeta}}}_{s-}\ra_{\hat{P}}\right) \notag\\
&\qquad
\times 1_{\{\vec{i} \in \Omega^{(\ell)}(\nu^{\vec{\zeta}}_{s-})\}} \times 1_{\{ \theta \leq K_\ell^{\gamma}\left(\mathcal{P}^{(\ell)}(\nu_{s-}^{\vec{\zeta}}, \vec{i}) \right)\}}  dN_{\ell}(s,\vec{i}, \theta)\notag\\
&
\quad+\sum_{\ell = \tilde{L} + 1}^L \int_{0}^{t}\int_{\mathbb{I}^{(\ell)}} \int_{\mathbb{Y}^{(\ell)}}\int_{\mathbb{R}_{+}^2}\left(\la f, \nu^{{\vec{\zeta}}}_{s-}- \sum_{j = 1}^J\sum_{r = 1}^{\alpha_{\ell j}} \delta_{(H_Q^{i_r^{(j)}}(\nu^{{\vec{\zeta}}, j}_{s-}),  S_j)} + \sum_{j = 1}^J \sum_{r = 1}^{\beta_{\ell j }} \delta_{(y_r^{(j)}, S_j)} \ra_{\hat{P}} -\la f,\nu^{{\vec{\zeta}}}_{s-}\ra_{\hat{P}}\right) \label{Eq:EM_formula}\\
&\qquad
\times 1_{\{\vec{i} \in \Omega^{(\ell)}(\nu^{\vec{\zeta}}_{s-})\}} \times 1_{\{ \theta_1 \leq K_\ell^{\gamma}\left(\mathcal{P}^{(\ell)}(\nu_{s-}^{\vec{\zeta}}, \vec{i}) \right)\}}\times 1_{\{\theta_2\leq m^{\eta}_\ell\left(\vec{y} \,  | \, \mathcal{P}^{(\ell)}(\nu_{s-}^{\vec{\zeta}}, \vec{i}) \right)\}}  dN_{\ell}(s,\vec{i}, \vec{y},\theta_1, \theta_2). \notag
\end{align}

Formula~\eqref{Eq:EM_formula} captures the dynamics of our particle system.Recall that $N_j(s) =\la 1, \nu_{s}^{\vec{\zeta}}\ra_{\hat{P}}$ denotes the total number of molecules at time $s$, and $D^i$ labels the diffusion coefficient for the $i$th molecule, taking values in $\{D_1, \dots, D_J\}$, where $D_j$ is the diffusion coefficient for species $S_j$, $ j = 1,\cdots, J$. The diffusion of each particle is modeled by the two integrals on the first line of~\eqref{Eq:EM_formula}.  The second and third lines model reactions with no products, while the fourth and fifth lines model reactions with products. The integrals involving the Poisson measures $N_{\ell}$ model the different components of the reaction processes, and correspond to sampling the times at which reactions occur, which reactant particles react, and where reaction products are placed. When the $\ell$th reaction happens for $\ell=\tilde{L}+1,\cdots, L$ (and analogously for $\ell=1,\cdots,\tilde{L}$), with probability per time given by the kernel $K_{\ell}^{\gamma}$, the system loses reactant particles and gains product particles. Sampling of possible reaction occurrences according to $K_{\ell}^{\gamma}$ occurs through the corresponding indicator functions on the third and fifth lines. The corresponding loss and gain of particles is encoded by the sums of delta functions on the second and fourth lines of~\eqref{Eq:EM_formula}. Product positions are sampled according to the placement density $m^{\eta}_\ell\left(\vec{y} \,  | \, \mathcal{P}^{(\ell)}(\nu_{s-}^{\vec{\zeta}}, \vec{i}) \right)$ through the indicator function on the fifth line. The indicators over elements of the sets $\Omega^{(\ell)}(\nu_{s-}^{\vec{\zeta}})$ ensure that reactions can only occur between particles that correspond to a possible set of reactants. Note, the particle labeled by $i$ in~\eqref{Eq:EM_formula} will change dynamically as reactions occur. For this reason, particle positions are accessed through the use of the state vectors, $H^i$ and $H^i_Q$, as is also done in structured population models~\cite{BM:2015}. Well-posedness properties of the model equation (\ref{Eq:EM_formula}) are further discussed in Section \ref{S:MainResult}, see also Chapter 6 of \cite{BM:2015} for related results in regards to the formulation and well-posedness.

We will subsequently assume that $N_j(t) =\la 1, \nu_s^{\vec{\zeta}}\ra$ is uniformly bounded in time in Assumption \ref{Assume:kernalBdd}. The stochastic integral with respect to Brownian motion in~\eqref{Eq:EM_formula} is then a martingale (for a fixed ${\vec{\zeta}}$). Taking the expectation, we obtain for the mean that
\begin{align}
\avg{\la f,\nu^{{\vec{\zeta}}}_{t}\ra_{\hat{P}}}&=\avg{\la f,\nu^{{\vec{\zeta}}}_{0}\ra_{\hat{P}}}+ \avg{\int_{0}^{t}\sum_{i=1}^{\la 1, \nu_{s-}^{\vec{\zeta}}\ra_{\hat{P}}} D^{i}\frac{\partial^{2} f}{\partial Q^{2}}(H^i(\nu_{s-}^{\vec{\zeta}}))ds}\nonumber\\
&
\quad+\sum_{\ell = 1}^{\tilde{L}} \avg{\int_{0}^{t}\sum_{\vec{i}\in\Omega^{(\ell)}(\nu_{s-}^{\vec{\zeta}})} \left(\la f, \nu^{{\vec{\zeta}}}_{s-}- \sum_{j = 1}^J\sum_{r = 1}^{\alpha_{\ell j}} \delta_{(H_Q^{i_r^{(j)}}(\nu^{{\vec{\zeta}}, j}_{s-}),  S_j)} \ra-\la f,\nu^{{\vec{\zeta}}}_{s-}\ra_{\hat{P}}\right) \times K_\ell^{\gamma}\left(\mathcal{P}^{(\ell)}(\nu_{s-}^{\vec{\zeta}}, \vec{i}) \right)  \, ds}\nonumber\\
&
\quad+\sum_{\ell = \tilde{L}+1}^L \avg{\int_{0}^{t}\sum_{\vec{i}\in\Omega^{(\ell)}(\nu_{s-}^{\vec{\zeta}})} \int_{\mathbb{Y}^{\ell}} \left(\la f, \nu^{{\vec{\zeta}}}_{s-}- \sum_{j = 1}^J\sum_{r = 1}^{\alpha_{\ell j}} \delta_{(H_Q^{i_r^{(j)}}(\nu^{{\vec{\zeta}}, j}_{s-}),  S_j)} + \sum_{j = 1}^J \sum_{r = 1}^{\beta_{\ell j}} \delta_{(y_r^{(j)}, S_j)} \ra_{\hat{P}} \right. \nonumber\\
&\qquad
-\left. \la f,\nu^{{\vec{\zeta}}}_{s-}\ra_{\hat{P}}\right) \times K_\ell^{\gamma}\left(\mathcal{P}^{(\ell)}(\nu_{s-}^{\vec{\zeta}}, \vec{i}) \right)  \times m_\ell^{\eta}\left(\vec{y} \, | \, \mathcal{P}^{(\ell)}(\nu_{s-}^{\vec{\zeta}}, \vec{i}) \right)  \,  d\vec{y} \, ds}\label{Eq:MEAN_EM_formula}.
\end{align}

\section{Assumptions on reaction functions and placement densities} \label{S:assumptions}
In studying the large population limit that $\gamma \to \infty$, we will constrain our choices of reaction kernels and placement densities through the following assumptions. Special cases of our choices include a variety of kernels and placement densities that are commonly used in modeling and simulation~\cite{DoiSecondQuantB,ErbanChapman2009,LipkovaErban2011,Isaacson2013,IZ:2018,DonevJCP2018}.
\begin{assumption}\label{Assume:kernalBdd0}
We assume that for all $1\leq \ell \leq L$, the reaction rate kernel $K_\ell(\vec{x})$ is uniformly bounded for all $\vec{x}\in \mathbb{X}^{(\ell)}$. We denote generic constants that depend on this bound by $C(K)$.
\end{assumption}

\begin{assumption}\label{Assume:measureP}
We assume that for any $\eta \geq 0$, $\tilde{L} + 1\leq \ell \leq L$, $\vec{y}\in \mathbb{Y}^{(\ell)}$ and $\vec{x}\in \mathbb{X}^{(\ell)}$,  the placement density $m^\eta_{\ell}(\vec{y} \, | \, \vec{x})$ is a bounded probability density, i.e. $\int_{\mathbb{Y}^{(\ell)}} m^\eta_{\ell}(\vec{y} \, | \, \vec{x})\, d\vec{y}= 1$.
\end{assumption}

As previously mentioned, we want to allow for placement densities involving delta-functions. To do so in a mathematically rigorous way we introduced the smoothing parameter $\eta$, through which we can define a corresponding mollifier in a standard way, as given by Definition~\eqref{def:molifier}. This is needed for~\eqref{Eq:EM_formula} to be well-defined, since expressions like $\{\theta_2\leq m^{\eta}_\ell\left(\vec{y} \,  | \, \mathcal{P}^{(\ell)}(\nu_{s-}^{\vec{\zeta}}, \vec{i}) \right)\}$ are non-sensical when $\eta=0$ and the placement density is a Dirac delta function.
\begin{definition}\label{def:molifier}
For $x\in\R^d$, let $G(x)$ denote a standard positive mollifier and $G_\eta(x) = \eta^{-d}G(x/\eta)$. That is, $G(x)$ is a smooth function on $\R^d$ satisfying the following four requirements
\begin{enumerate}
\item $G(x)\geq 0$,
\item $G(x)$ is compactly supported in $B(0, 1)$, the unit ball in $\R^d$,
\item $\int_{\R^d} G(x)\, dx = 1$,
\item $\lim_{\eta\to 0} G_\eta(x) = \lim_{\eta\to 0} \eta^{-d}G(x/\eta) = \delta_0(x) $, where $\delta_0(x)$ is the Dirac delta function and the limit is taken in the space of Schwartz distributions.
\end{enumerate}
\end{definition}

The allowable forms of the placement density for each possible reaction are given by Assumptions \ref{Assume:measureOne2One}-\ref{Assume:measureOne2Two}:



\begin{assumption}\label{Assume:measureOne2One}
If $\mathscr{R}_{\ell}$ is a first order reaction of the form $S_i \rightarrow S_j$, we assume that the placement density  $m^\eta_{\ell}(y\,|\, x)$ takes the mollified form of
$$m^\eta_{\ell}(y\, |\, x) =G_\eta(y-x),$$
with the distributional limit as $\eta \to 0$ given by
$$m_{\ell}(y\,|\, x) = \delta_x(y).$$
This describes that the newly created $S_j$ particle is placed  at the position of the reactant $S_i$ particle.
\end{assumption}

\begin{assumption}\label{Assume:measureTwo2One}
If $\mathscr{R}_{\ell}$ is a second order reaction of the form $S_i + S_k \rightarrow S_j$, we assume that the binding placement density  $m_{\ell}(z\, |\, x,y)$ takes the mollified form of
$$m^\eta_{\ell}(z\,|\,x,y) =\sum_{i=1}^{I}p_i \times G_\eta\left(z-(\alpha_i x +(1-\alpha_i)y)\right),$$
 with the distributional limit as $\eta \to 0$ given by
 $$m_{\ell}(z\,|\,x,y) = \sum_{i=1}^{I}p_i \times \delta\left(z-(\alpha_i x +(1-\alpha_i)y)\right),$$
  where $I$ is a fixed finite integer and $\sum_i p_i = 1$. This describes that the creation of particle $S_j$ is always on the segment connecting the reactant $S_i$ and reactant $S_k$ particles, but allows some random choice of position. A special case would be $I = 2$, $p_i = \tfrac{1}{2}$, $\alpha_1 = 0$ and $\alpha_2 = 1$, which corresponds to placing the particle randomly at the position of one of the two reactants. One common choice is taking $I = 1$, $p_1 = 1$ and choosing $\alpha_1$ to be the diffusion weighted center of mass~\cite{IZ:2018}.
\end{assumption}

\begin{assumption}\label{Assume:measureTwo2Two}
If $\mathscr{R}_{\ell}$ is a second order reaction of the form $S_i + S_k \rightarrow S_j + S_r$, we assume that the placement density  $m_\ell(z, w \, | \, x, y)$ takes the mollified form of
$$m^\eta_\ell(z, w \, | \, x, y) = p\times G_\eta\left(x-z\right)\otimes G_\eta\left(y-w\right)   + (1-p)\times G_\eta\left(x-w\right)\otimes G_\eta\left(y-z\right),$$
with the distributional limit as $\eta \to 0$ given by
$$m_\ell(z, w \, | \, x, y) = p\times\delta_{(x, y)}\left((z, w)\right)  + (1-p)\times\delta_{(x, y)}\left((w, z)\right).$$
This describes that newly created product $S_j$ and $S_r$ particles are always at the positions of the reactant $S_i$ and $S_k$ particles. $p$ is typically either $0$ or $1$, depending on the underlying physics of the reaction.
\end{assumption}

\begin{assumption}\label{Assume:measureOne2Two}
If $\mathscr{R}_{\ell}$ is a first order reaction of the form $S_i \rightarrow S_j + S_k$, we assume the unbinding displacement density is in the mollified form of $$m^\eta_{\ell}(x,y\,|\,z) = \rho(|x-y|) \sum_{i=1}^{I}p_i\times G_\eta\left(z-(\alpha_i x +(1-\alpha_i)y)\right),$$
with the distributional limit as $\eta \to 0$ given by
 $$m_{\ell}(x,y\,|\,z) = \rho(|x-y|) \sum_{i=1}^{I}p_i\times\delta\left(z-(\alpha_i x +(1-\alpha_i)y)\right),$$
with $\sum_i p_i = 1$. Here we assume the relative separation of the product $S_j$ and $S_k$ particles, $\abs{x -y}$, is sampled from the probability density $\rho(|x-y|)$. Their (weighted) center of mass is sampled from the density encoded by the sum of $\delta$ functions. Such forms are common for detailed balance preserving reversible bimolecular reactions~\cite{IZ:2018}.
\end{assumption}
We further assume some regularity of the separation placement density, $\rho(r)$, introduced in Assumption~\ref{Assume:measureOne2Two}:
\begin{assumption}\label{Assume:measureMrho}
For Assumption~\ref{Assume:measureP} to be true, we'll need
$$\int_{\R^d} \rho(|w|)\, dw = 1.$$
Since $\rho$ is a probability density and non-negative, this implies the tail estimate
 $$\int_{r>R}r^{d-1}\rho(r) \, dr \leq \epsilon,$$
which we will use in subsequent calculations.
\end{assumption}

Finally, to study the large-population limit of the population density measures, we must specify how the reaction kernels depend on the scaling parameter (i.e. system size parameter) $\gamma$. Motivated by the classical spatially homogeneous reaction network large-population limit~\cite{DK:2015}, we choose
\begin{assumption}\label{Assume:rescaling}
The reaction kernel is assumed to have the explicit $\gamma$ dependence that
\begin{equation*}
K_\ell^{\gamma}(\vec{x}) = \gamma^{1 - |\vec{\alpha}^{(\ell)}|}K_\ell(\vec{x})
\end{equation*}
for any $\vec{x}\in \mathbb{X}^{(\ell)}$, $1\leq \ell \leq L$.
\end{assumption}
When interpreting the scaling parameter $\gamma$ as Avogadro's number, or in bounded domains as the product of Avogadro's number and the domain volume~\cite{DK:2015}, such scalings can be derived by requiring the formal well-mixed (i.e. infinitely fast diffusion) limit of the volume reactivity PBSRD model to match the corresponding classical spatially homogeneous stochastic chemical kinetics model. See Appendix~\ref{A:gammascaling} for an illustrative example of how the chosen scalings arise in this case.

Recall that $|\vec{\alpha}^{(\ell)}|$ represents the number of reactant particles needed for the $\ell$-th reaction.
As we assume $\abs{\vec{\alpha}}^{\ell} \leq 2$, we obtain three scalings for the three allowable reaction orders:
\begin{itemize}
\item $|\vec{\alpha}^{(\ell)}| = 0$ corresponds to a pure birth reaction. By Assumption \ref{Assume:rescaling}, the scaling is $\gamma$; i.e. a larger system size implies more births. In a well-mixed model this would imply that as  $\gamma$ and the initial number of molecules are increased, we maintain a fixed rate \emph{with units of molar concentration per time} for the birth reaction to occur.
\item $|\vec{\alpha}^{(\ell)}| = 1$ corresponds to a unimolecular reaction. By Assumption \ref{Assume:rescaling}, there's no rescaling as it's linear. We assume the rates of first order reactions are internal processes to particles, and as such independent of the system size.
\item $|\vec{\alpha}^{(\ell)}| = 2$ corresponds to a bimolecular reaction. By Assumption \ref{Assume:rescaling}, the scaling of reaction kernel is $\gamma^{-1}$. As the system size increases it is harder for two \emph{individual} reactant particles to encounter each other.
\end{itemize}


\section{Main result and examples}\label{S:MainResult}

We now formulate a weak representation for the time evolution of scaled empirical measures $\mu^{{\vec{\zeta}}, j}_{t}=\frac{1}{\gamma}\nu^{{\vec{\zeta}}, j}_{t}$ with $j =1, \cdots, J$ and $\mu^{{\vec{\zeta}}}_{t}=\frac{1}{\gamma}\nu^{{\vec{\zeta}}}_{t}= \sum_{j = 1}^J\mu^{{\vec{\zeta}}, j}_{t}\delta_{S_j} $. $\mu_t^{{\vec{\zeta}}, j}$ physically corresponds to the molar concentration field for species $j$ at time $t$.

For a test function $f\in C^2_b(\R^d)$ and for each species $j = 1, \cdots, J$, let us define the generator
\begin{align*}
(\mathcal{L}_{j}f)(x)= D_j \Delta_x f(x).
\end{align*}
We'll focus on proving the convergence as $\vec{\zeta} \to 0$ of the marginal distribution vector
\begin{align*}
\left(\mu^{{\vec{\zeta}},1}_{t},\mu^{{\vec{\zeta}},2}_{t}, \cdots,  \mu^{{\vec{\zeta}},J}_{t}\right).
\end{align*}
We make two final assumptions before stating our main result. First, to simplify the analysis we assume
the total molar concentration is bounded as $\vec{\zeta} \to 0$:
\begin{assumption}\label{Assume:kernalBdd}
We assume that the total (molar) population concentration satisfies $\sum_{j = 1}^J \la 1, \mu_t^{{\vec{\zeta}}, j} \ra \leq C(\mu)$ for all $t<\infty$, i.e. is uniformly in time bounded by some constant $C(\mu)$. In the remainder we abuse notation and also denote generic constants that depend on this bound by $C(\mu)$.
\end{assumption}

\begin{remark}\label{R:StoppingTime}
If we define the stopping time
\begin{equation} \label{eq:stoppingtime}
\tau^{\vec{\zeta}}=\inf\left\{t\geq 0, \sum_{j = 1}^J \la 1, \mu_t^{{\vec{\zeta}}, j} \ra > C(\mu)\right\},
\end{equation}
Assumption~\ref{Assume:kernalBdd} essentially requires that $\mathbb{P}\left(\tau^{\vec{\zeta}}=\infty\right)=1$ for all $\vec{\zeta}$. We have chosen to use the condition of Assumption \ref{Assume:kernalBdd} instead of introducing the stopping time $\tau^{\vec{\zeta}}$ in order to simplify some of the arguments, notation, and presentation. Because of Assumption \ref{Assume:kernalBdd}, our main result, Theorem~\ref{thm:convergence}, does not apply to reaction networks that include zeroth order birth reactions (i.e. reactions of the form $\varnothing \to S_i$). Similarly, reactions of the form $S_i \to S_i + S_k$ would be excluded since they also allow the possibility of unbounded population growth. In order to include such reactions, one would need to introduce a stopping time like (\ref{eq:stoppingtime}) for when the total molar population concentration first exceeds $C(\mu)$, and study its limiting behavior as $\vec{\eta} \to 0$. Though we do not show it here, we conjecture that in these cases the large-population limit of Theorem~\ref{thm:convergence} will hold till time $t\wedge T_{0}$, where $T_{0}$ is any finite time over which the solution to the limiting mean-field equations is well defined (see \cite{IMS21a}). For large-population limits of \emph{non-spatial} stochastic chemical kinetic systems a stopping time-based approach is carried out in~\cite{DK:2015}, while for related structured population models a stopping time-based approach is used in \cite{BM:2015}.

Note, assuming a fixed, finite number of each particle at $t=0$, Assumption~\ref{Assume:kernalBdd} would, for example, always hold in systems with fully reversible reactions that do not create particles from nothing. These include reactions of the form $S_{i}+ S_{j} \leftrightarrow S_{k}+ S_{l}$, $S_{i}+S_{j}\leftrightarrow S_{i}+S_{k}$, $S_{i}+S_{j}\leftrightarrow S_{k}$, $2 S_{i} \leftrightarrow S_{j}$, $S_{i} \leftrightarrow S_{j}$, etc. Reversible reactions like $\varnothing \leftrightarrow S_{i}$, $S_{i} \leftrightarrow S_{i} + S_{j}$, or $S_{i} \leftrightarrow 2S_{i}$ would again be excluded since they involve the creation of new particles from nothing.
\end{remark}

\begin{remark}\label{R:WellPosednessPrelimit}
Let us now discuss the well-definiteness of the process $\{\nu^{\vec{\zeta}}_{t}\}_{t\geq 0}$ (equivalently $\{\mu^{\vec{\zeta}}_{t}\}_{t\geq 0}$) as defined via (\ref{Eq:EM_formula}). For general reaction-networks, one cannot expect (\ref{Eq:EM_formula}) to be well-posed for all times as, for instance, one could have finite time blow-up (consider, for example, the standard ODE model for the reaction $2 S_{i} \to 3 S_{i}$). On the other hand, for most biologically-motivated systems one does not expect almost sure blow-up in finite time. In fact, for systems with the pure conversion reactions
$S_{i}+ S_{j} \leftrightarrow S_{k}+ S_{l}$, one can check that our formulation is analogous to the formulation of \cite{Nolen2019}, where well-posedness of the pre-limit Markov system is indeed established. Instead of trying to prove for which combinations of reaction kernels and networks one has well-posedness, an open problem even for deterministic reaction-diffusion PDE models, we have made Assumption~\ref{Assume:kernalBdd}.

While we do not prove well-posedness of $\{\nu^{\vec{\zeta}}_{t}\}_{t\geq 0}$ here, a basic approach one could take to try to establish it is as follows. We first note that between the times that two consecutive reactions take place, the number of particles in the system is fixed, and each particle moves independently by Brownian motion. When a reaction occurs the number of particles changes, with the positions and types of reactants and substrates updated based on the sampling, reaction and placements rules of (\ref{Eq:EM_formula}) (see also the description after equation (\ref{Eq:EM_formula})). Assumption~\ref{Assume:kernalBdd} guarantees that the total population  concentration stays bounded. The Markov property holds because the sampling and placement rules at the next reaction time, say $\tau$, depend only on the state of the system at time $\tau-$. Hence, Assumption~\ref{Assume:kernalBdd}, together with the boundedness and regularity Assumptions \ref{Assume:kernalBdd0} and \ref{Assume:measureP}, is expected to lead to a well-defined process $\{\nu^{\vec{\zeta}}_{t}\}_{t\geq 0}$ (equivalently $\{\mu^{\vec{\zeta}}_{t}\}_{t\geq 0}$).

As also indicated in Chapter 6 in \cite{BM:2015}, instead of Assumption~\ref{Assume:kernalBdd} it should be sufficient to know that for every $T<\infty$ we have the integrability condition   $\mathbb{E}\left[\sup_{t\in[0,T]}\sum_{j = 1}^J \la 1, \nu_t^{{\vec{\zeta}}, j} \ra^{p}\right] < \infty$ for an appropriate $p\geq 1$ (together with appropriate boundedness and regularity of reaction kernels and placement densities). A potential method for establishing this would be to build the process step by step. An outline for this process is indicated in the related results of Chapter 6 in \cite{BM:2015} (see Theorem 6.4 there), where one builds the solution up to the time that the total population concentration reaches a certain threshold,  and then proves that the sequence of jump times goes almost surely to infinity as the aforementioned threshold tends to infinity. We have chosen to make the stronger Assumption~\ref{Assume:kernalBdd} in part to simplify some of the arguments for the a-priori bounds that are needed in order to prove our main convergence result, Theorem \ref{thm:convergence}.

We reiterate though; it is an open question to characterize all the possible general spatial reaction-networks for which (\ref{Eq:EM_formula}) is well-posed.
In the remainder we assume that we have a reaction-network for which this is the case, and, with that assumed, our goal is to establish the limit of $\{\mu^{\vec{\zeta}}_{t}\}_{t\geq 0}$ as  $\vec{\zeta}\to 0$.
\end{remark}

Finally, we assume convergence of the initial molar concentrations of each type at $t=0$ as $\gamma \to \infty$:
\begin{assumption}\label{Assume:initial}
We assume that the initial distribution $\mu_0^{{\vec{\zeta}}, j}\to \xi_0^j$ weakly as $\vec{\zeta}\to 0$, where $\xi_0^j$ is a compactly supported measure, for all $1\leq j\leq J$.
\end{assumption}

We are now ready to state our main result.  Let $M_F(\mathbb{R}^d)$ be the space of finite measures endowed with the weak topology and $\mathbb{D}_{M_F(\mathbb{R}^d)}[0, T]$ be the space of cadlag paths with values in $M_F(\mathbb{R}^d)$ endowed with Skorokhod topology.

\begin{theorem}{(Main result)} \label{thm:convergence}
Recall that $\vec{\zeta}=(1/\gamma,\eta)$ and assume that in the prelimit $\gamma<\infty$ and $\eta>0$. Let $T<T_{0}\leq \infty$ be given with $T_{0}$ to be specified later on.  Assume Assumptions~\ref{Assume:kernalBdd0}-\ref{Assume:measureMrho} for the reaction kernels and placement densities,  scaling Assumption~\ref{Assume:rescaling}, Assumption~\ref{Assume:kernalBdd}  (hence the initial total population concentration is assumed to be bounded), Assumption~\ref{Assume:initial}, and that the reaction-network is such that $\{(\mu^{{\vec{\zeta}}, 1}_t,\cdots,\mu^{{\vec{\zeta}}, J}_t)\}_{t\in [0, T]}\in \mathbb{D}_{\otimes_{j=1}^{J}M_F(\mathbb{R}^d)}([0, T])$ is well-defined (see Remark \ref{R:WellPosednessPrelimit}). Then, the sequence of measure-valued processes $\{(\mu^{{\vec{\zeta}}, 1}_t,\cdots,\mu^{{\vec{\zeta}}, J}_t)\}_{t\in [0, T]}\in \mathbb{D}_{\otimes_{j=1}^{J}M_F(\mathbb{R}^d)}([0, T])$ is relatively compact in $\mathbb{D}_{\otimes_{j=1}^{J}M_F(\mathbb{R}^d)}([0, T])$ for each $j= 1, 2,\cdots, J$. It converges in distribution to $\{(\xi^1_t,\cdots,\xi^{J}_{t})\}_{t\in [0, T]} \in C_{\otimes_{j=1}^{J} M_F(\mathbb{R}^d)}([0, T])$ as $\vec{\zeta}\to 0$, where the limit is taken such that $\eta > 0$ for each finite $\gamma$. Each $\xi^{j}_{t}$ is respectively the unique solution to
\begin{align}
\la f,\xi^j_{t}\ra
&
=\la f,\xi^{ j}_{0}\ra + \int_{0}^{t} \la (\mathcal{L}_j f)(x), \xi_{s}^{ j}(dx) \ra ds\nonumber\\
&
\quad-\sum_{\ell = 1}^{\tilde{L}} \int_{0}^{t} \int_{\tilde{\mathbb{X}}^{(\ell)}}     \frac{1}{\vec{\alpha}^{(\ell)}!}   K_\ell\left(\vec{x}\right)  \left(  \sum_{r = 1}^{\alpha_{\ell j}} f(x_r^{(j)}) \right)\,\lambda^{(\ell)}[\xi_{s}](d\vec{x}) \, ds\nonumber\\
&
\quad+\sum_{\ell = \tilde{L} + 1}^L \int_{0}^{t} \int_{\tilde{\mathbb{X}}^{(\ell)}}     \frac{1}{\vec{\alpha}^{(\ell)}!}   K_\ell\left(\vec{x}\right) \left(\int_{\mathbb{Y}^{(\ell)}}    \left( \sum_{r = 1}^{\beta_{\ell j}} f(y_r^{(j)}) \right) m_\ell\left(\vec{y} \, |\, \vec{x} \right)\, d \vec{y} - \sum_{r = 1}^{\alpha_{\ell j}} f(x_r^{(j)}) \right)\,\lambda^{(\ell)}[\xi_{s}](d\vec{x}) \, ds.\label{Eq:Limit_EM_formula2}
\end{align}
Here $T_{0}$ is the maximal time up to which the deterministic system (\ref{Eq:Limit_EM_formula2}) has a well defined solution.
\end{theorem}

\begin{remark}\label{R:WellPosednessLimitSystem}
In the companion paper \cite{IMS21a} we study well-posedness and regularity of the limiting system (\ref{Eq:Limit_EM_formula2}) and we present numerical studies related to the behavior of the pre-limit and limit systems. It is proven in \cite{IMS21a} that (\ref{Eq:Limit_EM_formula2}) always has an appropriate local in-time solution (i.e. there exists some $0<T_{0}< \infty$) and global in-time solution (i.e. $T_{0}=\infty$) is established for specific reaction networks. We refer the interested reader to
\cite{IMS21a} for details.
\end{remark}

\begin{remark}\label{R:ConvergenceInProbability}
Given that weak convergence to a constant implies convergence in probability, we get that Theorem \ref{thm:convergence} actually implies  convergence in probability. Namely, for any $\delta>0$:
\begin{equation*}
\lim_{\vec{\zeta}\rightarrow 0}\mathbb{P}\left[d_{\otimes_{j=1}^{J}M_F(\mathbb{R}^d)}\left((\mu^{{\vec{\zeta}}, 1},\cdots,\mu^{{\vec{\zeta}}, J}),(\xi^1,\cdots,\xi^{J})\right)\geq \delta\right]=0,
\end{equation*}
where $d_{\otimes_{j=1}^{J} M_F(\mathbb{R}^d)}$ is the metric for $\mathbb{D}_{\otimes_{j=1}^{J} M_F(\mathbb{R}^d)}[0, T]$, see for example Section 3.2 of \cite{Capponi} for an exposition in an analogous situation.
\end{remark}
\begin{proof}[Proof of Theorem \ref{thm:convergence}]
Let $\pi^{{\vec{\zeta}}}$ be the $\mathbb{P}-$law of $(\mu^{{\vec{\zeta}}, 1},\cdots,\mu^{{\vec{\zeta}}, J})$, i.e.
\[
\pi^{{\vec{\zeta}}}(A)=\mathbb{P}\left[(\mu^{{\vec{\zeta}}, 1},\cdots,\mu^{{\vec{\zeta}}, J}) \in A\right],
\]
for all $A\in\mathcal{B}(\mathbb{D}_{\otimes_{j=1}^{J}M_F(\mathbb{R}^d)}([0, T]))$. This means that for all $\zeta\in(0,1)^{2}$  we have that $\pi^{{\vec{\zeta}}}\in M(\mathbb{D}_{\otimes_{j=1}^{J}M_F(\mathbb{R}^d)}([0, T]))$.

By the relative compactness of Theorem  \ref{thm:tightness} we get that every subsequence $\pi^{{\vec{\zeta}_{k}}}$ has a further sub-sequence $\pi^{{\vec{\zeta}_{k_{m}}}}$ which converges weakly.  Lemma \ref{L:wconv} says that any limit point $\pi$ of $\pi^{{\vec{\zeta}_{k_{m}}}}$ is such that $\pi=\delta_{(\xi^{1},\cdots,\xi^{J})}$ where $\xi^{j}$ satisfies the evolution equation (\ref{Eq:Limit_EM_formula2}). Lemma \ref{lem:uniqueness} proves uniqueness of solutions to Eq (\ref{Eq:Limit_EM_formula2}). Therefore, by Prokhorov's theorem $\pi^{{\vec{\zeta}}}$ converges weakly to $\pi$, where $\pi$ is the distribution of $(\xi^{1},\cdots,\xi^{J})$, the unique solution to (\ref{Eq:Limit_EM_formula2}). That is to say that $(\mu^{{\vec{\zeta}}, 1},\cdots,\mu^{{\vec{\zeta}}, J})$ converges in distribution to $(\xi^1,\cdots,\xi^{J})$. Lemma \ref{lem:Continuity} proves that for each $j=1,\cdots,J$, $\{\xi^j_t\}_{t\in [0, T]} \in C_{M_F(\mathbb{R}^d)}([0, T])$. This concludes the proof of the theorem.

%
\end{proof}

\begin{remark}\label{Rm:density}
If the limiting measures $(\xi_t^1(dx), \cdots, \xi_t^J(dx))$ have marginal densities (i.e. molar concentrations) $(\rho_1(x, t),  \cdots,  \rho_J(x, t))$,  then the marginals are expected to solve the following reaction-diffusion partial integro-differential equations (PIDEs) in a weak sense.
\begin{align}
\partial_t\rho_j(x, t)
&
= D_{j} \lap_x \rho_j(x, t) - \sum_{\ell = 1}^L \paren{
\frac{1}{\vec{\alpha}^{(\ell)}!}  \sum_{r = 1}^{\alpha_{\ell j}} \int_{\tilde{\vec{x}} \in \mathbb{X}^{(\ell)}}   \delta_{x}(\tilde{x}_r^{(j)})  K_\ell(\tilde{\vec{x}}) \, \left( \Pi_{k = 1}^{J} \Pi_{s = 1}^{\alpha_{\ell k}}  \rho_{k}(\tilde{x}_{s}^{(k)}, t)\right) \, d\tilde{\vec{x}}
}  \nonumber\\
&
\quad+\sum_{\ell = \tilde{L} + 1}^L \paren{  \frac{1}{\vec{\alpha}^{(\ell)}!} \sum_{r = 1}^{\beta_{\ell j}}  \int_{\tilde{\vec{x}} \in\mathbb{X}^{(\ell)}}  K_\ell(\tilde{\vec{x}}) \left( \int_{\vy \in \mathbb{Y}^{(\ell)}}   \delta_{x}(y_r^{(j)}) m_\ell(\vec{y}\, | \,\tilde{\vec{x}}) \,d \vec{y} \right) \left( \Pi_{k = 1}^J \Pi_{s = 1}^{\alpha_{\ell k}}  \rho_{k}(\tilde{x}_s^{(k)}, t)\right) \, d\tilde{\vec{x}}
}.
\label{Eq:density_formula}
\end{align}
\end{remark}
\begin{remark}
We note that we have slightly abused notation in \cref{Eq:density_formula}. In particular,  the expressions $\int  \delta_{x}(\tilde{x}_r^{(j)}) \cdots \, d\tilde{x}_r^{(j)}$ and $\int  \delta_{x}(y_r^{(j)}) \cdots \, dy_r^{(j)}$ are used to represent replacing $\tilde{x}_r^{(j)}$ and $y_r^{(j)}$ with $x$ through the formal action of the $\delta$-function. That is to say for a given function $f$,
\begin{equation*}
\int_{{\R^{d}}}  \delta_{x}(\tilde{x}_r^{(j)}) f(\tilde{x}_r^{(j)}) \, d\tilde{x}_r^{(j)} \doteq f(x).
\end{equation*}
Similarly, for simplicity of notation we have written
\begin{equation*}
\int_{\vy \in \mathbb{Y}^{(\ell)}}   \delta_{x}(y_r^{(j)}) m_\ell(\vec{y}\, | \,\tilde{\vec{x}}) \,d \vec{y},
\end{equation*}
where the integrand formally contains products of shifted $\delta$-functions when using the specific choices of placement densities given by Assumptions~\ref{Assume:measureOne2One} through~\ref{Assume:measureOne2Two}. To write the precise and rigorous version of such expressions, we use the substitution
\begin{equation*}
\int_{\R^{d}} \delta(x-y) \delta(y-\tilde{x}) \, dy := \delta(x-\tilde{x}).
\end{equation*}

In Appendix~\ref{A:PlacementInts} we show how the nested integrals on the second line of~\eqref{Eq:density_formula} simplify using this identity for several choices of $m_{\ell}$ that appear in the following examples.
\end{remark}

To illustrate our main result, we now present a few examples to illustrate the limiting PIDEs for basic reaction types:
\begin{example}\label{Example:birthDeath}
A system with birth and death reactions for one species, $A$. Let $\mathcal{R}_1$ be the death reaction $A \rightarrow \emptyset$ with probability per time $K_1^\gamma(x)$ to happen for a particle at $x$. Since $\mathcal{R}_1$ involves only one reactant and one species, $\alpha_{11} = 1$. As there are no products, $\beta_{11} = 0$.  Let $\mathcal{R}_2$ be the birth reaction $\emptyset \rightarrow A$ with constant probability per time $K_2^\gamma$ to happen. When one birth event occurs, the position of the new A particle is sampled from the placement density $m_2(x)$. For $\mathcal{R}_2$ there are no reactants, so $\alpha_{21} = 0$. One product $A$ particle is generated, so $\beta_{21} = 1$. There are two types of reactions in total, $L = 2$, but reaction $\mathcal{R}_1$ has no products so $\tilde{L} = 1$.

Let the spatial number distribution for particle A at time $t$ be $\nu^{{\vec{\zeta}}, 1}_t\in M(\R^d)$, with $\nu^{{\vec{\zeta}}}_t = \nu^{{\vec{\zeta}}, 1}_t\delta_{S_1}\in M(\hat{P})$.  In this example, we would have $\lambda^{(1)}[\nu^{{\vec{\zeta}}}_t](dx) = \nu^{{\vec{\zeta}}, 1}_t(dx)$. By~\eqref{Eq:EM_formula} $\nu^{\vec{\zeta}}_t$ satisfies
\begin{align}
\la f,\nu^{{\vec{\zeta}}}_{t}\ra_{\hat{P}}&=\la f,\nu^{{\vec{\zeta}}}_{0}\ra_{\hat{P}}+\sum_{i\geq 1}\int_{0}^{t}1_{\{i\leq \la 1, \nu_{s-}^{\vec{\zeta}}\ra_{\hat{P}}\}}\sqrt{2D^{i}}\frac{\partial f}{\partial Q}(H^i(\nu_{s-}^{\vec{\zeta}}))dW^{i}_{s}+\int_{0}^{t}\sum_{i=1}^{\la 1, \nu_{s-}^{\vec{\zeta}}\ra_{\hat{P}}} D^{i}\frac{\partial^{2} f}{\partial Q^{2}}(H^i(\nu_{s-}^{\vec{\zeta}}))ds\nonumber\\
&\quad-\int_{0}^{t}\int_{\left(\N\setminus\{0\}\right)} \int_{\mathbb{R}_{+}} \la f, \delta_{(H_Q^{i}(\nu^{{\vec{\zeta}},1}_{s-}), S_1)} \ra_{\hat{P}}
\times 1_{\{i\leq \la 1,\nu^{{\vec{\zeta}},1}_{s-}\ra\}}\times 1_{\{ \theta\leq K_{1}^{\gamma}(H_Q^{i}(\nu^{{\vec{\zeta}},1}_{s-}))\}} dN_{1}(s, i, \theta)\nonumber\\
&\quad+\int_{0}^{t} \int_{\R^d}\int_{\mathbb{R}_{+}}\int_{\mathbb{R}_{+}}\la f, \delta_{(x, S_1)} \ra_{\hat{P}}
\times  1_{\{ \theta_1\leq K_{2}^{\gamma}\}}\times  1_{\{ \theta_2\leq  m^\eta_{2}(x)\}}
dN_{2}(s,x, \theta_1, \theta_2).
\label{Eq:EM_formula_Aempty}
\end{align}

If the limiting spatial distributed measure for species A has marginal density $\rho(x, t)$, by Remark~\ref{Rm:density} it must solve the following reaction-diffusion equation in a weak sense:
\begin{align}
\partial_t\rho(x, t) &=  D_{1} \lap_x \rho(x, t) -  K_{1}(x)\rho(x, t) + K_2m_{2}(x).\label{Eq:density_formula_birthDeath}
\end{align}
\end{example}

\begin{remark}\label{Remark:measureBirth}
Note that Theorem~\ref{thm:convergence} does not apply to reaction networks including zero order birth reactions. We conjecture it can be extended to allow zero order reactions on some deterministic time interval $[0,T_{*}]$ by introducing and analyzing the limiting behavior of the stopping time~\eqref{eq:stoppingtime}, assuming that the placement density for birth $m_2^{\eta}(x)$ is in $L^1(\R^d)$, and assuming that $m_2^{\eta}$ does not depend on $\eta$, i.e. $m^\eta_2(x) = m_2(x)$. The latter conditions ensure the product particle is most probably placed within a compact subset of $\R^d$.
\end{remark}

\begin{example}\label{Ex:ReversibleCase}
A system with three species, $A$, $B$ and $C$ that can undergo the reversible bimolecular reaction $A+B\rightleftarrows C$. Let $\mathcal{R}_1$ be the reaction $A+B\rightarrow C$, with $K_1^\gamma(x, y)$  the probability per time one $A$ particle at position $x$ and one $B$ particle at position $y$ bind. Once reaction $\mathcal{R}_1$ fires, we generate a new particle C at position $z$ following the placement density $m_1(z \vert x, y)$.  For $\mathcal{R}_1$, the reactants are particle A and B, so $\alpha_{11} =\alpha_{12} = 1$ and $\alpha_{13} = 0$. The product is particle C, so that $\beta_{11} = \beta_{12} = 0$, while $\beta_{13} = 1$.

Let $\mathcal{R}_2$ be the reaction $C \rightarrow A+B$, with $K_2^\gamma(z)$ the probability per time one $C$ particle at position $z$ unbinds. Once reaction $\mathcal{R}_2$ fires, we generate a new $A$ particle at position $x$ and $B$ particle at position $y$ following the placement density $m_2(x, y \vert z)$. For $\mathcal{R}_2$, the reactant is a $C$ particle, so $\alpha_{21} = \alpha_{22} = 0$ and $\alpha_{23} = 1$. The products are $A$ and $B$ particles, so that $\beta_{21} = \beta_{22} = 1$, while $\beta_{23} = 0$.

Let the spatial number distribution for $A$ particles at time $t$ be $\nu^{{\vec{\zeta}}, 1}_t\in M(\R^d)$, the spatial number distribution for $B$ particles at time $t$ be $\nu^{{\vec{\zeta}}, 2}_t\in M(\R^d)$ and the spatial number distribution for $C$ particles at time $t$ be $\nu^{{\vec{\zeta}}, 3}_t\in M(\R^d)$. Then $\nu^{{\vec{\zeta}}}_t = \nu^{{\vec{\zeta}}, 1}_t\delta_{S_1} + \nu^{{\vec{\zeta}}, 2}_t\delta_{S2}+ \nu^{{\vec{\zeta}}, 3}_t\delta_{S3}\in M(\hat{P})$. We have that $\lambda^{(1)}[\nu^{{\vec{\zeta}}}_t](d\vec{x}) = \nu^{{\vec{\zeta}}, 1}_t(dx_1^{(1)})\,\nu^{{\vec{\zeta}}, 2}_t(dx_1^{(2)})$ and $\lambda^{(2)}[\nu^{{\vec{\zeta}}}_t](d\vec{x}) = \nu^{{\vec{\zeta}}, 3}_t(dx_1^{(3)})$. $\nu_t^{\vec{\zeta}}$ then satisfies
\begin{align}
\la f,\nu^{{\vec{\zeta}}}_{t}\ra_{\hat{P}}&=\la f,\nu^{{\vec{\zeta}}}_{0}\ra_{\hat{P}}+\sum_{i\geq 1}\int_{0}^{t}1_{\{i\leq \la 1, \nu_{s-}^{\vec{\zeta}}\ra_{\hat{P}}\}}\sqrt{2D^{i}}\frac{\partial f}{\partial Q}(H^i(\nu_{s-}^{\vec{\zeta}}))dW^{i}_{s}+\int_{0}^{t}\sum_{i=1}^{\la 1, \nu_{s-}^{\vec{\zeta}}\ra_{\hat{P}}} D^{i}\frac{\partial^{2} f}{\partial Q^{2}}(H^i(\nu_{s-}^{\vec{\zeta}}))ds\nonumber\\
&
\quad+\int_{0}^{t}\int_{\left(\N\setminus\{0\}\right)^{2}} \int_{\R^d}\int_{\mathbb{R}_{+}}\int_{\mathbb{R}_{+}}
\la f, -\delta_{(H_Q^{i}(\nu^{{\vec{\zeta}},1}_{s-}), S_1)}-\delta_{(H_Q^{j}(\nu^{{\vec{\zeta}},2}_{s-}), S_2)}+\delta_{(z, S_3)} \ra_{\hat{P}} \times 1_{\{i\leq \la 1,\nu^{{\vec{\zeta}},1}_{s-}\ra\}} \nonumber\\
&\qquad \times 1_{\{j\leq \la 1,\nu^{{\vec{\zeta}},2}_{s-}\ra\}}\times 1_{\{ \theta_1\leq K_{1}^{\gamma}\left(H_Q^{i}(\nu^{{\vec{\zeta}},1}_{s-}), H_Q^{j}(\nu^{{\vec{\zeta}},2}_{s-})\right)\}}
\times 1_{\{ \theta_2\leq  m^\eta_{1}\left(z|H_Q^{i}(\nu^{{\vec{\zeta}},1}_{s-}),H_Q^{j}(\nu^{{\vec{\zeta}},2}_{s-})\right)\}}  dN_{1}(s,i,j,z,\theta_1, \theta_2)\nonumber\\
&\quad
+\int_{0}^{t}\int_{\N\setminus\{0\}} \int_{\R^d\times\R^d}\int_{\mathbb{R}_{+}}\int_{\mathbb{R}_{+}}\la f, \delta_{(x, S_1)}+\delta_{(y, S_2)}-\delta_{(H_Q^{k}(\nu^{{\vec{\zeta}},3}_{s-}), S_3)} \ra_{\hat{P}} \times 1_{\{k\leq \la 1,\nu^{{\vec{\zeta}},3}_{s-}\ra\}} \nonumber\\
&\qquad \times  1_{\{ \theta_1 \leq K_{2}^{\gamma}\left(H_Q^{k}(\nu^{{\vec{\zeta}},3}_{s-})\right)\}}\times  1_{\{ \theta_2 \leq \times  m^\eta_{2}\left(x,y|H_Q^{k}(\nu^{{\vec{\zeta}},3}_{s-})\right)\}}dN_{2}(s,k,x,y,\theta_1, \theta_2).\label{Eq:EM_formula_ABC}
\end{align}

If the limiting spatially distributed measures for species A, B and C have marginal densities $(\rho_1(x, t), \rho_2(x, t),  \rho_3(x, t))$ respectively, by Remark \ref{Rm:density} they must solve the following reaction-diffusion equations in a weak sense:
\begin{align}
\partial_t\rho_1(x, t) &=  D_1 \lap_x \rho_1(x, t) -  \left(\int_{\R^d} K_{1}(x, y) \rho_2(y, t) \, dy\right) \rho_1(x, t) + \int_{\R^d}K_{2}(z)\left( \int_{\R^d} m_{2}(x,y|z) dy \right)\rho_3(z, t)\, dz \nonumber\\
\partial_t \rho_2(y, t)&=  D_2 \lap_y \rho_2(y, t) - \left(\int_{\R^d} K_{1}(x, y) \rho_1(x, t) \, dx\right) \rho_2(y, t)+ \int_{\R^d}K_{2}(z)\left( \int_{\R^d} m_{2}(x,y|z) dx \right)\rho_3(z, t)\, dz \nonumber\\
\partial_t \rho_3(z, t)&=  D_3 \lap_z \rho_3(z, t) -  K_{2}(z)\rho_3(z, t) + \int_{\R^{2d}} K_{1}(x, y)m_1(z | x, y) \rho_1(x, t) \rho_2(y, t) \, dx \, dy.\nonumber\\
\label{Eq:density_formula_reversible}
\end{align}
\end{example}

\begin{example}
A system with two species, $A$ and $B$, that can undergo the reversible dimerization reaction $A+A\rightleftarrows B$. Let $\mathcal{R}_1$ be the reaction $A+A\rightarrow B$ with $K_1^\gamma(x, y)$ the probability per time one $A$ particle at position $x$ and another $A$ particle at position $y$ bind. Once reaction $\mathcal{R}_1$ fires, we generate a new $B$ particle at position $z$ by sampling from the placement density $m_1(z \vert x, y)$.  For $\mathcal{R}_1$, the reactants are two $A$ particles, so $\alpha_{11} = 2$ and $\alpha_{12} = 0$. The product is one $B$ particle, so that $\beta_{11}  = 0$ and $\beta_{12} = 1$.

Let $\mathcal{R}_2$ be the reaction $B \rightarrow A+A$, with $K_2^\gamma(z)$ the probability per time one $B$ particle at position $z$ unbinds. Once reaction $\mathcal{R}_2$ fires, we generate two new $A$ particles at positions $x$ and $y$ by sampling from the placement density $m_2(x, y \vert z)$. For $\mathcal{R}_2$, the reactant is one $B$ particle, so $\alpha_{21} = 0$ and $\alpha_{22} = 1$. The products are two $A$ particles, so that $\beta_{21}  = 2$ and $\beta_{22} = 0$.

Let the spatial number distribution for $A$ particles at time $t$ be $\nu^{{\vec{\zeta}}, 1}_t\in M(\R^d)$ and the spatial number distribution for $B$ particles at time $t$ be $\nu^{{\vec{\zeta}}, 2}_t\in M(\R^d)$. Then $\nu^{{\vec{\zeta}}}_t = \nu^{{\vec{\zeta}}, 1}_t\delta_{S_1} + \nu^{{\vec{\zeta}}, 2}_t\delta_{S2}\in M(\hat{P})$.  We have that $\lambda^{(1)}[\nu^{{\vec{\zeta}}}_t](d\vec{x}) = \nu^{{\vec{\zeta}}, 1}_t(dx_1^{(1)})\nu^{{\vec{\zeta}}, 1}_t(dx_2^{(1)})$ and $\lambda^{(2)}[\nu^{{\vec{\zeta}}}_t](d\vec{x}) = \nu^{{\vec{\zeta}}, 2}_t(dx_1^{(2)})$. $\nu_t^{\vec{\zeta}}$ then satisfies
\begin{align}
\la f,\nu^{{\vec{\zeta}}}_{t}\ra_{\hat{P}}&=\la f,\nu^{{\vec{\zeta}}}_{0}\ra_{\hat{P}}+\sum_{i\geq 1}\int_{0}^{t}1_{\{i\leq \la 1, \nu_{s-}^{\vec{\zeta}}\ra_{\hat{P}}\}}\sqrt{2D^{i}}\frac{\partial f}{\partial Q}(H^i(\nu_{s-}^{\vec{\zeta}}))dW^{i}_{s}+\int_{0}^{t}\sum_{i=1}^{\la 1, \nu_{s-}^{\vec{\zeta}}\ra_{\hat{P}}} D^{i}\frac{\partial^{2} f}{\partial Q^{2}}(H^i(\nu_{s-}^{\vec{\zeta}}))ds\nonumber\\
&
\quad+\int_{0}^{t}\int_{\left(\N\setminus\{0\}\right)^{2}} \int_{\R^d}\int_{\mathbb{R}_{+}}\int_{\mathbb{R}_{+}}
\la f, -\delta_{(H_Q^{i}(\nu^{{\vec{\zeta}},1}_{s-}), S_1)}-\delta_{(H_Q^{j}(\nu^{{\vec{\zeta}},1}_{s-}), S_1)}+\delta_{(z, S_2)} \ra_{\hat{P}} \times 1_{\{i <  j\leq \la 1,\nu^{{\vec{\zeta}},1}_{s-}\ra\}}\nonumber\\
&\qquad
\times 1_{\{ \theta_1 \leq K_{1}^{\gamma}\left(H_Q^{i}(\nu^{{\vec{\zeta}},1}_{s-}), H_Q^{j}(\nu^{{\vec{\zeta}},1}_{s-})\right)\}}\times 1_{\{ \theta_2 \leq  m^\eta_{1}\left(z|H_Q^{i}(\nu^{{\vec{\zeta}},1}_{s-}),H_Q^{j}(\nu^{{\vec{\zeta}},1}_{s-})\right)\}} dN_{1}(s,i,j,z,\theta_1, \theta_2)\nonumber\\
&\quad+\int_{0}^{t}\int_{\N\setminus\{0\}} \int_{\R^d\times\R^d}\int_{\mathbb{R}_{+}}\int_{\mathbb{R}_{+}}
\la f, \delta_{(x, S_1)}+\delta_{(y, S_1)}-\delta_{(H_Q^{k}(\nu^{{\vec{\zeta}},2}_{s-}), S_2)} \ra_{\hat{P}}
\times 1_{\{k\leq \la 1,\nu^{{\vec{\zeta}},2}_{s-}\ra\}}
\nonumber\\
&\qquad
\times 1_{\{ \theta_{1}\leq K_{2}^{\gamma}\left(H_Q^{k}(\nu^{{\vec{\zeta}},2}_{s-})\right)\}}
\times 1_{\{ \theta_{2}\leq  m^\eta_{2}\left(x,y|H_Q^{k}(\nu^{{\vec{\zeta}},2}_{s-})\right)\}}dN_{2}(s,k,x,y,\theta_1, \theta_2).
\label{Eq:EM_formula_AAB}
\end{align}

If the spatially distributed measures for species $A$ and $B$ have marginal densities $(\rho_1(x, t), \rho_2(z, t))$ respectively, then from Remark \ref{Rm:density} they must solve the following reaction-diffusion equations in a weak sense:
\begin{align}
\partial_t\rho_1(x, t) &=  D_1 \lap_x \rho_1(x, t) -  \left(\int_{\R^d} K_{1}(x, y) \rho_1(y, t)\, dy \right) \rho_1(x, t) + 2\int_{\R^d}K_{2}(z)\left( \int_{\R^d} m_{2}(x,y|z) dy \right)\rho_2(z, t)\, dz \nonumber\\
\partial_t \rho_2(z, t)&=  D_2 \lap_z \rho_2(z, t) -  K_{2}(z)\rho_2(z, t) + \frac{1}{2}\int_{\R^{2d}} K_{1}(x, y)m_1(z | x, y) \rho_1(x, t) \rho_1(y, t) \, dx \, dy.\nonumber\\
\label{Eq:density_formula_twoAreversible}
\end{align}
\end{example}

\section{Equivalence between measure valued formulation and forward Kolmogorov equation}\label{S:EquivalentDerivations}
In this section, we demonstrate equivalence of the measure-valued stochastic process formulation~\eqref{Eq:EM_formula} to the forward Kolmogorov equation representation of the volume-reactivity model popularized by Doi~\cite{DoiSecondQuantA,DoiSecondQuantB}. For ease of notation, and brevity of presentation, we restrict attention to the special case of the reversible $A+B\rightleftarrows C$ reaction, i.e. Example~\ref{Ex:ReversibleCase}. Though we do not show here the general case, we note that this reversible example includes the key complicating components; two-body particle interactions and changing (total) numbers of particles. It is therefore illustrative of other reactions that may only involve particle creation (e.g. $\varnothing \to A$), or involve interactions but preserve particle numbers (e.g. $A + B \to C + D$ or $A \to B$).

Denote by $K_1^{\gamma}(x, y)$ the probability per time a particle of type $A$ at $x$ and a particle of type $B$ at $y$ react, and by $K_2^{\gamma}( z )$ the probability per time a particle of type C at $z$ dissociates. We let $m^\eta_1(z\, |\, x, y)$ be the corresponding placement density for the $A+B\rightarrow C$ reaction,  producing a particle of type C at $z$, given a particle of type A at $x$ and a particle of type B at $y$. Similarly, $m^\eta_2(x, y \, | \, z)$ is the placement density for the $C\rightarrow A+B$ reaction, producing a particle of type A at $x$ and a particle of type B at $y$ given a particle of type C  at $z$ respectively.

\subsection{Weak MVSP  Formulation for the $\textrm{A} + \textrm{B} \leftrightarrows \textrm{C}$ Reaction}\label{sec:MFF_RR}
The weak MVSP representation is given by~\eqref{Eq:EM_formula_ABC}. Taking expectation we obtain for the mean
\begin{align}
\avg{\la f,\nu^{{\vec{\zeta}}}_{t}\ra_{\hat{P}}}
&
=\avg{\la f,\nu^{{\vec{\zeta}}}_{0}\ra_{\hat{P}}}+\avg{\int_{0}^{t}\sum_{i=1}^{\la 1, \nu_{s-}^{\vec{\zeta}}\ra_{\hat{P}}} D^{i}\frac{\partial^{2} f}{\partial Q^{2}}(H^i(\nu_{s-}^{\vec{\zeta}}))ds}\nonumber\\
&
\quad+\avg{\int_{0}^{t}\sum_{i=1}^{\la1,\nu^{{\vec{\zeta}},1}_{s-}\ra}\sum_{j=1}^{\la1,\nu^{{\vec{\zeta}},2}_{s-}\ra} \int_{\R^d}\la f,-\delta_{(H_Q^{i}(\nu^{{\vec{\zeta}},1}_{s-}), S_1)}-\delta_{(H_Q^{j}(\nu^{{\vec{\zeta}},2}_{s-}), S_2)}+\delta_{(z, S_3)} \ra_{\hat{P}} \nonumber\\
&
\quad\quad \times K_{1}^{\gamma}\paren{H_Q^{i}(\nu^{{\vec{\zeta}},1}_{s-}), \,H_Q^{j}(\nu^{{\vec{\zeta}},2}_{s-})} \, m^\eta_{1}\paren{z|H_Q^{i}(\nu^{{\vec{\zeta}},1}_{s-}),H_Q^{j}(\nu^{{\vec{\zeta}},2}_{s-})} dz \,ds}\nonumber\\
&
\quad+\avg{\int_{0}^{t} \sum_{k=1}^{\la1,\nu^{{\vec{\zeta}},3}_{s-}\ra} \int_{\R^{2d}}
\la f,\delta_{(x, S_1)}+\delta_{(y, S_2)}-\delta_{(H_Q^{k}(\nu^{{\vec{\zeta}},3}_{s-}), S_3)} \ra_{\hat{P}}\nonumber \\
&\quad\quad
\times K_{2}^{\gamma}\paren{H_Q^{k}(\nu^{{\vec{\zeta}},3}_{s-})} m^\eta_{2}\paren{x,y|H_Q^{k}(\nu^{{\vec{\zeta}},3}_{s-})}  dx\, dy\, ds}\nonumber \\
&
=\avg{\la f,\nu^{{\vec{\zeta}}}_{0}\ra_{\hat{P}}}+\avg{\int_{0}^{t} \la \mathcal{L}f, \nu^{{\vec{\zeta}}}_{s-}\ra_{\hat{P}} ds}\nonumber\\
&
\quad+\avg{\int_{0}^{t}\int_{\R^{3d}}\left(- f(x, S_1) - f(y, S_2)+ f(z, S_3) \right) K_{1}^{\gamma}(x, y) \, m^\eta_{1}(z|x,y) \,\nu^{{\vec{\zeta}}, 1}_{s-}(dx)\, \nu^{{\vec{\zeta}}, 2}_{s-}(dy)\,dz \,ds}\nonumber\\
&
\quad+\avg{\int_{0}^{t}  \int_{\R^{3d}} \left( f(x, S_1) + f(y, S_2)-f(z, S_3) \right) K_{2}^{\gamma}(z) \, m^\eta_{2}(x,y|z) \, dx\, dy\, \nu^{{\vec{\zeta}}, 3}_{s-}(dz)\,ds}.
\label{Eq:MEAN_EM_formula_ABC}
\end{align}
In Eq (\ref{Eq:MEAN_EM_formula_ABC}), we denote $\mathcal{L}f(Q, S_j) = D_{j}\Delta_{Q}f$, for all $j = 1, 2, 3$. As we shall demonstrate soon, this is consistent with what we expect from the forward equation~\eqref{eq:multipartABtoCEqs}.

\subsection{Doi Forward Kolmogorov Equation  for the $\textrm{A} + \textrm{B} \leftrightarrows \textrm{C}$ Reaction}\label{sec:forwardEq}
We use a notation consistent with that introduced by Doi~\cite{DoiSecondQuantA,DoiSecondQuantB}. Suppose $A(t)$ is the stochastic process for the number of species \textrm{A} particles in the system at time $t$, with $B(t)$ and $C(t)$ defined similarly. Values of $A(t)$, $B(t)$ and $C(t)$ will be given by $a$, $b$ and $c$ (i.e. $A(t)=a$). When $A(t)=a$, we will let $\vQa_l(t) \in \R^d$ label the stochastic process for the position of the $l$th molecule of species \textrm{A}. $\vqa_l$ will denote a possible value of $\vQa_l(t)$.  The species \textrm{A} position vector when $A(t)=a$ is then given by
\begin{equation*}
  \vQa(t) = (\vQa_1(t), \dots, \vQa_a(t)) \in \R^{da}.
\end{equation*}

Similarly, $\vqa$ will denote a possible value of $\vQa(t)$,
\begin{equation*}
  \vQa(t) = \vqa = (\vqa_1,\dots,\vqa_a).
\end{equation*}
$\vQb(t)$, $\vQb_m(t)$, $\vQc(t)$,
$\vQc_n(t)$, $\vqb_m$, $\vqc_n$, $\vqb$ and $\vqc$ will all be defined analogously. The state of the system is then a hybrid discrete--continuous state stochastic process given by
$\paren{A(t), B(t), C(t),\vQ^{A(t)}, \vQ^{B(t)}, \vQ^{C(t)}}$.

With this notation, denote by $\pabc(\vqa,\vqb,\vqc,t)$ the probability density that $A(t)=a$, $B(t)=b$ and $C(t)=c$ with $\vQa(t)=\vqa$, $\vQb(t)=\vqb$ and $\vQc(t) = \vqc$. We assume that particles of the same species are indistinguishable, that is for $1 \leq l< l' \leq a$ fixed
\begin{multline*}
\pabc \paren{\vqa_1,\dots,\vqa_{l-1},\vqa_l,\vqa_{l+1},\dots,\vqa_{l'-1},\vqa_{l'},\vqa_{l'+1},
  \dots,\vqa_a,\vqb,\vqc,t} \\
= \pabc \paren{\vqa_1,\dots,\vqa_{l-1},\vqa_{l'},\vqa_{l+1},\dots,\vqa_{l'-1},\vqa_{l},\vqa_{l'+1},
  \dots,\vqa_a,\vqb,\vqc,t},
\end{multline*}
with similar relations holding for permutations of the molecule orderings within $\vqb$ and $\vqc$.  With this assumption the $p^{(a,b,c)}$ are chosen to be normalized so that
\begin{equation*}
  \sum_{a=0}^{\infty} \sum_{b=0}^{\infty} \sum_{c=0}^{\infty} \brac{ \frac{1}{a!\, b!\, c!}
  \int_{\R^{da}} \int_{\R^{db}} \int_{\R^{dc}} \pabc\paren{\vqa,\vqb,\vqc,t} \, d\vqc \, d\vqb \, d\vqa
  } = 1.
\end{equation*}
Here the bracketed term corresponds to the probability of having a given number of each species, i.e.
\begin{equation*}
  \prob \brac{A(t) = a, B(t) = b, C(t) = c}  = \frac{1}{a!\, b!\, c!}
  \int_{\R^{da}} \int_{\R^{db}} \int_{\R^{dc}} \pabc\paren{\vqa,\vqb,\vqc,t} \, d\vqc \, d\vqb \, d\vqa.
\end{equation*}

Let $\vP(t) = \{p^{(a,b,t)}(\vqa,\vqb,\vqc,t)\}_{a,b,c=0}^{\infty}$ denote the vector of all the probabilities. The forward equation (see~\cite{IZ:2018}) is given by the coupled system of PIDEs that
\begin{equation}
  \label{eq:multipartABtoCEqs}
   \partial_t \vP(t) = (\diffop + \Rp + \Rm) \vP(t).
\end{equation}
Here the linear operators $\diffop$, $\Rp$ and $\Rm$ correspond to diffusion, the forward association reaction and the reverse dissociation reaction respectively. The diffusion operator in the $(a,b,c)$ equation is given by
\begin{equation}
  \label{eq:multipartDiffop}
  (\diffop \vP(t))_{(a,b,c)} (\vq^a,\vq^b,\vq^c. t) = \paren{D_1 \sum_{l=1}^{a} \lap_{\vqa_l}
    + D_2 \sum_{m=1}^{b} \lap_{\vqb_m}
    + D_3 \sum_{n=1}^{c} \lap_{\vqc_n}} \pabc(\vqa,\vqb,\vqc,t),
\end{equation}
where $\lap_{\vqa_l}$ denotes the $d$-dimensional Laplacian acting on the $\vqa_l$ coordinate, and $\lap_{\vqb_m}$ and $\lap_{\vqc_n}$ are defined similarly. (Recall $D_1, D_2, D_3$ are the diffusivity of species \textrm{A},  \textrm{B}.  and \textrm{C} respectively.) To define the reaction operators, $\Rp$ and $\Rm$, we introduce notations for adding or removing a particle from a given state, $\vqa$. Let
\begin{align*}
  \vqa \cup \vx &= \paren{ \vqa_1, \dots, \vqa_a, \vx}, &
  \vqa \setminus \vqa_l &= \paren{ \vqa_1,\dots,\vqa_{l-1}, \vqa_{l+1}, \dots, \vqa_a},
\end{align*}
which correspond to adding a particle to species \textrm{A} at $\vx$, and removing the $l$th particle of species \textrm{A} respectively. With these definitions, the reaction operator for the $\textrm{A} + \textrm{B} \to \textrm{C}$ association reaction in the $(a,b,c)$ equation is
\begin{equation} \label{eq:multipartRp}
  \begin{aligned}
    (\Rp \vP(t))_{(a,b,c)} (\vq^a,\vq^b,\vq^c. t) =
    &-\paren{\sum_{l=1}^a \sum_{m=1}^bK_1^{\gamma} \paren{\vqa_{l}, \vqb_{m}}} \pabc(\vqa,\vqb,\vqc,t) \\
    &+ \sum_{n=1}^{c} \!\brac{\int_{\R^{2d}} m^\eta_1(\vqc_n \vert \vx, \vy) K_1^{\gamma} \paren{\vx, \vy}
      p^{(a+1,b+1,c-1)}(\vqa \cup \vx, \vqb \cup \vy, \vqc \setminus \vqc_n, t) d\vx d\vy}\!\!,
  \end{aligned}
\end{equation}
while the reaction operator for the dissociation reaction $\textrm{C} \to \textrm{A} + \textrm{B}$ in the $(a,b,c)$ equation is
\begin{multline} \label{eq:multipartRm}
  (\Rm \vP(t))_{(a,b,c)} (\vq^a,\vq^b,\vq^c. t)= - \paren{ \sum_{n=1}^{c} K_2^{\gamma}(\vqc_n)} \pabc(\vqa, \vqb, \vqc, t) \\
  + \sum_{l=1}^a \sum_{m=1}^b \brac{\int_{\R^d} m^\eta_2\paren{\vqa_l,\vqb_m\vert\vz}K_2^{\gamma}(\vz)
    p^{(a-1,b-1,c+1)}\paren{\vqa\setminus \vqa_l, \vqb\setminus \vqb_m, \vqc \cup \vz, t} d\vz}.
\end{multline}
This representation is consistent with the classical second quantization representation of Doi~\cite{DoiSecondQuantA,DoiSecondQuantB}.

Suppose the initial condition $\vP(0) = \vP_0 = \{p_0^{(a,b,c)}\}_{a,b,c=0}^{\infty}$ is fixed, and we have $(a_0, b_0, c_0)$ particles of A, B and C respectively at time zero. We consider the evolution of $\vP(t)$ as a vector in a $L^2$ Fock Space $F=L^{2}(X)$ equipped with inner product defined by~\eqref{eq:FockInnerProduct}, where
\begin{equation*}
X=\bigoplus_{\substack{a,b,c\geq0\\a+b+2c = a_0+b_0+2c_0}}^{a_0 \vee b_0 + c_0} \mathbb{R}^{d(a+b+c)}.
\end{equation*}

\begin{remark}\label{rem:finiteSum}
For the $\textrm{A} + \textrm{B} \leftrightarrows \textrm{C}$ reaction, the quantity $A(t)+B(t)+ 2 C(t) = a_0+b_0+2c_0$ is always conserved. For our example, $X$ is therefore a finite sum of Euclidean spaces over $a, b, c\in\N_+$  such that $a+b+2c = a_0+b_0+2c_0$.
\end{remark}

To simplify the calculation of regularity results for~\eqref{eq:multipartABtoCEqs} \emph{for comparison to the forward equation}, in this section we make
\begin{assumption}\label{assumption:twoToOneRate}
We assume the reaction kernel function $K(x, y)$ for the $A + B \rightarrow C$ reaction only depends on the separation of two reactant particles, $|x-y|$, denoted as $K(x, y) = \tilde{K}(|x-y|)$. Furthermore, we assume $\tilde{K}(|w|)\in L^2(\R^d)$, $w\in\R^d$.
\end{assumption}
\begin{assumption}\label{assumption:oneToTwoRate}
We assume the function $\rho(|w|)$, $w\in\R^d$, defined in Assumption~\ref{Assume:measureOne2Two} for the $C \rightarrow A+B$ reaction, is in $L^2(\R^d)$.
\end{assumption}

Under these two assumptions the following regularity theorem holds, for which the proof is given in Appendix~\ref{A:ProofPDE}.
\begin{theorem}\label{thm:regularityABC}
Given Assumptions~\ref{assumption:twoToOneRate} and~\ref{assumption:oneToTwoRate}, there exists a unique global mild solution to~\eqref{eq:multipartABtoCEqs},  $\vP(\cdot)\in \mathcal{C}([0,\infty);H^2(X))$. That is, $\vP(t)$ satisfies
\begin{equation*}
\vP(t)=e^{t\diffop} \vP_{0} + \int_{0}^{t} e^{(t-s)\diffop}(\Rp \vP(s) + \Rm \vP(s) ) \, ds,
\end{equation*}
with the initial condition $\vP(0) = \vP_0\in H^1(X)$. Further, if $p_0^{(a,b,c)} \geq 0$ for each $(a,b,c)$, and satisfies the normalization condition
\begin{equation*}
  \sum_{a=0}^{\infty} \sum_{b=0}^{\infty} \sum_{c=0}^{\infty} \brac{ \frac{1}{a!\, b!\, c!}
  \int_{\R^{da}} \int_{\R^{db}} \int_{\R^{dc}} \pabc_0\paren{\vqa,\vqb,\vqc} \, d\vqc \, d\vqb \, d\vqa
  } = 1,
\end{equation*}
then $\vP(t)$ is always non-negative for all $t\geq 0$ and the same normalization condition holds,
\begin{equation*}
  \sum_{a=0}^{\infty} \sum_{b=0}^{\infty} \sum_{c=0}^{\infty} \brac{ \frac{1}{a!\, b!\, c!}
  \int_{\R^{da}} \int_{\R^{db}} \int_{\R^{dc}} \pabc\paren{\vqa,\vqb,\vqc,t} \, d\vqc \, d\vqb \, d\vqa
  } = 1.
\end{equation*}
Note, as $a + b + 2 c$ is conserved, see Remark~\ref{rem:finiteSum}, the above summation is only over a finite set of indices.
\end{theorem}

\subsection{Equivalence of the two approaches}\label{SS:EquivalentApproaches}

Now, we are in position to compare the two approaches as described in Sections \ref{sec:MFF_RR} and \ref{sec:forwardEq}. For this purpose we have Proposition \ref{P:EquivalentGenerator} and Proposition \ref{P:ParticleSpatialDensity} whose proofs are deferred to Appendix~\ref{A:ProofsEquivalentDerivations}.

\begin{proposition}\label{P:EquivalentGenerator}
For any $M\in \N$, any function $\varphi\in C^\infty(\R^M)$ and any set of functions $\{f_m\}_{m = 1}^M\in C_b^2(\hat{P})$, the evolution equations satisfied by $\avg{\varphi\left(\la f_1, \nu^{\vec{\zeta}}_t \ra_{\hat{P}}, \la f_2, \nu^{\vec{\zeta}}_t \ra_{\hat{P}}, \cdots, \la f_M, \nu^{\vec{\zeta}}_t \ra_{\hat{P}} \right)}$ are the same when derived using either the formulation of the weak measure-valued process representation, or the formulation based on the forward Kolmogorov equation. This implies that these two approaches produce the same statistics.
\end{proposition}

By Proposition \ref{P:EquivalentGenerator} we know that the measure-valued formulation and the forward Kolmogorov equation yield the same statistics, at least when the statistics involve smooth test functions. We next derive equations for the mean particle spatial field density of different species at time $t$ in Proposition~\ref{P:ParticleSpatialDensity}. Note that Proposition~\ref{P:ParticleSpatialDensity} can be viewed as a special case of the equations derived in Proposition~\ref{P:EquivalentGenerator} by formally setting $M = 1$, $\varphi = 1$ and $f_1 = \delta_{(x, S_1)}$, $f_1 = \delta_{(y, S_2)}$ or $f_1 = \delta_{(z, S_3)}$ respectively for representing the spatial field  density for species $A$, $B$, and $C$. Using the weak measure-valued process representation, we do not show here but expect that one can make this rigorous by introducing appropriate mollifiers and taking the mollification to zero afterwards. For simplicity, we instead use the formulation via the forward Kolmogorov equation, which for deriving the mean density fields in Proposition~\eqref{P:ParticleSpatialDensity} does not require the introduction of mollifiers.

\begin{proposition}\label{P:ParticleSpatialDensity}
Let $A(\vx, t)$, $B(\vy, t)$ and $C(\vz, t)$ denote the spatial field number density in the particle model at time $t$ for species A at position $x$, B at position $y$ and C at position $z$ respectively.  Under the assumptions of Theorem~\ref{thm:regularityABC}, the evolution equations for their expectations at time $t$ satisfy
\begin{align}\label{Eq:Mean_Field_Model_From_Foward_Eqs0}
    \partial_t \avg{A(\vx,t)} &= D_1 \lap_{\vx} \avg{A(\vx,t)} - \int_{\R^{d}} K_1^{\gamma} \paren{\vx,\vy}\avg{A(\vx,t) B(\vy,t)} \, d\vy \nonumber\\
    &\hspace{6cm}+ \int_{\R^{d}} \brac{\int_{\R^d}m_2\paren{\vx,\vy\vert\vz}K_2^{\gamma}(\vz) d\vy } \avg{C(\vz,t)} d\vz,\nonumber\\
    \partial_t \avg{B(\vy,t)} &= D_2\lap_{\vy} \avg{B(\vy,t)} - \int_{\R^{d}} K_1^{\gamma}\paren{\vx,\vy}\avg{A(\vx,t) B(\vy,t)} \, d\vx \nonumber\\
   &\hspace{6cm} + \int_{\R^{d}} \brac{\int_{\R^d}m_2\paren{\vx,\vy\vert\vz}K_2^{\gamma}(\vz) d\vx } \avg{C(\vz,t)} d\vz,\nonumber\\
    \partial_t \avg{C(\vz,t)} &= D_3 \lap_{\vz} \avg{C(\vz,t)} -  K_2^{\gamma} \paren{\vz}\avg{ C(\vz,t)}  +  \int_{\R^{2d}}m_1 \paren{\vz\vert \vx,\vy}K_1^{\gamma}( \vx,\vy)\avg{A(\vx,t) B(\vy,t)} \, d\vx\, d\vy .\nonumber\\
\end{align}
\end{proposition}

\section{Details of the Proof of Theorem~\ref{thm:convergence}}\label{S:ProofMainResult}
The purpose of this section is to prove the various lemmas and theorems cited in the proof of our main result, Theorem~\ref{thm:convergence}. Without loss of generality we assume that $\tilde{L} = 0$ in this section. The case when $\tilde{L} > 0$ follows by similar arguments as we now give in the $\tilde{L}=0$ case.

To rigorously determine the large-population limit of the MVSP, we use the martingale problem approach for studying solutions to stochastic differential equations developed by Stroock and Varadhan~\cite{EthierKurtz, StroockVaradhan}. The proof is organized as follows.  In Subsection~\ref{SS:PathLevelDescription} we provide the path level description of $\mu^{{\vec{\zeta}}, j}_{t}$, analogous to~\eqref{Eq:EM_formula} for $\nu_t^{\vec{\zeta}}$, and in Subsection~\ref{SS:TakingExpectations} we derive equations for its expectation. Assuming that the large-population limit exists, its identification is presented in Subsection~\ref{SS:LimitIdentification}. Then, in Subsection~\ref{S:Tightness} we prove that the limit exists by proving that the sequence of measures is appropriately tight. We conclude in Subsection~\ref{S:Uniqueness} by proving that the limit equation has a unique solution. Collectively, these results imply Theorem~\ref{thm:convergence}.

\subsection{Path level description}\label{SS:PathLevelDescription}

Using that we can write the marginal distribution (molar concentration) of species $j$ as
\begin{equation*}
\mu_t^{{\vec{\zeta}}, j}(dx) = \frac{1}{{\gamma}}\sum_{i = 1}^{{\gamma}\la 1, \mu^{{\vec{\zeta}}, j}_t\ra} \delta_{H_Q^i({\gamma}\mu_t^{{\vec{\zeta}}, j})}(dx), \quad j \in \{1,\dots,J\},
\end{equation*}
we have, analogously to~\eqref{Eq:EM_formula}, the coupled system
\begin{align}
\la f,\mu^{{\vec{\zeta}}, j}_{t}\ra
&
=\la f,\mu^{{\vec{\zeta}}, j}_{0}\ra + \frac{1}{{\gamma}}\sum_{i\geq 1}\int_{0}^{t}1_{\{i\leq \gamma \la 1, \mu_{s-}^{{\vec{\zeta}}, j}\ra\}}\sqrt{2D_{j}}\frac{\partial f}{\partial Q}(H^i(\gamma\mu_{s-}^{{\vec{\zeta}}, j}))dW^{i}_{s}+\frac{1}{\gamma}\int_{0}^{t}\sum_{i=1}^{\gamma\la 1, \mu_{s-}^{{\vec{\zeta}}, j}\ra} D_{j}\frac{\partial^{2} f}{\partial Q^{2}}(H^i(\gamma\mu_{s-}^{{\vec{\zeta}}, j}))ds\nonumber\\
&
\quad+\sum_{\ell = 1}^L \int_{0}^{t}\int_{\mathbb{I}^{(\ell)}} \int_{\mathbb{Y}^{(\ell)}}\int_{\mathbb{R}_{+}^2}\left(\la f, \mu^{{\vec{\zeta}}, j}_{s-}-\frac{1}{\gamma} \sum_{r = 1}^{\alpha_{\ell j}} \delta_{H_Q^{i_r^{(j)}}(\gamma\mu^{{\vec{\zeta}},j}_{s-})} + \frac{1}{\gamma}\sum_{r = 1}^{\beta_{\ell j}} \delta_{y_r^{(j)}} \ra  - \la f,\mu^{{\vec{\zeta}}, j}_{s-}\ra\right)\nonumber\\
&
\qquad
\times 1_{\{\vec{i} \in \Omega^{(\ell)}(\gamma\mu^{\vec{\zeta}}_{s-})\}}   \times 1_{ \{ \theta_1 \leq K_\ell^{\gamma}\left(\mathcal{P}^{(\ell)}(\gamma\mu_{s-}^{\vec{\zeta}}, \vec{i}) \right) \}}  \times 1_{ \{ \theta_2 \leq  m^\eta_\ell\left(\vec{y} \,  | \, \mathcal{P}^{(\ell)}(\gamma\mu_{s-}^{\vec{\zeta}}, \vec{i}) \right) \}}  dN_{\ell}(s,\vec{i}, \vec{y},\theta_1, \theta_2),
\label{Eq:EM_j_formula}
\end{align}
for $j \in \{1,\dots,J\}$.

In this formulation, one important fact is that for fixed ${\vec{\zeta}}$, $\gamma\la 1,\mu^{{\vec{\zeta}}, j}_{s-}\ra$ is finite by assumption, which provides exchangeability of the sum and Lebesgue integral.

\subsection{Taking expectations}\label{SS:TakingExpectations}

By taking expectation on (\ref{Eq:EM_j_formula}) we obtain
\begin{align}
&\avg{\la f,\mu^{{\vec{\zeta}}, j}_{t}\ra}
=\avg{\la f,\mu^{{\vec{\zeta}}, j}_{0}\ra} + \avg{\frac{1}{\gamma}\int_{0}^{t}\sum_{i=1}^{\gamma\la 1, \mu_{s-}^{{\vec{\zeta}}, j}\ra} D_{j}\frac{\partial^{2} f}{\partial Q^{2}}(H^i(\gamma\mu_{s-}^{{\vec{\zeta}}, j}))ds}+\sum_{\ell = 1}^L \avg{ \frac{1}{\gamma}\int_{0}^{t} \int_{\mathbb{Y}^{(\ell)}} \sum_{\vec{i}\in\Omega^{(\ell)}({\gamma\mu^{\vec{\zeta}}_{s-}})}   \nonumber\\
&
\quad \left( - \sum_{r = 1}^{\alpha_{\ell j}} f(H_Q^{i_r^{(j)}}(\gamma\mu^{{\vec{\zeta}},j}_{s-})) + \sum_{r = 1}^{\beta_{\ell j}} f(y_r^{(j)}) \right)\times K_\ell^{\gamma}\left(\mathcal{P}^{(\ell)}(\gamma\mu_{s-}^{\vec{\zeta}}, \vec{i}) \right)   \times  m_\ell\left(\vec{y} |\mathcal{P}^{(\ell)}(\gamma\mu_{s-}^{\vec{\zeta}}, \vec{i}) \right)\, d \vec{y}\, ds } ,\nonumber\\
&
=\avg{\la f,\mu^{{\vec{\zeta}}, j}_{0}\ra} + \avg{\int_{0}^{t} \int_{\R^d} \frac{1}{\gamma} \sum_{i=1}^{\gamma\la 1, \mu_{s-}^{{\vec{\zeta}}, j}\ra} D_{j}\frac{\partial^{2} f}{\partial Q^{2}}(x)\delta_{H_Q^i(\gamma\mu_{s-}^{{\vec{\zeta}}, j})}(dx)ds}\nonumber\\
&
\quad- \sum_{\ell = 1}^L \avg{\int_{0}^{t} \frac{1}{\gamma} \sum_{\vec{i}\in\Omega^{(\ell)}({\gamma\mu^{\vec{\zeta}}_{s-}})} \left(\sum_{r = 1}^{\alpha_{\ell j}} f(H_Q^{i_r^{(j)}}(\gamma\mu^{{\vec{\zeta}}, j}_{s-})) \right) K_\ell^{\gamma}\left(\mathcal{P}^{(\ell)}(\gamma\mu_{s-}^{\vec{\zeta}}, \vec{i})\right) \left(\int_{\mathbb{Y}^{(\ell)}}   m^\eta_\ell\left(\vec{y} |\mathcal{P}^{(\ell)}(\gamma\mu_{s-}^{\vec{\zeta}}, \vec{i})\right)\, d \vec{y} \right)\, ds } ,\nonumber\\
&
\quad+\sum_{\ell = 1}^L \avg{\int_{0}^{t} \frac{1}{\gamma}  \sum_{\vec{i}\in\Omega^{(\ell)}(\gamma\mu^{\vec{\zeta}}_{s-})}   K_\ell^{\gamma}\left(\mathcal{P}^{(\ell)}(\gamma\mu_{s-}^{\vec{\zeta}}, \vec{i})\right)  \left(\int_{\mathbb{Y}^{(\ell)}}     \left( \sum_{r = 1}^{\beta_{\ell j}} f(y_r^{(j)}) \right) m^\eta_\ell\left(\vec{y} |\mathcal{P}^{(\ell)}(\gamma\mu_{s-}^{\vec{\zeta}}, \vec{i})\right)\, d \vec{y} \right)\, ds }  ,\nonumber\\
&
=\avg{\la f,\mu^{{\vec{\zeta}}, j}_{0}\ra} + \avg{\int_{0}^{t} \la (\mathcal{L}_j f)(x), \mu_{s-}^{{\vec{\zeta}}, j}(dx) \ra ds}\nonumber\\
&
\quad- \sum_{\ell = 1}^L \avg{\int_{0}^{t} \frac{1}{\gamma} \int_{\mathbb{X}^{(\ell)}} \sum_{\vec{i}\in\Omega^{(\ell)}({\gamma\mu^{\vec{\zeta}}_{s-}})}   \left(\sum_{r = 1}^{\alpha_{\ell j}} f(x_r^{(j)}) \right)  K_\ell^{\gamma}\left(\vec{x}\right)\, \delta_{\mathcal{P}^{(\ell)}(\gamma\mu_{s-}^{\vec{\zeta}}, \vec{i})}(d \vec{x}) \, ds } ,\nonumber\\
&
\quad+\sum_{\ell = 1}^L \avg{\int_{0}^{t} \frac{1}{\gamma} \int_{\mathbb{X}^{(\ell)}}   \sum_{\vec{i}\in\Omega^{(\ell)}({\gamma\mu^{\vec{\zeta}}_{s-}})}    K_\ell^{\gamma}\left(\vec{x}\right) \left(\int_{\mathbb{Y}^{(\ell)}}    \left( \sum_{r = 1}^{\beta_{\ell j}} f(y_r^{(j)}) \right) m^\eta_\ell\left(\vec{y} \, |\, \vec{x} \right)\, d \vec{y} \right)\,\delta_{\mathcal{P}^{(\ell)}(\gamma\mu_{s-}^{\vec{\zeta}}, \vec{i})}(d \vec{x}) \, ds } ,\nonumber\\
&
=\avg{\la f,\mu^{{\vec{\zeta}}, j}_{0}\ra} + \avg{\int_{0}^{t} \la (\mathcal{L}_j f)(x), \mu_{s-}^{{\vec{\zeta}}, j}(dx) \ra ds}\nonumber\\
&
\quad- \sum_{\ell = 1}^L \avg{\int_{0}^{t} \int_{\tilde{\mathbb{X}}^{(\ell)}}    \frac{1}{\vec{\alpha}^{(\ell)}!}  \left(\sum_{r = 1}^{\alpha_{\ell j}} f(x_r^{(j)}) \right)  K_\ell\left(\vec{x}\right) \, \lambda^{(\ell)}[\mu_{s-}^{{\vec{\zeta}}}](d\vec{x}) \, ds } ,\nonumber\\
&
\quad+\sum_{\ell = 1}^L \avg{\int_{0}^{t} \int_{\tilde{\mathbb{X}}^{(\ell)}}     \frac{1}{\vec{\alpha}^{(\ell)}!}   K_\ell\left(\vec{x}\right) \left(\int_{\mathbb{Y}^{(\ell)}}    \left( \sum_{r = 1}^{\beta_{\ell j}} f(y_r^{(j)}) \right) m^\eta_\ell\left(\vec{y} \, |\, \vec{x} \right)\, d \vec{y} \right)\,\lambda^{(\ell)}[\mu_{s-}^{{\vec{\zeta}}}](d\vec{x}) \, ds } ,\nonumber\\
&
=\avg{\la f,\mu^{{\vec{\zeta}}, j}_{0}\ra} + \avg{\int_{0}^{t} \la (\mathcal{L}_j f)(x), \mu_{s-}^{{\vec{\zeta}}, j}(dx) \ra ds}\nonumber\\
&
\quad+\sum_{\ell = 1}^L \avg{\int_{0}^{t} \int_{\tilde{\mathbb{X}}^{(\ell)}}     \frac{1}{\vec{\alpha}^{(\ell)}!}   K_\ell\left(\vec{x}\right) \left(\int_{\mathbb{Y}^{(\ell)}}    \left( \sum_{r = 1}^{\beta_{\ell j}} f(y_r^{(j)}) \right) m^\eta_\ell\left(\vec{y} \, |\, \vec{x} \right)\, d \vec{y} - \sum_{r = 1}^{\alpha_{\ell j}} f(x_r^{(j)}) \right)\,\lambda^{(\ell)}[\mu_{s-}^{{\vec{\zeta}}}](d\vec{x}) \, ds } ,\nonumber\\
\label{Eq:EM_EXP_j_formula}
\end{align}
where in the third equality we use the assumption that $ \int_{\mathbb{Y}^{(\ell)}}   m^\eta_\ell(\vec{y} |\vec{x} )\, d \vec{y} = 1$. For the second to last equality we switch integrals of the form $\int_{\mathbb{X}^{(\ell)}} \sum_{\vec{i}\in\Omega^{(\ell)}({\gamma\mu^{\vec{\zeta}}_{s-}})} \cdots \delta_{\mathcal{P}^{(\ell)}(\gamma\mu_{s-}^{\vec{\zeta}}, \vec{i})}(d \vec{x})$ to $ \int_{\tilde{\mathbb{X}}^{(\ell)}}     \frac{1}{\vec{\alpha}^{(\ell)}!}  \cdots \lambda^{(\ell)}[\mu_{s-}^{{\vec{\zeta}}}](d\vec{x})$ using the definition of $\mu^{{\vec{\zeta}}, j}_s(dx)$ and $\lambda^{(\ell)}[\,\cdot\,]$ (see Definition~\ref{def:lambda}), and removing probability zero sets where two particles with the same type are simultaneously located at the same spatial location (see Definition~\ref{def:offDiagSpace}). Note, by definition the allowable reactant index sampling space $\Omega^{(\ell)}$ (see Definition \ref{def:effSampllingSpace}) orders indices for particles of the same species. In converting from integrals involving the positions of individual particles (i.e. $\delta_{\mathcal{P}^{(\ell)}(\gamma\mu_{s-}^{\vec{\zeta}}, \vec{i})}(d \vec{x})$) to integrals involving product measures ($\lambda^{(\ell)}[\mu_{s-}^{{\vec{\zeta}}}](d\vec{x})$) we need to remove the "diagonal" indices by means of integrating on $\tilde{\mathbb{X}}^{(\ell)}$ (see Definition~\ref{def:offDiagSpace}) and normalizing by the total number of index orderings, $(\vec{\alpha}^{(\ell)}!)$.


\subsection{Identification of the Limit}\label{SS:LimitIdentification}
Inspired by Eq (\ref{Eq:EM_EXP_j_formula}), we expect that if the weak limit, as ${\vec{\zeta}}$ goes to zero, of the marginal distribution vector $\vec{\mu}^{\vec{\zeta}}_t :=( \mu_t^{{\vec{\zeta}}, 1} , \mu_t^{{\vec{\zeta}}, 2}  \cdots, \mu_t^{{\vec{\zeta}}, J})$  exists and is unique, then it will satisfy the analogous equation. For this purpose, denote $\vec{\xi}_t := (\xi_t^1, \xi_t^2, \cdots, \xi_t^J)$ to be  the corresponding limiting particle distribution on $\R^{J\times d}$ and $\xi_t = \sum_{j = 1}^J \xi_t^j\delta_{S_j}$ to be  the corresponding limiting particle distribution on $\hat{P}$. Then  for each $1\leq j \leq J$, the following must be satisfied,
\begin{align}
\la f,\xi^j_{t}\ra
&
=\la f,\xi^{ j}_{0}\ra + \int_{0}^{t} \la (\mathcal{L}_j f)(x), \xi_{s}^{ j}(dx) \ra ds\nonumber\\
&
\quad+\sum_{\ell = 1}^L \int_{0}^{t} \int_{\tilde{\mathbb{X}}^{(\ell)}}     \frac{1}{\vec{\alpha}^{(\ell)}!}   K_\ell\left(\vec{x}\right) \left(\int_{\mathbb{Y}^{(\ell)}}    \left( \sum_{r = 1}^{\beta_{\ell j}} f(y_r^{(j)}) \right) m_\ell\left(\vec{y} \, |\, \vec{x} \right)\, d \vec{y} - \sum_{r = 1}^{\alpha_{\ell j}} f(x_r^{(j)}) \right)\,\lambda^{(\ell)}[\xi_{s}](d\vec{x}) \, ds.
\label{Eq:Limit_EM_formula}
\end{align}
Existence of the limit is shown in the tightness Section \ref{S:Tightness}, while uniqueness is shown in Section \ref{S:Uniqueness}.

Let $\mathcal{S}$ be the collection of elements $\Phi$ in $B(\otimes_{j=1}^{J}M_{F}(\R^d))$  i.e., bounded functionals) of the form
\begin{equation}\label{E:form}
\Phi(\vec{\mu}) = \varphi\left(\la f_1, \vec{\mu}\ra ,\la f_2, \vec{\mu}\ra \dots \la f_M,\vec{\mu}\ra \right)
 \end{equation}
for some $M\in \N$, some $\varphi\in C^\infty(\R^{J \times M})$, $\la f_m,\vec{\mu}\ra=\left(\la f_{1,m},\mu^{1}\ra,\cdots,\la f_{J,m},\mu^{J}\ra\right)$ where each $\{f_{j,m}\}\in  C^2_b(\R^d)$ for $j=1,\cdots, J$ and $m=1,\cdots, M$.  Then $\mathcal{S}$ separates points in $\otimes_{j=1}^{J} M_{F}(\R^d)$ (see Chapter 3.4 of \cite{EthierKurtz} and Proposition 3.3 of \cite{Capponi}).  As long as the limiting process exists and is unique, to identify the limit, it thus suffices to show convergence of the martingale
problem for functions of the form \eqref{E:form}.

For $\Phi\in \mathcal{S}$ of the form \eqref{E:form}, $\vec{\mu} :=( \mu^1 , \mu^2 \cdots, \mu^J)\in \otimes_{j=1}^{J}M_{F}(\R^d)$  and $\mu = \sum_{j=1}^J \mu^j\delta_{S_j} \in M_{F}(\hat{P})$ with each $\mu^j\in M_{F}(\R^d)$,  define
\begin{align}
(\genA\Phi)(\vec{\mu}) \Def & \sum_{m=1}^M \sum_{j=1}^{J} \frac{\partial \varphi}{\partial x_{(m-1)*J + j}}\left(\la f_1,\vec{\mu} \ra,\la f_2, \vec{\mu} \ra\dots \la f_M, \vec{\mu}\ra \right) \Bigg\{ \la \genL_\ell f_{j,m},\mu^j \ra   \nonumber\\
& \quad
+ \sum_{\ell = 1}^L \int_{\tilde{\mathbb{X}}^{(\ell)}}     \frac{1}{\vec{\alpha}^{(\ell)}!}   K_\ell\left(\vec{x}\right) \left(\int_{\mathbb{Y}^{(\ell)}}    \left( \sum_{r = 1}^{\beta_{\ell j}} f_{j,m}(y_r^{(j)}) \right) m_\ell\left(\vec{y} \, |\, \vec{x} \right)\, d \vec{y} - \sum_{r = 1}^{\alpha_{\ell j}} f_{j,m}(x_r^{(j)}) \right)\,\lambda^{(\ell)}[\mu](d\vec{x}) \,  \Bigg\} \nonumber \\
 \label{E:limgenA}
 \end{align}
We claim that $\genA$, which is the generator of the system described by (\ref{Eq:Limit_EM_formula}) for $1\leq j \leq J$, will be the generator of the limiting martingale problem.

\begin{lemma}[Weak Convergence]\label{L:wconv} For any $\Phi\in \mathcal{S}$ and $0\le r_1\le r_2\dots \leq r_W=s<t<T$ and $\{\psi_w\}_{w=1}^W\subset B(\otimes_{j=1}^{J}M_{F}(\R^d))$, we have that
\begin{equation}
\lim_{{\vec{\zeta}}\to 0}\avg{\lb \Phi(\vec{\mu}^{{\vec{\zeta}}}_t)-\Phi(\vec{\mu}^{{\vec{\zeta}}}_s)-\int_{s}^t (\genA\Phi)(\vec{\mu}^{{\vec{\zeta}}}_r)dr\rb \prod_{w=1}^W \psi_w(\vec{\mu}^{{\vec{\zeta}}}_{r_w})}=0.\label{Eq:Identification}
\end{equation}
\end{lemma}

\begin{proof}
For each $j=1,\cdots J$, we can rewrite Eq (\ref{Eq:EM_j_formula}) as
\begin{align*}
\la f,\mu^{{\vec{\zeta}}, j }_{t}\ra &=\la f,\mu^{{\vec{\zeta}}, j}_{0}\ra + M_t^{f, j} + A_t^{f, j},
\end{align*}
where
\begin{align}
A_t^{f, j}
&
=  \int_{0}^{t} \la (\mathcal{L}_j f)(x), \mu_{s-}^{{\vec{\zeta}}, j}(dx) \ra ds\nonumber\\
&
\quad+\sum_{\ell = 1}^L \int_{0}^{t} \int_{\tilde{\mathbb{X}}^{(\ell)}}     \frac{1}{\vec{\alpha}^{(\ell)}!}   K_\ell\left(\vec{x}\right) \left(\int_{\mathbb{Y}^{(\ell)}}    \left( \sum_{r = 1}^{\beta_{\ell j}} f(y_r^{(j)}) \right) m^\eta_\ell\left(\vec{y} \, |\, \vec{x} \right)\, d \vec{y} - \sum_{r = 1}^{\alpha_{\ell j}} f(x_r^{(j)}) \right)\,\lambda^{(\ell)}[\mu^{\vec{\zeta}}_{s-}](d\vec{x}) \, ds.
\label{Eq:Mean_f_j}
\end{align}
and
\begin{align}
M_t^{f, j} &=   \frac{1}{\gamma}\sum_{i\geq 1}\int_{0}^{t}1_{\{i\leq \gamma \la 1, \mu_{s-}^{{\vec{\zeta}}, j}\ra\}}\sqrt{2D_{j}}\frac{\partial f}{\partial Q}(H^i(\gamma\mu_{s-}^{{\vec{\zeta}}, j}))dW^{i}_{s}\nonumber\\
&
\quad+\sum_{\ell = 1}^L \int_{0}^{t}\int_{\mathbb{I}^{(\ell)}} \int_{\mathbb{Y}^{(\ell)}}\int_{\mathbb{R}^{2}_{+}}\left(\la f, \mu^{{\vec{\zeta}}, j}_{s-}-\frac{1}{\gamma} \sum_{r = 1}^{\alpha_{\ell j}} \delta_{H_Q^{i_r^{(j)}}(\gamma\mu^{{\vec{\zeta}},j}_{s-})} + \frac{1}{\gamma}\sum_{r = 1}^{\beta_{\ell j}} \delta_{y_r^{(j)}} \ra  - \la f,\mu^{{\vec{\zeta}}, j}_{s-}\ra\right)\nonumber\\
&
\qquad
\times 1_{\{\vec{i} \in \Omega^{(\ell)}(\gamma\mu^{\vec{\zeta}}_{s-})\}} \times 1_{ \{ \theta_1 \leq K_\ell^{\gamma}\left(\mathcal{P}^{(\ell)}(\gamma\mu_{s-}^{\vec{\zeta}}, \vec{i}) \right) \}}  \times 1_{ \{ \theta_2\leq m^\eta_\ell\left(\vec{y} \,  | \, \mathcal{P}^{(\ell)}(\gamma\mu_{s-}^{\vec{\zeta}}, \vec{i}) \right)\}}   d\tilde{N}_{\ell}(s,\vec{i}, \vec{y},\theta_1, \theta_2),
\label{Eq:Martingale_f_j}
\end{align}
is a square integrable martingale (See Proposition 2.4 in \cite{NW:2014}) with quadratic variation
\begin{align}
\la M^{f, j}\ra_t &=  \frac{1}{\gamma^2}\int_{0}^{t}\sum_{i = 1}^{\gamma\la 1,\mu^{{\vec{\zeta}}, j}_{s-}\ra}\left(\sqrt{2D_{j}}\frac{\partial f}{\partial Q}(H^{i}(\gamma\mu^{{\vec{\zeta}},j}_{s-}))\right)^2ds \nonumber\\
&\quad
+\sum_{\ell = 1}^L \int_{0}^{t}\int_{\mathbb{Y}^{(\ell)}} \frac{1}{\gamma^2} \sum_{\{\vec{i} \in \Omega^{(\ell)}(\gamma\mu^{\vec{\zeta}}_{s-})\}}   \left( - \sum_{r = 1}^{\alpha_{\ell j}} f(H_Q^{i_r^{j}}(\gamma\mu^{{\vec{\zeta}},j}_{s-})) + \sum_{r = 1}^{\beta_{\ell j}} f(y_r^{j}) \right)^2\nonumber\\
&\qquad\quad
\times  K_\ell^{\gamma}\left(\mathcal{P}^{(\ell)}(\gamma\mu_{s-}^{\vec{\zeta}}, \vec{i}) \right)  \times  m^\eta_\ell\left(\vec{y} \,  | \, \mathcal{P}^{(\ell)}(\gamma\mu_{s-}^{\vec{\zeta}}, \vec{i}) \right)\, d \vec{y} \, ds  ,\nonumber\\
&
\leq  \frac{1}{\gamma}\int_{0}^{t} \la 2D_{j} \left(\frac{\partial f}{\partial x}\right)^2, \mu_{s-}^{{\vec{\zeta}}, j}\ra ds \nonumber\\
& \quad
+\sum_{\ell = 1}^L \int_{0}^{t}\int_{\mathbb{Y}^{(\ell)}} \frac{1}{\gamma^2} \sum_{\{\vec{i} \in \Omega^{(\ell)}(\gamma\mu^{\vec{\zeta}}_{s-})\}}   ||f||_{C_b(\R^d)}^2(\alpha_{\ell j} + \beta_{\ell j})^2 K_\ell^{\gamma}\left(\mathcal{P}^{(\ell)}(\gamma\mu_{s-}^{\vec{\zeta}}, \vec{i}) \right)  m^\eta_\ell\left(\vec{y} \,  | \, \mathcal{P}^{(\ell)}(\gamma\mu_{s-}^{\vec{\zeta}}, \vec{i}) \right)\, d \vec{y} \, ds  ,\nonumber\\
&
\leq \frac{C(D_j)t||f||^2_{C_b^1(\mathbb{R}^d)}}{\gamma}  C(\mu) +  \sum_{\ell = 1}^L  ||f||_{C_b(\R^d)}^2(\alpha_{\ell j} + \beta_{\ell j})^2  \int_{0}^{t}\frac{1}{\gamma^2} \sum_{\{\vec{i} \in \Omega^{(\ell)}(\gamma\mu^{\vec{\zeta}}_{s-})\}}  K_\ell^{\gamma}\left(\mathcal{P}^{(\ell)}(\gamma\mu_{s-}^{\vec{\zeta}}, \vec{i}) \right)   \, ds  ,\nonumber\\
&
\leq \frac{C(D_j)t||f||^2_{C_b^1(\mathbb{R}^d)}}{\gamma}  C(\mu) +  C(K)4^2||f||^2_{C_b(\R^d)}   \sum_{\ell = 1}^L \int_{0}^{t}\frac{1}{\gamma^2} \sum_{\{\vec{i} \in \Omega^{(\ell)}(\gamma\mu^{\vec{\zeta}}_{s-})\}}  \gamma^{1-|\alpha^{(\ell)}|}   \, ds  ,\nonumber\\
&
\leq \frac{C(D_j)t||f||^2_{C_b^1(\mathbb{R}^d)}}{\gamma}  C(\mu) +  \frac{C(K)tL||f||^2_{C_b(\R^d)}}{\gamma}  C(\mu).
\label{Eq:Quadratic_Variation_f_j}
\end{align}
The quadratic variation is therefore uniformly bounded and goes to 0 as ${\vec{\zeta}}\to 0$ ($\gamma\rightarrow\infty$) since $f$ and its partial derivatives are uniformly bounded.

Now define $M_t^{f, j} = \mathscr{C}_t^{f, j} + \mathscr{D}_t^{f, j} $, where
\begin{align}
\mathscr{C}_t^{f, j} &=   \frac{1}{\gamma}\sum_{i\geq 1}\int_{0}^{t}1_{\{i\leq \gamma \la 1, \mu_{s-}^{{\vec{\zeta}}, j}\ra\}}\sqrt{2D_{j}}\frac{\partial f}{\partial Q}(H^i(\gamma\mu_{s-}^{{\vec{\zeta}}, j}))dW^{i}_{s}\nonumber\\
\label{Eq:MartingaleC_f_j}
\end{align}
is the continuous martingale part and
\begin{align}
\mathscr{D}_t^{f, j} &=  \sum_{\ell = 1}^L \int_{0}^{t}\int_{\mathbb{I}^{(\ell)}} \int_{\mathbb{Y}^{(\ell)}}\int_{\mathbb{R}^{2}_{+}}\left(\la f, \mu^{{\vec{\zeta}}, j}_{s-}-\frac{1}{\gamma} \sum_{r = 1}^{\alpha_{\ell j}} \delta_{H_Q^{i_r^{(j)}}(\gamma\mu^{{\vec{\zeta}},j}_{s-})} + \frac{1}{\gamma}\sum_{r = 1}^{\beta_{\ell j}} \delta_{y_r^{(j)}} \ra  - \la f,\mu^{{\vec{\zeta}}, j}_{s-}\ra\right)\nonumber\\
&
\qquad
\times 1_{\{\vec{i} \in \Omega^{(\ell)}(\gamma\mu^{\vec{\zeta}}_{s-})\}} \times 1_{ \{ \theta_1\leq K_\ell^{\gamma}\left(\mathcal{P}^{(\ell)}(\gamma\mu_{s-}^{\vec{\zeta}}, \vec{i}) \right)\}} \times 1_{ \{ \theta_2\leq  m^\eta_\ell\left(\vec{y} \,  | \, \mathcal{P}^{(\ell)}(\gamma\mu_{s-}^{\vec{\zeta}}, \vec{i}) \right)\}} d\tilde{N}_{\ell}(s,\vec{i}, \vec{y},\theta_1, \theta_2),
\label{Eq:MartingaleD_f_j}
\end{align}
is the martingale part coming from the stochastic integral with respect to the Poisson point processes. Here, for simplicity of notation, we let
\begin{align}
g^{\ell, f, \mu^{{\vec{\zeta}},j}}(s, \vec{i}, \vec{y},\theta_1, \theta_2)&= \left(\la f, \mu^{{\vec{\zeta}}, j}_{s-}-\frac{1}{\gamma} \sum_{r = 1}^{\alpha_{\ell j}} \delta_{H_Q^{i_r^{(j)}}(\gamma\mu^{{\vec{\zeta}},j}_{s-})} + \frac{1}{\gamma}\sum_{r = 1}^{\beta_{\ell j}} \delta_{y_r^{(j)}} \ra  - \la f,\mu^{{\vec{\zeta}}, j}_{s-}\ra\right)\nonumber\\
&
\qquad
\times 1_{\{\vec{i} \in \Omega^{(\ell)}(\gamma\mu^{\vec{\zeta}}_{s-})\}} \times 1_{ \{ \theta_1 \leq K_\ell^{\gamma}\left(\mathcal{P}^{(\ell)}(\gamma\mu_{s-}^{\vec{\zeta}}, \vec{i}) \right) \}} \times 1_{ \{ \theta_2 \leq  m^\eta_\ell\left(\vec{y} \,  | \, \mathcal{P}^{(\ell)}(\gamma\mu_{s-}^{\vec{\zeta}}, \vec{i}) \right)\}} \nonumber\\
&
 = \frac{1}{\gamma}\left(- \sum_{r = 1}^{\alpha_{\ell j}} f\left(H_Q^{i_r^{(j)}}(\gamma\mu^{{\vec{\zeta}},j}_{s-})\right) + \sum_{r = 1}^{\beta_{\ell j}} f\left(y_r^{(j)}\right)\right) \nonumber\\
&\qquad \times 1_{\{\vec{i} \in \Omega^{(\ell)}(\gamma\mu^{\vec{\zeta}}_{s-})\}}\times 1_{ \{ \theta_1 \leq K_\ell^{\gamma}\left(\mathcal{P}^{(\ell)}(\gamma\mu_{s-}^{\vec{\zeta}}, \vec{i}) \right) \}}\times 1_{ \{ \theta_2 \leq  m^\eta_\ell\left(\vec{y} \,  | \, \mathcal{P}^{(\ell)}(\gamma\mu_{s-}^{\vec{\zeta}}, \vec{i}) \right)\}},
\label{Eq:MartingaleD_integrand}
\end{align}
which represents the jumps and is uniformly bounded by $\mathcal{O}(\frac{1}{\gamma})$.  With some abuse of notation we shall write $\vec{g}^{\ell, f, \vec{\mu}^{{\vec{\zeta}}}}$ for the vector $(g^{\ell, f, \mu^{{\vec{\zeta}},1}},\cdots,g^{\ell, f, \mu^{{\vec{\zeta}},J}})$. Then (\ref{Eq:MartingaleD_f_j}) becomes \
\[
\mathscr{D}_t^{f, j} =  \sum_{\ell = 1}^L \int_{0}^{t}\int_{\mathbb{I}^{(\ell)}}\int_{\mathbb{Y}^{(\ell)}}\int_{\mathbb{R}^{2}_{+}} g^{\ell, f, \mu^{{\vec{\zeta}},j}}(s, \vec{i}, \vec{y},\theta_1, \theta_2) d\tilde{N}_{\ell}(s,\vec{i}, \vec{y},\theta_1, \theta_2).
\]

Now we apply It\^{o}'s formula (See Theorem 5.1 in \cite{NW:2014}) to $\Phi(\vec{\mu}^{{\vec{\zeta}}}_t)$. We obtain for $\theta=(\theta_{1},\theta_{2})$,
\begin{align}
&  \Phi(\vec{\mu}^{{\vec{\zeta}}}_t)-\Phi(\vec{\mu}^{{\vec{\zeta}}}_s)-\int_{s}^t (\genA\Phi)(\vec{\mu}^{{\vec{\zeta}}}_r)dr
=  \int_{s}^t\sum_{m=1}^M  \sum_{j=1}^{J} \frac{\partial \varphi}{\partial x_{(m-1)*J + j}}\left(\la f_1,\vec{\mu}^{{\vec{\zeta}}}_r\ra , \la f_2,\vec{\mu}^{{\vec{\zeta}}}_r\ra \dots \la f_M,\vec{\mu}^{{\vec{\zeta}}}_r\ra \right) d\mathscr{C}_r^{f_{j,m}, j} \nonumber\\
&
\quad + \frac{1}{2} \int_{s}^t\sum_{m=1}^M \sum_{j=1}^{J} \frac{\partial^2 \varphi}{\partial x_{(m-1)*J + j}^2} \left(\la f_1, \vec{\mu}^{{\vec{\zeta}}}_r\ra,\la f_2, \vec{\mu}^{{\vec{\zeta}}}_r\ra \dots \la f_M, \vec{\mu}^{{\vec{\zeta}}}_r\ra \right) d\la \mathscr{C}^{f_{j,m}, j}\ra_r \nonumber\\
&
\quad +  \sum_{\ell = 1}^L\int_{s}^t  \int_{\mathbb{I}^{(\ell)}}\int_{\mathbb{Y}^{(\ell)}}\int_{\mathbb{R}^{2}_{+}}\left(\varphi\left(\la f_1, \vec{\mu}^{{\vec{\zeta}}}_r\ra + \vec{g}^{\ell, f_1, \vec{\mu}^{{\vec{\zeta}}}}(r, \vec{i}, \vec{y},\theta),  \dots , \la f_M,\vec{\mu}^{{\vec{\zeta}}}_r\ra+ \vec{g}^{\ell, f_M, \vec{\mu}^{{\vec{\zeta}}}}(r, \vec{i}, \vec{y},\theta) \right) \right. \nonumber\\
&
\qquad -\left.   \varphi\left(\la f_1,\vec{\mu}^{{\vec{\zeta}}}_r\ra , \la f_2, \vec{\mu}^{{\vec{\zeta}}}_r\ra \dots \la f_M, \vec{\mu}^{{\vec{\zeta}}}_r\ra \right) \right)d\tilde{N}_{\ell}(r,\vec{i}, \vec{y},\theta)\nonumber\\
&
\quad +  \sum_{\ell = 1}^L\int_{s}^t\int_{\mathbb{I}^{(\ell)}}\int_{\mathbb{Y}^{(\ell)}}\int_{\mathbb{R}^{2}_{+}}  \left(\varphi\left(\la f_1,\vec{\mu}^{{\vec{\zeta}}}_r\ra + \vec{g}^{\ell, f_1, \vec{\mu}^{{\vec{\zeta}}}}(s, \vec{i}, \vec{y},\theta),  \dots , \la f_M,\vec{\mu}^{{\vec{\zeta}}}_r\ra+ \vec{g}^{\ell, f_M, \vec{\mu}^{{\vec{\zeta}}}}(s, \vec{i}, \vec{y},\theta) \right) \right. \nonumber\\
&
\qquad -\left.  \varphi\left(\la f_1,\vec{\mu}^{{\vec{\zeta}}}_r\ra , \la f_2, \vec{\mu}^{{\vec{\zeta}}}_r\ra \dots \la f_M, \vec{\mu}^{{\vec{\zeta}}}_r\ra \right) - \sum_{m = 1}^M \sum_{j=1}^{J}g^{\ell, f_{j,m}, \mu^{{\vec{\zeta}},j}}(s, \vec{i}, \vec{y},\theta)  \frac{\partial \varphi}{\partial x_{(m-1)J + j}}\left(\la f_1,\vec{\mu}^{{\vec{\zeta}}}_r\ra , \dots, \la f_M,\vec{\mu}^{{\vec{\zeta}}}_r\ra \right)\right)\nonumber\\
&\qquad\qquad  d\bar{N}_{\ell}(s,\vec{i}, \vec{y},\theta)\nonumber\\
&
\quad +  \sum_{m=1}^M \sum_{j=1}^{J} \sum_{\ell = 1}^L \int_{0}^{t} \frac{\partial \varphi}{\partial x_{(m-1)J+j}}\left(\la f_1,\vec{\mu}^{{\vec{\zeta}}}_r\ra , \la f_2, \vec{\mu}^{{\vec{\zeta}}}_r\ra \dots \la f_M,\vec{\mu}^{{\vec{\zeta}}}_r\ra \right)\nonumber\\
&\qquad \times\int_{\tilde{\mathbb{X}}^{(\ell)}}     \frac{1}{\vec{\alpha}^{(\ell)}!}   K_\ell\left(\vec{x}\right) \left(\int_{\mathbb{Y}^{(\ell)}}    \left( \sum_{r = 1}^{\beta_{\ell j}} f_{j,m}(y_r^{(j)}) \right) \left( m^\eta_\ell\left(\vec{y} \, |\, \vec{x} \right) - m_\ell\left(\vec{y} \, |\, \vec{x} \right)\right)\, d \vec{y}  \right)\,\lambda^{(\ell)}[\mu^{\vec{\zeta}}_{s-}](d\vec{x}) ds\nonumber\\
&=\sum_{\kappa=1}^{5}\Lambda_{\kappa}^{\vec{\zeta}}(t).\nonumber\\\label{Eq:ItoFormulaIdentification}
 \end{align}

We now use the Skorokhod representation theorem (Theorem 1.8 in \cite{EthierKurtz}) which, for the purposes of identifying the limit and proving (\ref{Eq:Identification}), allows us to assume that the aforementioned claimed convergence of $\vec{\mu}^{\vec{\zeta}}_t = ( \mu_t^{{\vec{\zeta}}, 1} , \mu_t^{{\vec{\zeta}}, 2}  \cdots, \mu_t^{{\vec{\zeta}}, J})$ holds with probability one in the topology of weak convergence of measures. The Skorokhod representation theorem involves the introduction of another probability space, but we ignore this distinction in the notation.
To show~\eqref{Eq:Identification}, it is then sufficient to prove that the left hand side of (\ref{Eq:ItoFormulaIdentification}) goes to zero in probability. With this goal in mind we proceed with proving convergence in probability to zero for $\Lambda_{\kappa}^{\vec{\zeta}}(t)$ for $\kappa=1,\cdots,5$.

By Lemma \ref{lem:etaConv}, we immediately have that
\[
\lim_{\vec{\zeta}\rightarrow 0}\sup_{t\in[0,T]}\BE|\Lambda_{5}^{\vec{\zeta}}(t)|=0.
\]
In addition, notice that $\Lambda_{1}^{\vec{\zeta}}(t)$ and $\Lambda_{3}^{\vec{\zeta}}(t)$ are square integrable martingales. In fact, by (\ref{Eq:Quadratic_Variation_f_j}), (\ref{Eq:MartingaleD_integrand}) and the fact that the jump size is uniformly bounded by $\mathcal{O}(\frac{1}{\gamma})$, we have that
\[
\lim_{\vec{\zeta}\rightarrow 0}\sup_{t\in[0,T]}\BE|\Lambda_{1}^{\vec{\zeta}}(t)+\Lambda_{3}^{\vec{\zeta}}(t)|^{2}=0.
\]
For similar reasons, we also have by (\ref{Eq:MartingaleC_f_j}) that
\[
\lim_{\vec{\zeta}\rightarrow 0}\sup_{t\in[0,T]}\BE|\Lambda_{2}^{\vec{\zeta}}(t)|=0,
\]
and by (\ref{Eq:MartingaleD_integrand}) that
\[
\lim_{\vec{\zeta}\rightarrow 0}\sup_{t\in[0,T]}\BE|\Lambda_{4}^{\vec{\zeta}}(t)|=0.
\]
We then have that the left hand side of (\ref{Eq:ItoFormulaIdentification}) goes to zero in probability, concluding the proof of the lemma.

\end{proof}


\subsection{Tightness}\label{S:Tightness}

 Recall that $M_F(\mathbb{R}^d)$ denotes the space of finite measures endowed with the weak topology, and denote by $M_F'(\mathbb{R}^d)$ the space of finite measures endowed with the vague topology. In this section, we prove tightness of the measure-valued processes $\{\mu^{{\vec{\zeta}}, j}_t\}_{t\in [0, T]}, j = 1, 2, \cdots, J$  on $\mathbb{D}_{M_F(\mathbb{R}^d)}[0, T]$, the space of cadlag paths with values in $M_F(\mathbb{R}^d)$ endowed with Skorokhod topology. Towards this aim, we first show that the processes $\{\mu^{{\vec{\zeta}}, j}_t\}_{t\in [0, T]}, j = 1, \cdots, J, $ are tight on $\mathbb{D}_{M_F'(\mathbb{R}^d)}[0, T]$.


\subsubsection{Tightness in $\mathbb{D}_{M_F'(\mathbb{R}^d)}[0, T]$}
It suffices to show that the real-valued processes $\{\la f, \mu^{{\vec{\zeta}}, j}_{t}\ra\}$, $j = 1, \cdots, J$, for any test function $f(x) \in C^2_0(\mathbb{R}^d)$, which is dense in $C_0(\mathbb{R}^d)$, are tight in $\mathbb{D}_{\mathbb{R}}[0, T]$, see \cite{Roelly:1986}. In establishing this we use the Rebolledo Criterion \cite{AM:1986} (Lemma \ref{lem:RebolledoCriterion}) and the Aldous Condition \cite{Aldous:1978} (Lemma \ref{lem:AldousCondition}).
\begin{lemma}\label{lem:tightStoppingTimes}
For any $T > 0$ and $\delta > 0$ , there exists constants $C$ and $C'$ such that for any pair of stopping times $\left( \sigma, \tau\right)$ with $0\leq \sigma \leq \tau \leq  \sigma + \delta \leq T$, we have
\begin{equation*}
\mathbb{E}[\la M^{f, j}\ra_\tau - \la M^{f, j}\ra_\sigma] \leq C\delta.
\end{equation*}
and
\begin{equation*}
\mathbb{E}[|A^{f, j}_\tau - A^{f, j}_\sigma|^2] \leq C'\delta^2,
\end{equation*}
for $j = 1, \cdots, J$ and  $A^{f, j}$ , $M^{f, j}$ follow the definitions (\ref{Eq:Mean_f_j}), (\ref{Eq:Martingale_f_j}) respectively.
\end{lemma}
\begin{proof}
Following (\ref{Eq:Quadratic_Variation_f_j}), we can obtain that
\begin{align*}
\avg{\la M^{f, j}\ra_\tau-\la M^{f, j}\ra_\sigma} &=  \frac{1}{\gamma^2}\avg{\int_{\sigma}^{\tau}\sum_{i = 1}^{\gamma\la 1,\mu^{{\vec{\zeta}}, j}_{s-}\ra}\left(\sqrt{2D_{j}}\frac{\partial f}{\partial Q}(H^{i}(\gamma\mu^{{\vec{\zeta}},j}_{s-}))\right)^2ds }\nonumber\\
& \quad
+\sum_{\ell = 1}^L\avg{ \int_{\sigma}^{\tau}\int_{\mathbb{Y}^{(\ell)}} \frac{1}{\gamma^2} \sum_{\{\vec{i} \in \Omega^{(\ell)}(\gamma\mu^{\vec{\zeta}}_{s-})\}}   \left( - \sum_{r = 1}^{\alpha_{\ell j}} f(H_Q^{i_r^{j}}(\gamma\mu^{{\vec{\zeta}},j}_{s-})) + \sum_{r = 1}^{\beta_{\ell j}} f(y_r^{j}) \right)^2\nonumber\\
&\qquad\quad
\times  K_\ell^{\gamma}\left(\mathcal{P}^{(\ell)}(\gamma\mu_{s-}^{\vec{\zeta}}, \vec{i}) \right)  \times  m^\eta_\ell\left(\vec{y} \,  | \, \mathcal{P}^{(\ell)}(\gamma\mu_{s-}^{\vec{\zeta}}, \vec{i}) \right)\, d \vec{y} \, ds } ,\nonumber\\
&
\leq \frac{C(\mu) \avg{\tau - \sigma}||f||^2_{C_0^1(\mathbb{R}^d)}}{\gamma}(C(D_j) + C(K)L) ,\nonumber\\
&\leq C\delta. \nonumber
\end{align*}
From (\ref{Eq:Mean_f_j}), we obtain
\begin{align*}
\avg{|A^{f, j}_\tau - A^{f, j}_\sigma|^2}&= \avg{| \int_{\sigma}^{\tau}\la (\mathcal{L}_j f)(x), \mu_{s-}^{{\vec{\zeta}}, j}(dx) \ra ds + \sum_{\ell = 1}^L \int_{\sigma}^{\tau} \int_{\tilde{\mathbb{X}}^{(\ell)}}     \frac{1}{\vec{\alpha}^{(\ell)}!}   K_\ell\left(\vec{x}\right) \nonumber\\
&
\qquad  \times\left(\int_{\mathbb{Y}^{(\ell)}}    \left( \sum_{r = 1}^{\beta_{\ell j}} f(y_r^{(j)}) \right) m^\eta_\ell\left(\vec{y} \, |\, \vec{x} \right)\, d \vec{y} - \sum_{r = 1}^{\alpha_{\ell j}} f(x_r^{(j)}) \right)\,\lambda^{(\ell)}[\mu^{\vec{\zeta}}_{s-}](d\vec{x}) \, ds |^2}.\nonumber\\
&
\leq 2  \avg{| \int_{\sigma}^{\tau} \la |(\mathcal{L}_j f)(x)|, \mu_{s-}^{{\vec{\zeta}}, j}(dx) \ra ds|^2} + 2\avg{| \sum_{\ell = 1}^L \int_{\sigma}^{\tau} \int_{\tilde{\mathbb{X}}^{(\ell)}}     \frac{1}{\vec{\alpha}^{(\ell)}!}   K_\ell\left(\vec{x}\right)  \nonumber\\
& \qquad
 \times\left(\int_{\mathbb{Y}^{(\ell)}}    \left( \sum_{r = 1}^{\beta_{\ell j}} |f(y_r^{(j)})| \right) m^\eta_\ell\left(\vec{y} \, |\, \vec{x} \right)\, d \vec{y} +  \sum_{r = 1}^{\alpha_{\ell j}} |f(x_r^{(j)}) |\right)\,\lambda^{(\ell)}[\mu^{\vec{\zeta}}_{s-}](d\vec{x}) \, ds |^2}\nonumber\\
&
\leq C(D_j)||f||^2_{C_0^1(\mathbb{R}^d)}C(\mu)^2 \avg{|\tau-\sigma|^2}+ C(K)||f||^2_{C^{2}_0(\mathbb{R}^d)}L^2C(\mu)^4 \avg{|\tau-\sigma|^2}  \nonumber\\
&\leq C'\delta^2.
\end{align*}
\end{proof}

\begin{lemma}[Aldous condition]\label{lem:AldousCondition}
For any $T > 0$, $\epsilon_1 > 0$, $\epsilon_2 > 0$, $j = 1, 2, \cdots,  J,$ there exists $\delta > 0$ and $n_0$ such that for any sequence $\left( \sigma_n, \tau_n\right)_{n\in \N}$ of pairs of stopping times with $\sigma_n\leq \tau_n\leq T$,
\begin{equation*}
\sup_{n\geq n_0}\mathbb{P}\{|\la f, \mu^{n, j}_{\sigma_n}\ra - \la f, \mu^{n, j}_{\tau_n}\ra | \geq \epsilon_2 , \tau_n\leq \sigma_n + \delta\}\leq \epsilon_1.
\end{equation*}
\end{lemma}
\begin{proof}
This follows as for $ \tau_n\leq \sigma_n + \delta$.
\begin{align*}
&  \sup_{n\geq n_0}\mathbb{P}\{ |M^{f, j}_{\sigma_n} - M^{f, j}_{\tau_n}| + |A^{f, j}_{\sigma_n} - A^{f, j}_{\tau_n}|  \geq \epsilon_2\}\nonumber \\
&\leq \sup_{n\geq n_0} 2\frac{1}{\epsilon_2^2} \mathbb{E}[\la M^{f, j}\ra_{\tau_n} - \la M^{f, j}\ra_{\sigma_n} + |A^{f, j}_{\tau_n} - A^{f, j}_{\sigma_n}|^2   ] \quad \text{ (by Markov inequality)}\nonumber\\
&\leq \frac{2}{\epsilon_2^2}(C\delta + C'\delta^2)\leq \epsilon_1 \quad \text{(by Lemma \ref{lem:tightStoppingTimes} and when $\delta$ sufficiently small)}
\end{align*}

\end{proof}

\begin{lemma}[Rebolledo Criterion] \label{lem:RebolledoCriterion}
For each $j = 1, \cdots, J$, the sequence of real-valued processes $\{\la f, \mu^{{\vec{\zeta}}, j}_{t}\ra\}_{\vec{\zeta}\in(0,1)^{2}}$ is tight in $\mathbb{D}_{\mathbb{R}}[0, T]$ .

\end{lemma}

\begin{proof}
By assumption, we always have $\la 1, \mu^{{\vec{\zeta}}, j}_{t}\ra$, $j = 1, \cdots, J$ are uniformly bounded. Since $f\in C_0^2(\R^d)$, we have
\begin{equation*}
\sup_{\|\vec{\zeta}\|\leq 1}\mathbb{E}[\sup_{t\in[0, T]}|\la f, \mu^{{\vec{\zeta}}, j}_{t}\ra|] < \infty
\end{equation*}
holds.
Combined with the Aldous condition from Lemma \ref{lem:AldousCondition}, we obtain that, for each $j = 1, 2, \cdots, J$, the sequence of real-valued processes $\{\la f, \mu^{{\vec{\zeta}}, j}_{t}\ra\}_{\vec{\zeta}\in(0,1)^{2}}$ is tight in $\mathbb{D}_{\mathbb{R}}[0, T]$ by the Rebolledo Criterion \cite{AM:1986}.

\end{proof}

\subsubsection{Tightness in $\mathbb{D}_{M_F(\mathbb{R}^d)}[0, T]$}

By the next Lemma \ref{lem:MassControl}, we're able to control the mass of measures outside of compact sets so that we can go from tightness in $\mathbb{D}_{M_F'(\mathbb{R}^d)}[0, T]$ to tightness in $\mathbb{D}_{M_F(\mathbb{R}^d)}[0, T]$.

\begin{lemma} \label{lem:MassControl}
There exists a sequence of $C^2_b(\mathbb{R}^d)$ functions $\{f_m(x)\}_{m\geq 0}$, in particular, $f_0\equiv 1$,  such that
\begin{align}
f_m(x)  = 0 & \text{ when } ||x|| \leq m-1 \nonumber\\
f_m(x)  = 1 & \text{ when } ||x|| > m \nonumber\\
0\leq f_m(x) \leq 1& \text{ when } m-1 < ||x|| \leq m.\label{Eq:Fm_def}
\end{align}
and furthermore, $\sup_{m\geq 0}||f_m(x)||_{C^2_b(\mathbb{R}^d)} := \sup_{m\geq 0}\sup_{x\in\mathbb{R}^d, |\alpha|\leq 2} |D^\alpha f_m(x)| <\infty$.
For such sequence of functions $\{f_m(x)\}_{m\geq 0}$,
\begin{equation}\label{eq:MassControl1}
\lim_{m\to \infty}\limsup_{{\vec{\zeta}}\to 0}\mathbb{E}\left(\sup_{t\in[0, T]} \la f_m, \mu_t^{{\vec{\zeta}}, j}\ra\right) = 0.
\end{equation}
for all $j = 1, 2, \cdots, J$.
\end{lemma}

\begin{proof}
Following~\cite{BMW:2012}, consider the function $\psi(s) = 6s^5 - 15s^4 + 10s^3\in C^2([0, 1])$. One can check that $\psi(0) = 1-\psi(1) = \psi'(0) = \psi'(1) = \psi''(0) = \psi''(1) = 0$. Now we define our functions $\{f_m(x)$, $x\in \mathbb{R}^d$, $m\geq 1\}$ as $f_m(x) = \psi(0\vee (||x|| - (m-1)) \wedge 1)$.
The derivatives of $f_m's$ are uniformly controlled by the derivatives of $\psi$, thus this choice satisfies our conditions. For any $\epsilon > 0$, by Assumption \ref{Assume:measureMrho}, there exists a large enough integer-valued radius $R$ such that $\int_{r > R}\rho(r) r^{d-1} dr < \epsilon$. As a consequence, for $m$ sufficiently large
\begin{align}\label{Eq:Control0}
\la f_m, \mu_t^{{\vec{\zeta}}, j} \ra
&
 = \la f_m, \mu_0^{{\vec{\zeta}}, j} \ra + M^{f_m, j}_t + \int_{0}^{t} \la (\mathcal{L}_j f_m)(x), \mu_{s-}^{{\vec{\zeta}}, j}(dx) \ra \nonumber\\
 &
\quad+\sum_{\ell = 1}^L \int_{0}^{t} \int_{\tilde{\mathbb{X}}^{(\ell)}}     \frac{1}{\vec{\alpha}^{(\ell)}!}   K_\ell\left(\vec{x}\right) \left(\int_{\mathbb{Y}^{(\ell)}}    \left( \sum_{r = 1}^{\beta_{\ell j}} f_m(y_r^{(j)}) \right) m^\eta_\ell\left(\vec{y} \, |\, \vec{x} \right)\, d \vec{y} - \sum_{r = 1}^{\alpha_{\ell j}} f_m(x_r^{(j)}) \right)\,\lambda^{(\ell)}[\mu^{\vec{\zeta}}_{s-}](d\vec{x}) \, ds ,\nonumber\\
&
\leq \la f_m, \mu_0^{{\vec{\zeta}}, j} \ra + M^{f_m, j}_t + ||f_m||_{C^2_b(\R^d)} \int_{0}^{t} \la 1_{\{m-1\leq||x|| < m\}} , \mu_{s-}^{{\vec{\zeta}}, j}(dx) \ra \nonumber\\
 &
\quad+\sum_{\ell = 1}^L \int_{0}^{t} \int_{\tilde{\mathbb{X}}^{(\ell)}}     \frac{1}{\vec{\alpha}^{(\ell)}!}   K_\ell\left(\vec{x}\right) \left(\int_{\mathbb{Y}^{(\ell)}}    \left( \sum_{r = 1}^{\beta_{\ell j}} f_m(y_r^{(j)}) \right) m^\eta_\ell\left(\vec{y} \, |\, \vec{x} \right)\, d \vec{y}\right)\,\lambda^{(\ell)}[\mu^{\vec{\zeta}}_{s-}](d\vec{x}) \, ds ,\nonumber\\
&
\leq \la f_m, \mu_0^{{\vec{\zeta}}, j} \ra + M^{f_m, j}_t + ||f_m||_{C^2_b(\R^d)} \int_{0}^{t} \la f_{m-1}(x) , \mu_{s-}^{{\vec{\zeta}}, j}(dx) \ra \nonumber\\
 &
\quad  +\sum_{\ell = 1}^L \int_{0}^{t} \int_{\tilde{\mathbb{X}}^{(\ell)}}     \frac{1}{\vec{\alpha}^{(\ell)}!}   K_\ell\left(\vec{x}\right) \left(\int_{\mathbb{Y}^{(\ell)}}    \left( \sum_{r = 1}^{\beta_{\ell j}} f_m(y_r^{(j)}) \right) m^\eta_\ell\left(\vec{y} \, |\, \vec{x} \right)\, d \vec{y}\right)\,\lambda^{(\ell)}[\mu^{\vec{\zeta}}_{s-}](d\vec{x}) \, ds .\nonumber\\
\end{align}

Taking supremum over time and then expectation on Eq \ref{Eq:Control0}, we  get
\begin{align}\label{Eq:Control}
\avg{\sup_{t\in[0, T]}  \la f_m, \mu_t^{{\vec{\zeta}}, j}\ra}
&\leq \avg{  \la f_m, \mu_0^{{\vec{\zeta}}, j} \ra }+ \avg{\sup_{t\in[0, T]}  \left|M^{f_m, j}_{t}\right|} + ||f_m||_{C^2_b(\R^d)} \avg{\sup_{t\in[0, T]}   \int_{0}^{t} \la f_{m-1}(x) , \mu_{s-}^{{\vec{\zeta}}, j}(dx) \ra} \nonumber\\
 &
\quad  +\avg{\sup_{t\in[0, T]}  \sum_{\ell = 1}^L \int_{0}^{t} \int_{\tilde{\mathbb{X}}^{(\ell)}}     \frac{1}{\vec{\alpha}^{(\ell)}!}   K_\ell\left(\vec{x}\right) \left(\int_{\mathbb{Y}^{(\ell)}}    \left( \sum_{r = 1}^{\beta_{\ell j}} f_m(y_r^{(j)}) \right) m^\eta_\ell\left(\vec{y} \, |\, \vec{x} \right)\, d \vec{y}\right)\,\lambda^{(\ell)}[\mu^{\vec{\zeta}}_{s-}](d\vec{x}) \, ds} .\nonumber\\
&\leq \avg{  \la f_m, \mu_0^{{\vec{\zeta}}, j} \ra }+ \avg{\sup_{t\in[0, T]} \left| M^{f_m, j}_{t}\right|} + ||f_m||_{C^2_b(\R^d)} \avg{\sup_{t\in[0, T]}   \int_{0}^{t} \la f_{m-1}(x) , \mu_{s-}^{{\vec{\zeta}}, j}(dx) \ra} \nonumber\\
 &
\quad  +\sup_{1\leq i \leq J}\avg{\sup_{t\in[0, T]}  C\int_{0}^{t} \la f_{m-1-R}(x), \mu^{{\vec{\zeta}}, i}_{s-}(dx) \ra \, ds} + C_2T(\epsilon + C\eta) ,\nonumber\\
&\qquad \text{( by Lemma \ref{lem:MassControlMollifier} studying the different cases of allowed reactions)}\nonumber\\
&
\leq  \avg{\sup_{t\in[0, T]}  \la f_m, \mu_0^{{\vec{\zeta}}, j} \ra }+ \avg{\sup_{t\in[0, T]} \left| M^{f_m, j}_{t}\right|}  \nonumber\\
&
\qquad+  C_1  \int_{0}^{T}\sup_{1\leq i \leq J} \avg{\sup_{s\in[0, t]} \la f_{m-1-R}(x), \mu^{{\vec{\zeta}}, i}_{s-}(dx) \ra }\, dt + C_2T(\epsilon + C\eta),\nonumber\\
\end{align}
where $C = 2L C(K)\left(C(\mu)\vee 1 \right)$, $C_1 = 2 (C\vee ||f_m||_{C^2_b(\R^d)} ) $ and $C_2 = 2LC(K)C(\mu)||f_m||_{C_b^2(\R^d)}$.

Let $Y_T^{m, {\vec{\zeta}}} := \sup_{1\leq j \leq J}\mathbb{E}\left[\sup_{t\in[0, T]} \left( \la f_m, \mu_t^{{\vec{\zeta}}, j}\ra  \right)\right]$. By construction, we always have $Y_T^{m, {\vec{\zeta}}} \leq Y_T^{m-1, {\vec{\zeta}}} \leq Y_T^{0, {\vec{\zeta}}}$. Due to the uniform boundedness of $||f_m||_{C^2_b(\mathbb{R}^d)}$, we have that $\mathbb{E}\la M^{f_m, j} \ra_T= \mathcal{O}(\frac{1}{\gamma})$, uniformly for each $1\leq j\leq J$ and for all $m$ based on Eq (\ref{Eq:Quadratic_Variation_f_j}). Without loss of generality, let's consider the subsequence where $m$ is divisible by $R+1$. Then Eq (\ref{Eq:Control}) gives
\begin{align}\label{Eq:Control2}
Y_T^{m, {\vec{\zeta}}}
&\leq Y_0^{m, {\vec{\zeta}}} + \sup_{1\leq j\leq J}\avg{\sup_{t\in[0, T]} |M^{f_m, j}_t|}  +  C_1\int_0^T Y_t^{m-(R+1), {\vec{\zeta}}} dt + C_2T(\epsilon + C\eta)\nonumber, \\
&\leq  Y_0^{m, {\vec{\zeta}}} + 2\sup_{1\leq j\leq J} \sqrt{\mathbb{E}\la M^{f_m, j}\ra_T}  +  C_1\int_0^TY_t^{m-(R+1), {\vec{\zeta}}} dt + C_2T(\epsilon + C\eta) \nonumber\\
& \qquad \qquad \text{ ( by Jensen's inequality and Doob's inequality ) }\nonumber\\
&
\leq  Y_0^{m, {\vec{\zeta}}} + C_{3}\frac{1}{\sqrt{\gamma}}+ C_2T(\epsilon + C\eta) + C_1\int_0^TY_t^{m-(R+1), {\vec{\zeta}}} dt \nonumber\\
&
\leq Y_0^{m, {\vec{\zeta}}} + C_{3}\frac{1}{\sqrt{\gamma}} + C_2T(\epsilon + C\eta)\nonumber\\
&
\qquad+ C_1\int_0^T (Y_0^{m-(R+1), {\vec{\zeta}}} + C_{3}\frac{1}{\sqrt{\gamma}} + C_2T(\epsilon + C\eta) +  C_1\int_0^tY_{t_1}^{m-2(R+1), {\vec{\zeta}}} dt_1) dt\nonumber\\
&
= Y_0^{m, {\vec{\zeta}}} + C_{3}\frac{1}{\sqrt{\gamma}} + C_2T(\epsilon + C\eta) \nonumber \\
&
\qquad + C_1T\left(Y_0^{m-(R+1), {\vec{\zeta}}} +  C_{3}\frac{1}{\sqrt{\gamma}} + C_2T(\epsilon + C\eta)\right)  + (C_1)^2\int_0^T\int_0^tY_{t_1}^{m-2(R+1), {\vec{\zeta}}} dt_1) dt\nonumber \\
&
\leq Y_0^{m, {\vec{\zeta}}} + C_{3}\frac{1}{\sqrt{\gamma}}+ C_2T(\epsilon + C\eta)+ C_1T\left(Y_0^{m-(R+1), {\vec{\zeta}}} +  C_{3}\frac{1}{\sqrt{\gamma}}+ C_2T(\epsilon + C\eta) \right) \nonumber \\
&\quad+ (C_1)^2\int_0^T\int_0^t (Y_0^{m-2(R+1), {\vec{\zeta}}} +  C_{3}\frac{1}{\sqrt{\gamma}} + C_2T(\epsilon + C\eta) + C_1\int_0^{t_{1}}Y_{t_2}^{m-3(R+1), {\vec{\zeta}}} dt_2) dt_1 dt \nonumber\\
&
= Y_0^{m, {\vec{\zeta}}} + C_{3}\frac{1}{\sqrt{\gamma}} + C_2T(\epsilon + C\eta)+ C_1T\left(Y_0^{m-(R+1), {\vec{\zeta}}} +  C_{3}\frac{1}{\sqrt{\gamma}}+ C_2T(\epsilon + C\eta)\right) \nonumber \\
&\quad+ \frac{(C_1T)^2}{2!}\left(Y_0^{m-2(R+1), {\vec{\zeta}}} +  C_{3}\frac{1}{\sqrt{\gamma}}+ C_2T(\epsilon + C\eta) \right) + (C_1)^3\int_0^T\int_0^t \int_0^{t_{1}}Y_{t_2}^{m-3(R+1), {\vec{\zeta}}} dt_2 dt_1 dt \nonumber\\
&
\leq \sum_{l=0}^{m/(R+1)-1} \frac{(C_1T)^{l}}{l!}\left(Y_0^{m-(R+1)l, {\vec{\zeta}}} + C_{3}\frac{1}{\sqrt{\gamma}}+ C_2T(\epsilon + C\eta) \right) + \frac{(C_1T)^{m/(R+1)}}{{(m/(R+1))}!}Y_T^{0, {\vec{\zeta}}}\nonumber\\
&
\leq  \sum_{l=0}^{\left\lfloor m/(2R+2) \right\rfloor} \frac{(C_1T)^l}{l!}\left(Y_0^{m-(R+1)l, {\vec{\zeta}}} +  C_{3}\frac{1}{\sqrt{\gamma}}+ C_2T(\epsilon + C\eta)\right) \nonumber\\
&\quad+ \sum_{l=\left\lfloor m/(2R+2) \right\rfloor +1 }^{m/(R+1)-1} \frac{(C_1T)^l}{l!}\left(Y_0^{m-(R+1)l, {\vec{\zeta}}} +  C_{3}\frac{1}{\sqrt{\gamma}}+ C_2T(\epsilon + C\eta) \right) + \frac{(C_1T)^{m/(R+1)}}{(m/(R+1))!}Y_T^{0, {\vec{\zeta}}}\nonumber\\
&
\leq \left(Y_0^{m-\left\lfloor m/2 \right\rfloor-(R+1), {\vec{\zeta}}} + C_{3}\frac{1}{\sqrt{\gamma}} + C_2T(\epsilon + C\eta) \right) \left(\sum_{l=0}^{\left\lfloor m/(2R+2) \right\rfloor} \frac{(C_1T)^l}{l!}\right) \nonumber\\
&\quad+ \left(Y_0^{0, {\vec{\zeta}}} + C_{3}\frac{1}{\sqrt{\gamma}}+ C_2T(\epsilon + C\eta)\right)\left( \sum_{l=\left\lfloor m/(2R+2) \right\rfloor +1 }^{\infty} \frac{(C_1T)^l}{l!} \right)+ \frac{(C_1T)^{m/(R+1)}}{(m/(R+1))!}Y_T^{0, {\vec{\zeta}}}\nonumber\\
&
\leq \left(Y_0^{m-\left\lfloor m/2\right\rfloor - (R+1), {\vec{\zeta}}} + C_{3}\frac{1}{\sqrt{\gamma}} + C_2T(\epsilon + C\eta)\right) e^{C_1T} \nonumber\\
&\quad+ \left(Y_0^{0, {\vec{\zeta}}} + C_{3}\frac{1}{\sqrt{\gamma}}+ C_2T(\epsilon + C\eta)\right)\left( \sum_{l=\left\lfloor m/(2R+2) \right\rfloor +1 }^{\infty} \frac{(C_1T)^l}{l!} \right)+ \frac{(C_1T)^{m/(R+1)}}{(m/(R+1))!}Y_T^{0, {\vec{\zeta}}}.\nonumber\\
\end{align}

From Eq (\ref{Eq:Control2}), we're able to obtain
\begin{align}\label{Eq:Control3}
&\lim_{m\to \infty}\limsup_{{\vec{\zeta}}\to 0} \sup_{1\leq j\leq J}\mathbb{E}\left[\sup_{t\in[0, T]} \left(\la f_m, \mu_t^{{\vec{\zeta}}, j}\ra \right)\right] \nonumber\\
&
 \leq \lim_{m\to \infty}\limsup_{{\vec{\zeta}}\to 0}\left[ \left(Y_0^{m-\left\lfloor m/2\right\rfloor - (R+1), {\vec{\zeta}}} + C_{3}\frac{1}{\sqrt{\gamma}} + C_2T(\epsilon + C\eta)\right) e^{C_1T} \right.\nonumber\\
&\quad+\left. \left(Y_0^{0, {\vec{\zeta}}} + C_{3}\frac{1}{\sqrt{\gamma}}+ C_2T(\epsilon + C\eta)\right)\left( \sum_{l=\left\lfloor m/(2R+2) \right\rfloor +1 }^{\infty} \frac{(C_1T)^l}{l!} \right)+ \frac{(C_1T)^{m/(R+1)}}{(m/(R+1))!}Y_T^{0, {\vec{\zeta}}}\right]\nonumber\\
&
 \quad =C_2T e^{C_1T}\epsilon + \lim_{m\to \infty}\left[\tilde{Y}_0^{m-\left\lfloor m/2 \right\rfloor - (R+1)} e^{CT} + \left(\tilde{Y}_0^{0}+ C_2T\epsilon  \right)\left( \sum_{l=\left\lfloor m/(2R+2) \right\rfloor +1 }^{\infty} \frac{(CT)^l}{l!} \right) +  \frac{(CT)^{m/(R+1)}}{(m/(R+1))!}\tilde{Y}_T^{0}\right]\nonumber\\
 &\qquad\qquad  \text{ ( where we denote $\tilde{Y}_t^{l} = \limsup_{{\vec{\zeta}}\to 0} Y_t^{l, {\vec{\zeta}}} $ ) }\nonumber\\
&\quad = C_2T e^{C_1T}\epsilon . \nonumber\\
\end{align}
Here, in the second last row of Eq (\ref{Eq:Control3}), inside the squared bracket, the first term vanishes as
\begin{align} \label{eq:ControlInitial}
\lim_{m\to \infty} \tilde{Y}_0^{m-\left\lfloor m/2 \right\rfloor - (R+1)}
&=  \lim_{m\to \infty} \limsup_{{\vec{\zeta}}\to 0}    \sup_{1\leq j \leq J}\mathbb{E}\left[\left( \la f_{m-\left\lfloor m/2 \right\rfloor-(R+1)}, \mu_0^{{\vec{\zeta}}, j}\ra  \right)\right] \nonumber\\
&=   \sup_{1\leq j \leq J} \lim_{m\to \infty}  \limsup_{{\vec{\zeta}}\to 0}  \mathbb{E}\left[\left( \la f_{m-\left\lfloor m/2 \right\rfloor-(R+1)}, \mu_0^{{\vec{\zeta}}, j}\ra  \right)\right] \nonumber\\
&=    \sup_{1\leq j \leq J} \lim_{m\to \infty} \mathbb{E} \la f_{m-\left\lfloor m/2 \right\rfloor-(R+1)}, \xi_0^{j}\ra  = 0.
\end{align}
The exchange of limits and supremum is allowed as the supremum is taken over the finite set $\{1, \cdots, J\}$. We can get the third line of Eq (\ref{eq:ControlInitial}) from the second line because of Assumption~\ref{Assume:initial} that the initial distribution $\mu_0^{{\vec{\zeta}}, j}$ converges weakly to $\xi_0^j$, for all $1\leq j \leq J$. Finally, we obtain the limit as zero using Assumption~\ref{Assume:initial} on the $\xi_0^j$'s, i.e. that they are compactly supported. For the second term, $ \tilde{Y}_0^{0}$, the initial concentration is bounded, whereas $\left( \sum_{l=\left\lfloor m/4 \right\rfloor +1 }^{\infty} \frac{(\tilde{C}T)^l}{l!} \right)$ is the remainder of exponential function expansion, which will go to 0 as $m\to \infty$. For the last term, as $\tilde{Y}_T^{0}  = \limsup_{{\vec{\zeta}}\to 0}\sup_{1\leq j\leq J}\mathbb{E}\left[\sup_{t\in[0, T]} \left(\la 1, \mu_t^{{\vec{\zeta}}, j}\ra  \right)\right] $ is bounded by $C(\mu)$,  the whole third term will vanish as $m\to\infty$.

Note, $C_1$ and $C_2$ do not depend on $\eta$, $R$ or $\epsilon$.  This implies $$\lim_{m\to \infty}\limsup_{{\vec{\zeta}}\to 0}\mathbb{E}\left[\sup_{t\in[0, T]} \left(\la f_m, \mu_t^{{\vec{\zeta}}, j}\ra  \right)\right] $$ is less than an arbitrary small number $C_2T e^{C_1T}\epsilon$, i.e. the limit is zero, for all $1\leq j\leq J$.
\end{proof}

Let $(\xi_t^1, \xi_t^2, \cdots,\xi_t^J)$ denote the weak limit of a subsequence of $(\mu_t^{{\vec{\zeta}}, 1}, \mu_t^{{\vec{\zeta}}, 2}, \cdots, \mu_t^{{\vec{\zeta}}, J})$ in $\mathbb{D}_{M_F'(\mathbb{R}^{3d})}([0, T])$ as $\vec{\zeta} \to 0$, where we abuse notation and let $(\mu_t^{{\vec{\zeta}}, 1}, \mu_t^{{\vec{\zeta}}, 2}, \cdots, \mu_t^{{\vec{\zeta}}, J})$ also denote the corresponding subsequence. Then

\begin{lemma}\label{lem:Continuity}
$\{\xi_t^j \}_{t\in [0, T]}$ is continuous process from $[0, T]$ to both $M_F'(\mathbb{R}^{d})$ and  $M_F(\mathbb{R}^{d})$ for each $j = 1, 2, \cdots,  J$.
\end{lemma}

\begin{proof}
By construction, see for example the proof of Lemma \ref{L:wconv}
 and in particular
 (\ref{Eq:MartingaleD_f_j})-(\ref{Eq:MartingaleD_integrand}), we have that
\[
 \sup_{t\in[0, T]}\sup_{f\in C_0^2(\mathbb{R}^n), ||f||_{L^\infty}\leq 1}|\la f, \mu_t^{{\vec{\zeta}}, j}\ra - \la f, \mu_{t-}^{{\vec{\zeta}}, j}\ra|\leq \frac{C}{\gamma}
\]
holds for some constant $C$ independent of $\gamma$. In addition, by Proposition 5.3 in Chapter 3 of \cite{EthierKurtz},  the mapping $\nu\mapsto \sup_{t\in [0, T]}|\la f, \nu_t \ra - \la f, \nu_{t-}\ra|$ is continuous on $\mathbb{D}_{M_F'(\mathbb{R}^d)}([0, T])$ for each $f\in C^2_0(\R^d)$. Then, by Theorem 10.2 in Chapter 3 of \cite{EthierKurtz},   we obtain as we take $\vec{\zeta}\to 0$,
that $\{\xi_t^j\}_{t\in [0, T]}$ is continuous process from $[0, T]$ to $M_F'(\mathbb{R}^{n})$. Next, we'll show that $\{\xi_t^j\}_{t\in [0, T]} \in \mathbb{D}_{M_F(\mathbb{R}^d)}([0, T])$, is a continuous process from $[0, T]$ to $M_F(\mathbb{R}^{d})$. To this end, we need to be able to control what happens to the total mass of the measures (see also \cite{BMW:2012}).

Adapting the notations in Lemma \ref{lem:MassControl}, let's define compactly supported functions $f_{m, r} = f_m(1-f_r)$. Notice that $f_{m,r}$ will converge monotonically to $f_{m}$ as $r\rightarrow\infty$. Then \begin{equation}\label{Eq:Continuity:DC}
\mathbb{E}\left(\sup_{t\in[0, T]}\la f_{m, r}, \xi^j_t\ra\right) = \lim_{{\vec{\zeta}}\to 0} \mathbb{E}\left(\sup_{t\in[0, T]}\la f_{m, r}, \mu^{{\vec{\zeta}}, j}_t\ra\right)\leq  \liminf_{{\vec{\zeta}}\to 0} \mathbb{E}\left(\sup_{t\in[0, T]}\la f_{m}, \mu^{{\vec{\zeta}}, j}_t\ra\right) < \infty
\end{equation}
by the continuity of the mapping $\nu\mapsto \sup_{t\in [0, T]} \la f, \nu_t \ra$, giving the first equality, and monotonicity of $f_{m, r}\leq f_m$, providing the second inequality.

From Eq (\ref{Eq:Continuity:DC}), taking the limit $r\to\infty$ first and using monotone convergence theorem, we'll get
\begin{equation*}
\mathbb{E}\left(\sup_{t\in[0, T]}\la f_{m}, \xi^j_t\ra\right) \leq  \liminf_{{\vec{\zeta}}\to 0} \mathbb{E}\left(\sup_{t\in[0, T]}\la f_{m}, \mu^{{\vec{\zeta}}, j}_t\ra\right) < \infty
\end{equation*}
which implies
\begin{equation}\label{Eq:Continuity:DC3}
\mathbb{E}\left(\sup_{t\in[0, T]}\la 1, \xi^j_t\ra\right)  < \infty.
\end{equation}
Using Lemma \ref{lem:MassControl}, it follows that
\begin{equation}\label{Eq:Continuity:DC4}
\lim_{m\to\infty}\mathbb{E}\left(\sup_{t\in[0, T]}\la f_{m}, \xi^j_t\ra\right ) = 0.
\end{equation}

This gives that there exists subsequence of $\sup_{t\in[0, T]} \la f_{m}, \xi^j_t\ra$, $m\geq0$, such that $\sup_{t\in[0, T]}\la f_{m}, \xi^j_t\ra \to 0$ almost surely, $\{\xi^{j}_{t}\}_{t\leq T}$ is a tight sequence almost surely  and  $\{\xi_t^j\}_{t\in [0, T]}$ is in $\mathcal{C}_{M^{'}_F(\mathbb{R}^d)}([0, T])$. Due to the latter fact and due to (\ref{Eq:Continuity:DC4}), $\{\xi_t^j\}_{t\in [0, T]}$  is in $\mathcal{C}_{M_F(\mathbb{R}^d)}([0, T])$ as well.
\end{proof}

\begin{theorem}[Tightness]\label{thm:tightness}
The measure-valued process $\{\mu^{{\vec{\zeta}}, j}_t\}_{t\in [0, T]}$ is  tight in $\mathbb{D}_{M_F(\mathbb{R}^d)}[0, T]$, for each $j= 1, 2,\cdots, J$.
\end{theorem}
\begin{proof}
Referring to Meleard and Roelly \cite{MS:1993}, to prove tightness in $\mathbb{D}_{M_F(\mathbb{R}^d)}[0, T]$, it suffices to prove $\la 1, \mu_t^{{\vec{\zeta}}, j}\ra$ converges in law to $\la 1, \xi_t^j \ra$ in $\mathbb{D}_\mathbb{R}([0, T])$, where note that $\xi_t^j$ is the limit point of $ \mu_t^{{\vec{\zeta}}, j}$ in $\mathbb{D}_{M_F'(\mathbb{R}^d)}[0, T]$, which also lies in $\mathcal{C}_{M_F(\mathbb{R}^d)}([0, T])$ by Lemma \ref{lem:Continuity}.

Let $F$ be any globally Lipschitz continuous bounded function from $\mathbb{R}$ to $\mathbb{R}$. Then
\begin{align}\label{Eq:LipContConv}
&\limsup_{{\vec{\zeta}}\to 0} |\mathbb{E}[F(\la 1, \mu_t^{{\vec{\zeta}}, j}\ra) - F(\la 1, \xi_t^j \ra)]|\nonumber\\
&\leq \lim_{m\to\infty}\limsup_{{\vec{\zeta}}\to 0} |\mathbb{E}[F(\la 1, \mu_t^{{\vec{\zeta}}, j}\ra) - F(\la 1 - f_m, \mu_t^{{\vec{\zeta}}, j}\ra)]| + \lim_{m\to\infty}\limsup_{{\vec{\zeta}}\to 0} |\mathbb{E}[F(\la 1 - f_m, \mu_t^{{\vec{\zeta}}, j}\ra)- F(\la 1-f_m, \xi_t^j \ra)  ]|\nonumber\\
&\quad +\lim_{m\to\infty} |\mathbb{E}[F(\la 1-f_m, \xi_t^j \ra) - F(\la 1, \xi_t^j \ra)]|,\nonumber\\
& = 0,\nonumber\\
\end{align}
where on the righthand side of Eq (\ref{Eq:LipContConv}), the first and third term become 0 as a result of Lemma  \ref{lem:MassControl}, while the second term vanishes due to the continuity of the mapping $\nu\mapsto \sup_{t\in [0, T]} \la 1- f_m, \nu_t \ra$ by noting that $1-f_m$ is compactly supported.
\end{proof}


\subsection{Uniqueness of Limiting Solution}\label{S:Uniqueness}

We've established tightness of the measure-valued processes $\{\mu^{{\vec{\zeta}}, j}_t\}_{t\in [0, T]}$, for all $1\leq j \leq J$ (See Theorem~\ref{thm:tightness}). 
We now show that the limiting measure is unique.  

For a measurable complete metric space $E$, $\nu \in M_F(E)$, define the norm $||\cdot||_{M_F(E)}$ on $M_F(E)$ as
\[
||\nu||_{M_F(E)} = \sup_{f\in L^\infty(E), ||f||_{L^\infty}\leq 1} |\la f, \nu \ra_E|,
\]
which is the variation norm of finite measures. Using density argument, one can show that this is equivalent to  (See step 4 of Theorem 3.2. of \cite{BMW:2012})
$$||\nu||_{M_F(E)} = \sup_{f\in C^2_b(E), ||f||_{L^\infty}\leq 1} |\la f, \nu \ra_E|. $$
For our purpose, we'll use test function $f\in C^2_b(E)$. The following two results then imply uniqueness:
\begin{lemma}\label{lem:productMeasure}

Let $E = (\R^d)^n$ be a product space of $\R^d$, $n\geq 1$. Let $\mu^1, \cdots , \mu^n \in M_F(\R^d)$, $\nu^1, \cdots , \nu^n \in M_F(\R^d)$ and $\otimes_{i = 1}^n \mu^i$, $\otimes_{i = 1}^n \nu^i$ be product measures on $E$. Then
$$ ||\otimes_{i = 1}^n \mu^i - \otimes_{i = 1}^n \nu^i||_{M_F(E)} \leq \sum_{i = 1}^n \left( ||\mu^i - \nu^i||_{M_F(\R^d)}  \times \Pi_{j = 1}^{i-1} \la 1, \mu^j\ra \times \Pi_{j = i+1}^{n} \la 1, \nu^j\ra \right).$$

\end{lemma}
\begin{proof}
For any $f\in L^\infty(E)$, $||f||_{L^\infty}\leq 1$, we have
\begin{align}
|\la f, \otimes_{i = 1}^n \mu^i - \otimes_{i = 1}^n \nu^i\ra_E|
&
= |\la f, \sum_{i = 1}^n\left((\otimes_{j = 1}^{i-1} \mu^j) \otimes (\mu^i - \nu^i)  \otimes ( \otimes_{j = i+1}^{n} \nu^j)\right) \ra_E| ,\nonumber\\
&
\leq \sum_{i = 1}^n |\la f, \left( (\otimes_{j = 1}^{i-1} \mu^j) \otimes (\mu^i - \nu^i)  \otimes ( \otimes_{j = i+1}^{n} \nu^j)\right) \ra_E|,\nonumber\\
&
\leq \sum_{i = 1}^n \left(||\mu^i - \nu^i||_{M_F(\R^d)} \times\Pi_{j = 1}^{i-1} \la 1, \mu^j\ra\times \Pi_{j = i+1}^{n} \la 1, \nu^j\ra \right).
\label{Eq:productMeasureIneq}
\end{align}
The last inequality is due to the assumption that $||f||_{L^\infty(E)}\leq 1$ and using the definition of signed measure norms.

Since Eq (\ref{Eq:productMeasureIneq}) is true for all  $f\in L^\infty(E)$, $||f||_{L^\infty}\leq 1$, Lemma \ref{lem:productMeasure} is proved.
\end{proof}

\begin{corollary}\label{cor:productMeasure}
Let $E = (\R^d)^n$ be a product space of $\R^d$, $n\geq 1$. Let $\mu^1, \cdots , \mu^n \in M_F(\R^d)$, $\nu^1, \cdots , \nu^n \in M_F(\R^d)$ and $\otimes_{i = 1}^n \mu^i$, $\otimes_{i = 1}^n \nu^i$ be product measures on $E$. If there exists $M > 0$, such that $|\la 1, \mu^i\ra |\leq M$ and $|\la 1, \nu^i\ra |\leq M$ for all $1\leq i \leq n$, then
$$ ||\otimes_{i = 1}^n \mu^i - \otimes_{i = 1}^n \nu^i||_{M_F(E)} \leq M^{n-1}\sum_{i = 1}^n ||\mu^i - \nu^i||_{M_F(\R^d)} .$$
\end{corollary}
\begin{proof}
This is a consequence of Lemma \ref{lem:productMeasure} using the fact that $\la 1, \mu^i\ra$ or $\la 1, \nu^i\ra$ are uniform bounded by $M$.
\end{proof}

\begin{lemma}[Uniqueness]\label{lem:uniqueness}
The solution to (\ref{Eq:Limit_EM_formula2}) is unique in $\mathcal{C}_{M_F(\mathbb{R}^d)}([0, T])$.
\end{lemma}

\begin{proof}
Suppose we have two different set of solutions to (\ref{Eq:Limit_EM_formula2}), $\{(\xi_t^1, \xi_t^2, \cdots, \xi_t^J)\}_{t\in [0, T]}$ and $\{(\bar{\xi}_t^1, \bar{\xi}_t^2, \cdots, \bar{\xi}_t^J)\}_{t\in [0, T]}$, with the same initial condition  $(\xi_0^1, \xi_0^2, \cdots, \xi_0^J) = (\bar{\xi}_0^1, \bar{\xi}_0^2, \cdots,  \bar{\xi}_0^J)$. In Eq~(\ref{Eq:Limit_EM_formula2}), if we use a test function of the form of $\psi_t(x) \in C^{1, 2}_b(\mathbb{R}_+\times\mathbb{R}^d)$, it becomes
\begin{align}
\la \psi_t,\xi^j_{t}\ra
&
=\la \psi_0,\xi^{ j}_{0}\ra + \int_{0}^{t} \la \partial_s\psi_s + (\mathcal{L}_j \psi_s)(x), \xi_{s}^{ j}(dx) \ra ds\nonumber\\
&
\quad+\sum_{\ell = 1}^L \int_{0}^{t} \int_{\tilde{\mathbb{X}}^{(\ell)}}     \frac{1}{\vec{\alpha}^{(\ell)}!}   K_\ell\left(\vec{x}\right) \left(\int_{\mathbb{Y}^{(\ell)}}    \left( \sum_{r = 1}^{\beta_{\ell j}} \psi_s(y_r^{(j)}) \right) m_\ell\left(\vec{y} \, |\, \vec{x} \right)\, d \vec{y} - \sum_{r = 1}^{\alpha_{\ell j}} \psi_s(x_r^{(j)}) \right)\,\lambda^{(\ell)}[\xi_{s}](d\vec{x}) \, ds.
\label{Eq:Limit_EM_formula_tDepTest}
\end{align}

Let $\mathcal{P}_{j, t}, t\geq 0$, be the semigroup generated by $\mathcal{L}_j$, $j = 1, 2, \cdots, J$. Choose $\psi_s(x) = \mathcal{P}_{j, t-s}f(x)$,  respectively for each $1\leq j \leq J$ , where $f\in C_b^2(\mathbb{R}^d)$, $||f||_{L^\infty}\leq 1$, then Eq (\ref{Eq:Limit_EM_formula_tDepTest})  becomes

\begin{align}
\la f,\xi^j_{t}\ra
&
=\la \mathcal{P}_{j, t}f, \xi^{ j}_{0}\ra +\sum_{\ell = 1}^L \int_{0}^{t} \int_{\tilde{\mathbb{X}}^{(\ell)}}     \frac{1}{\vec{\alpha}^{(\ell)}!}  \left(\int_{\mathbb{Y}^{(\ell)}}    \left( \sum_{r = 1}^{\beta_{\ell j}}  \mathcal{P}_{\ell, t-s}f(y_r^{(j)}) \right) m_\ell\left(\vec{y} \, |\, \vec{x} \right)\, d \vec{y} - \sum_{r = 1}^{\alpha_{\ell j}}  \mathcal{P}_{\ell, t-s}f(x_r^{(j)}) \right)\, \nonumber\\
&\qquad \times   K_\ell\left(\vec{x}\right) \lambda^{(\ell)}[\xi_{s}](d\vec{x}) \, ds.
\label{Eq:Limit_EM_formula_tDepTest2}
\end{align}

From Eq (\ref{Eq:Limit_EM_formula_tDepTest2}), we obtain the following estimates for $\la f,\xi^{j}_{t} - \bar{\xi}^{j}_{t}\ra$. By (\ref{Eq:Continuity:DC3}) we have that
\[
M = 1\vee \sup_{\{t\in [0, T], j = 1, 2, \cdots, J\}}|\la 1, \xi_t^j\ra|  \vee |\la 1, \bar{\xi}_t^j\ra| < \infty,
\]
which then gives
\begin{align}
|\la f,\xi^j_{t} - \bar{\xi}^{j}_t\ra|
&
\leq \sum_{\ell = 1}^L \int_{0}^{t} |\int_{\tilde{\mathbb{X}}^{(\ell)}}     \frac{1}{\vec{\alpha}^{(\ell)}!}  \left(\int_{\mathbb{Y}^{(\ell)}}    \left( \sum_{r = 1}^{\beta_{\ell j}}  \mathcal{P}_{\ell, t-s}f(y_r^{(j)}) \right) m_\ell\left(\vec{y} \, |\, \vec{x} \right)\, d \vec{y} - \sum_{r = 1}^{\alpha_{\ell j}}  \mathcal{P}_{\ell, t-s}f(x_r^{(j)}) \right)\, \nonumber\\
&\qquad \times   K_\ell\left(\vec{x}\right) \left( \lambda^{(\ell)}[\xi_{s}](d\vec{x}) -  \lambda^{(\ell)}[\bar{\xi}_{s}](d\vec{x})\right) |\, ds, \nonumber\\
&
\leq C(K) \sum_{\ell = 1}^L  \frac{\alpha_{\ell j} + \beta_{\ell j}}{\vec{\alpha}^{(\ell)}!} \int_{0}^{t} || \lambda^{(\ell)}[\xi_{s}] -  \lambda^{(\ell)}[\bar{\xi}_{s}] ||_{M_F(\mathbb{X}^{(\ell)})}\, ds,  \nonumber\\
&
\leq C(K) \sum_{\ell = 1}^L  \frac{\alpha_{\ell j} + \beta_{\ell j}}{\vec{\alpha}^{(\ell)}!} \int_{0}^{t} || \otimes_{i = 1}^J ( \otimes_{r = 1}^{\alpha_{\ell i }} \xi^i_s ) - \otimes_{i = 1}^J ( \otimes_{r = 1}^{\alpha_{\ell i }}  \bar{\xi}^i_s)||_{M_F(\mathbb{X}^{(\ell)})}\, ds,  \nonumber\\
&
\leq C(K) \sum_{\ell = 1}^L  \frac{\alpha_{\ell j} + \beta_{\ell j}}{\vec{\alpha}^{(\ell)}!} \int_{0}^{t} M^{|\alpha^{(\ell)}|-1}\sum_{i = 1, \cdots, J} \alpha_{\ell i }  || \xi^i_{s} - \bar{\xi}^i_{s} ||_{M_F(\R^d)}\, ds,
\label{Eq:Limit_EM_formula_tDepTest_j}
\end{align}
where the last inequality is due to Corollary \ref{cor:productMeasure} since $\la 1, \xi_s^i\ra$ or $\la 1, \bar{\xi}_s^i\ra$ are uniformly bounded by $M$ for all $1\leq i\leq J$. In the second to the last equality of Eq (\ref{Eq:Limit_EM_formula_tDepTest_j}), we have used the following estimates.
\begin{align}
&|  \frac{1}{\vec{\alpha}^{(\ell)}!}  \left(\int_{\mathbb{Y}^{(\ell)}}    \left( \sum_{r = 1}^{\beta_{\ell j}}  \mathcal{P}_{\ell, t-s}f(y_r^{(j)}) \right) m_\ell\left(\vec{y} \, |\, \vec{x} \right)\, d \vec{y} - \sum_{r = 1}^{\alpha_{\ell j}}  \mathcal{P}_{\ell, t-s}f(x_r^{(j)}) \right)  K_\ell\left(\vec{x}\right) |\nonumber\\
&\leq |  \frac{1}{\vec{\alpha}^{(\ell)}!}  \left(\int_{\mathbb{Y}^{(\ell)}}    \left( \sum_{r = 1}^{\beta_{\ell j}}  |\mathcal{P}_{\ell, t-s}f(y_r^{(j)}) |\right) m_\ell\left(\vec{y} \, |\, \vec{x} \right)\, d \vec{y} + \sum_{r = 1}^{\alpha_{\ell j}} | \mathcal{P}_{\ell, t-s}f(x_r^{(j)}) |\right)  K_\ell\left(\vec{x}\right) |  \nonumber\\
& \leq  \frac{\alpha_{\ell j} + \beta_{\ell j}}{\vec{\alpha}^{(\ell)}!} C(K). \nonumber\\
&\quad\quad \text{ (here we use the fact that $||\mathcal{P}_tf||_{L\infty} \leq 1$ and $ \int_{\mathbb{Y}^{(\ell)}}    m_\ell\left(\vec{y} \, |\, \vec{x} \right)\, d \vec{y} = 1$)}
\end{align}

Based on Eq (\ref{Eq:Limit_EM_formula_tDepTest_j}), we obtain
\begin{align}
\sum_{j= 1}^J ||\xi_t^j - \bar{\xi}_t^j ||_{M_F(\R^d)}
&
\leq C(K) \sum_{\ell = 1}^L \sum_{j= 1}^J   \frac{\alpha_{\ell j} + \beta_{\ell j}}{\vec{\alpha}^{(\ell)}!} \int_{0}^{t} M^{|\alpha^{(\ell)}|-1}\sum_{i = 1}^{J} \alpha_{\ell i }  || \xi^i_{s} - \bar{\xi}^i_{s} ||_{M_F(\R^d)}\, ds,  \nonumber\\
&
\leq C(K)LJ   \max_{1\leq \ell \leq L, 1\leq j \leq J}\left\{\frac{\alpha_{\ell j} + \beta_{\ell j}}{\vec{\alpha}^{(\ell)}!} M^{|\alpha^{(\ell)}|-1}\alpha_{\ell j } \right\} \int_{0}^{t} \sum_{i = 1}^{J}  || \xi^i_{s} - \bar{\xi}^i_{s} ||_{M_F(\R^d)}\, ds,  \nonumber
\end{align}

Applying Gronwall's inequality, we get $\sum_{j = 1}^J ||\xi_t^j - \bar{\xi}_t^j ||_{M_F(\R^d)}  = 0 $ for all $t\in [0, T]$, which proves the uniqueness of solution, concluding the proof of the lemma.
\end{proof}

\section*{Acknowledgements}
SAI thanks Romain Yvinec for first introducing him to the MVSP representation for particle systems and the work of~\cite{BM:2015}.




\begin{appendices}
\section{$\gamma$ dependency of reaction kernels} \label{A:gammascaling}
In this section we demonstrate one way in which the claimed $\gamma$ scaling given in Assumption~\ref{Assume:rescaling} arises for bimolecular reactions. 
We ignore the zeroth order case as our main result, Theorem~\ref{thm:convergence}, does not allow for such reactions. In the first order case the reaction rate kernel, $K^{\gamma}(\vx)$, is usually interpreted as an internal property of molecules, giving the probability per time an individual reactant particle can undergo the reaction. As such, it would not be expected to depend on $\gamma$. In contrast, the reaction rate kernel for a bimolecular reaction is often calibrated to agree with a known well-mixed reaction rate constant in the limit that the system is forced to be well-mixed (i.e. the limit that particle diffusivities are taken to be infinite), which ultimately gives rise to the $\gamma$ dependence.

Consider an isolated system containing only two particles that can undergo a bimolecular annihilation reaction of the form $\textrm{A} + \textrm{B} \to \varnothing$. Assume we are considering the reaction within a bounded domain $\Omega \subset \R^d$ with hypervolume $\abs{\Omega}$. In modeling chemical reaction systems, one is often given a spatially-homogeneous, well-mixed, macroscopic reaction rate, $\Kwm$, with units of $(\text{molar concentration})^{1-\alpha} (\text{time})^{-1}$ for a reaction of order $\alpha$. The corresponding reaction rate used in a spatially-homogeneous well-mixed stochastic chemical kinetics model is then $\bar{K} = (\gamma \abs{\Omega})^{1-\alpha} \Kwm$, with units of $(\text{time})^{-1}$. For our second order reaction $\bar{K}$ gives the probability per time for the pair of \textrm{A} and \textrm{B} molecules to react and annihilate in the well-mixed stochastic model.

Consider the PBSRD model's dynamics until the two reactants annihilate. Let $p(\vx,\vy,t)$ denote the probability density the particle positions are $\vx$ and $\vy$ respectively at time $t$, and no reaction has yet occurred. Then
\begin{align*}
  \PD{p}{t}(\vx,\vy,t) = (\DA \lap_{\vx} + \DB \lap_{\vy})p(\vx,\vy,t) - K^{\gamma}(\vx,\vy) p(\vx,\vy,t), \quad \vx \in \Omega, \, \vy \in \Omega, \, t > 0,
\end{align*}
with reflecting no-flux boundary conditions for $\vx$ or $\vy$ in $\partial \Omega$. In the formal well-mixed limit that the particle diffusivities are taken to be infinite, we expect that $p(\vx,\vy,t) = p(t)$. Letting $P(t) = p(t) \abs{\Omega}^2$ denote the probability the reaction has not yet occurred, we then have
\begin{equation*}
  \D{P}{t} = - \paren{\frac{1}{\abs{\Omega}^2} \int_{\Omega^2} K^{\gamma}(\vx,\vy) d\vx \, d\vy} P(t).
\end{equation*}
To match the well-mixed stochastic model we would then require that
\begin{equation*}
  \frac{1}{\abs{\Omega}^2} \int_{\Omega^2} K^{\gamma}(\vx,\vy) d\vx \, d\vy = \bar{K} = \frac{\Kwm}{\gamma \abs{\Omega}}.
\end{equation*}
If we assume that $K^{\gamma}(\vx,\vy) = \gamma^{\beta} K(\vx,\vy)$, then we immediately obtain the scaling given in Assumption~\ref{Assume:rescaling}, i.e. $\beta = -1$.

More concretely, consider the widely used Doi interaction $K^{\gamma}(\vx,\vy) = \lambda 1_{\brac{0,\varepsilon}}(\abs{\vx-\vy})$~\cite{DoiSecondQuantA,DoiSecondQuantB}. We find that
\begin{equation*}
  \lambda = \frac{\Kwm \abs{\Omega}}{\gamma \abs{\mathcal{R} \cap \Omega^2}},
\end{equation*}
where $\mathcal{R} = \{(\vx,\vy) \in \R^{2d} | \abs{\vx-\vy} \leq \epsilon \}$. As $\Omega \to \R^d$ we formally expect
\begin{equation*}
  \lambda \to \frac{\Kwm}{\gamma \abs{B_{\varepsilon}}},
\end{equation*}
where $\abs{B_{\varepsilon}}$ denotes the hypervolume of the ball of radius $\varepsilon$. This  demonstrates that the scaling of Assumption~\ref{Assume:rescaling} persists in freespace.

\section{Mass Control Lemmas for Different Cases of Reactions}\label{A:appendixA}

The goal of this section is to prove the following key estimate.
\begin{lemma}\label{lem:MassControlMollifier}
Recall the definition~\eqref{Eq:Fm_def} of the functions $f_m$. For the $\ell$-th reaction, $1\leq \ell \leq L$, let $\eta$ be sufficiently small, $\epsilon>0$ and $R\in\mathbb{N}$ as in Assumption \ref{Assume:measureMrho}. Then, the following estimates hold for $m$ large enough,
\begin{align*}
&\avg{\sup_{t\in[0, T]}\int_{0}^{t} \int_{\tilde{\mathbb{X}}^{(\ell)}}     \frac{1}{\vec{\alpha}^{(\ell)}!}   K_\ell\left(\vec{x}\right) \left(\int_{\mathbb{Y}^{(\ell)}}    \left( \sum_{r = 1}^{\beta_{\ell j}} f_m(y_r^{(j)}) \right) m^\eta_\ell\left(\vec{y} \, |\, \vec{x} \right)\, d \vec{y}\right)\,\lambda^{(\ell)}[\mu^{\vec{\zeta}}_{s-}](d\vec{x}) \, ds} \nonumber\\
&
\leq 2C(K)(C(\mu)\vee 1)  \sup_{1\leq i\leq J}\avg{\sup_{t\in[0, T]}\int_{0}^{t}  \la f_{m-1-R}(x) , \mu_{s-}^{{\vec{\zeta}}, i}(dx) \ra  \, ds } +2C(K)C(\mu) ||f_m||_{C_b^2(\R^d)}T\left(C\eta + \epsilon\right).   \nonumber\\\end{align*}
\end{lemma}

To do so we first prove some intermediate results and the proof of Lemma \ref{lem:MassControlMollifier} will then follow at the end of this section.

\begin{lemma}\label{lem:etaConv}
For any $\eta \geq 0$ small enough, $\tilde{L} + 1\leq \ell \leq L$, $\vec{y}\in \mathbb{Y}^{(\ell)}$, $\vec{x}\in \mathbb{X}^{(\ell)}$, and $f\in C^2_b(\mathbb{Y}^{(\ell)})$, there exists a constant $C$ such that
$$\abs{\int_{\mathbb{Y}^{(\ell)}} f(\vec{y})\left( m^\eta_{\ell}(\vec{y} \, | \, \vec{x}) - m_{\ell}(\vec{y} \, | \, \vec{x}) \right)\, d\vec{y}}\leq C ||f||_{C^2_b(\mathbb{Y}^{(\ell)})}\eta,$$
given Definition \ref{def:molifier} on the choice of positive mollifier and Assumptions \ref{Assume:measureOne2One} - \ref{Assume:measureMrho} on the placement densities.
\end{lemma}

\begin{proof}
This is essentially a result from the definition of mollifiers. We'll discuss this estimate for each of the following different cases of reactions. The upper bound is always some constant of the order of $\eta$ times $ ||f||_{C^1_b(\mathbb{Y}^{(\ell)})}$. 
\begin{description}

\item[Case 1] Reaction of the form $S_i \rightarrow S_j$.

Plugging in the definitions of $m^\eta_{\ell}(\vec{y} \, | \, \vec{x})$ and $m_{\ell}(\vec{y} \, | \, \vec{x})$ from Assumption \ref{Assume:measureOne2One}, we will get
\begin{align*}
|\int_{\mathbb{Y}^{(\ell)}} f(\vec{y})\left( m^\eta_{\ell}(\vec{y} \, | \, \vec{x}) - m_{\ell}(\vec{y} \, | \, \vec{x}) \right)\, d\vec{y}|
&= |\int_{\mathbb{R}^d} f(y)G_\eta(y - x) \, d y - f(x)|\\
& = |\int_{\mathbb{R}^d} \left(f(y) - f(x)\right)G_\eta(y - x) \, d y|\\
&\leq |\int_{B(x, \eta)} |f(y) - f(x)|G_\eta(y - x) \, d y|\\
&\leq |\int_{B(x, \eta)} ||f||_{C_b^1(\R^d)}\eta \times G_\eta(y - x) \, d y|\\
&\leq ||f||_{C_b^1(\R^d)}\eta.
\end{align*}

\item[Case 2] Reaction of the form $S_i \rightarrow S_j + S_k$.

Plugging in the definitions of $m^\eta_{\ell}(\vec{y} \, | \, \vec{x})$ and $m_{\ell}(\vec{y} \, | \, \vec{x})$ from Assumption \ref{Assume:measureOne2Two}, we will get
\begin{align*}
&|\int_{\mathbb{Y}^{(\ell)}} f(\vec{y})\left( m^\eta_{\ell}(\vec{y} \, | \, \vec{x}) - m_{\ell}(\vec{y} \, | \, \vec{x}) \right)\, d\vec{y}| \\
&= |\sum_{i = 1}^I p_i \times \left[\int_{\mathbb{R}^{2d}} f(y_1, y_2)\rho(|y_1 - y_2|)G_\eta\left(x - (\alpha_i y_1 + (1 - \alpha_i)y_2)\right)\, d y_1\,dy_2 \right.\\
&\quad - \left.\int_{\mathbb{R}^{2d}} f(y_1, y_2)\rho(|y_1 - y_2|)\delta\left(x - (\alpha_i y_1 + (1 - \alpha_i)y_2)\right)\, d y_1\,dy_2 \right]|\text{ (Let $w = y_1 - y_2$)}\\
&= |\sum_{i = 1}^I p_i \times \left[\int_{\mathbb{R}^{2d}} f(w + y_2, y_2)\rho(|w|)G_\eta\left(x - \alpha_i w  - y_2)\right)\, d w\,dy_2 \right.\\
&\quad - \left.\int_{\mathbb{R}^{2d}} f(w + y_2, y_2)\rho(|w|)\delta\left(x - \alpha_i w - y_2)\right)\, d w\,dy_2 \right]|\\
&= |\sum_{i = 1}^I p_i \times \left[\int_{\mathbb{R}^d} \rho(|w|) \left( \int_{\mathbb{R}^d}\left( f(w + y_2, y_2)  - f(w + x - \alpha_i w , x - \alpha_i w) \right)G_\eta\left(x - \alpha_i w  - y_2 \right) \,dy_2 \right)\, d w \right]\\
&\leq |\sum_{i = 1}^I p_i \times \left[\int_{\mathbb{R}^d} \rho(|w|) \left( \int_{B(x - \alpha_i w, \eta)}| f(w + y_2, y_2)  - f(w + x - \alpha_i w , x - \alpha_i w) | \times G_\eta\left(x - \alpha_i w  - y_2 \right) \,dy_2 \right)\, d w \right]\\
&\leq |\sum_{i = 1}^I p_i \times \left[\int_{\mathbb{R}^d} \rho(|w|) \left( \int_{B(x - \alpha_i w, \eta)}||f||_{C^1(\R^{2d})} \eta\times G_\eta\left(x - \alpha_i w  - y_2 \right) \,dy_2 \right)\, d w \right]\\
&\leq ||f||_{C^1(\R^{2d})}\eta \quad \text{( by noting that } \int_{\R^d} \rho(|w|)\, dw =1 \text{ in Assumption \ref{Assume:measureMrho} }.
\end{align*}

\item[Case 3] Reaction of the form $S_i + S_k \rightarrow S_j$.

Plugging in the definitions of $m^\eta_{\ell}(\vec{y} \, | \, \vec{x})$ and $m_{\ell}(\vec{y} \, | \, \vec{x})$ from Assumption \ref{Assume:measureTwo2One}, we will get
\begin{align*}
&|\int_{\mathbb{Y}^{(\ell)}} f(\vec{y})\left( m^\eta_{\ell}(\vec{y} \, | \, \vec{x}) - m_{\ell}(\vec{y} \, | \, \vec{x}) \right)\, d\vec{y}| \\
&= |\sum_{i = 1}^I p_i\times\left[\int_{\mathbb{R}^{d}} f(y)G_\eta\left(y - (\alpha_i x_1 + (1-\alpha_i) x_2) \right) \, d y - f(\alpha_i x_1 + (1-\alpha_i) x_2)\right]|\\
&= |\sum_{i = 1}^I p_i\times\int_{\mathbb{R}^{d}} \left( f(y) - f(\alpha_i x_1 + (1-\alpha_i) x_2) \right)G_\eta\left(y - (\alpha_i x_1 + (1-\alpha_i) x_2) \right) \, d y|\\
&\leq |\sum_{i = 1}^I p_i\times\int_{B(\alpha_i x_1 + (1-\alpha_i) x_2, \eta)} | f(y) - f(\alpha_i x_1 + (1-\alpha_i) x_2) |G_\eta\left(y - (\alpha_i x_1 + (1-\alpha_i) x_2) \right) \, d y|\\
&\leq |\sum_{i = 1}^I p_i\times\int_{B(\alpha_i x_1 + (1-\alpha_i) x_2, \eta)}  ||f||_{C_b^1(\R^d)}\eta\times G_\eta\left(y - (\alpha_i x_1 + (1-\alpha_i) x_2) \right) \, d y|\\
&\leq ||f||_{C_b^1(\R^d)}\eta
\end{align*}

\item[Case 4] Reaction of the form $S_i + S_k \rightarrow S_j + S_r$.

Plugging in the definitions of $m^\eta_{\ell}(\vec{y} \, | \, \vec{x})$ and $m_{\ell}(\vec{y} \, | \, \vec{x})$ from Assumption \ref{Assume:measureTwo2Two}, we will get
\begin{align*}
&|\int_{\mathbb{Y}^{(\ell)}} f(\vec{y})\left( m^\eta_{\ell}(\vec{y} \, | \, \vec{x}) - m_{\ell}(\vec{y} \, | \, \vec{x}) \right)\, d\vec{y}| \\
&= |p\times\left[\int_{\mathbb{R}^{2d}} f(y_1, y_2)G_\eta(y_1 - x_1) G_\eta(y_2 - x_2)\, d y_1 \, d y_2 - f(x_1, x_2)\right]\\
& \quad + (1-p)\times\left[\int_{\mathbb{R}^{2d}} f(y_1, y_2)G_\eta(y_2 - x_1) G_\eta(y_1 - x_2)\, d y_1 \, d y_2 - f(x_2, x_1)\right]|\\
&= |p\times\int_{\mathbb{R}^{2d}} \left(f(y_1, y_2) - f(x_1, x_2)\right)G_\eta(y_1 - x_1) G_\eta(y_2 - x_2)\, d y_1 \, d y_2 \\
& \quad + (1-p)\times\int_{\mathbb{R}^{2d}} \left(f(y_1, y_2) - f(x_2, x_1)\right)G_\eta(y_2 - x_1) G_\eta(y_1 - x_2)\, d y_1 \, d y_2|\\
&\leq |p\times\int_{B((x_1, x_2), \sqrt{2}\eta)} |f(y_1, y_2) - f(x_1, x_2)|G_\eta(y_1 - x_1) G_\eta(y_2 - x_2)\, d y_1 \, d y_2 \\
& \quad + (1-p)\times\int_{B((x_2, x_1), \sqrt{2}\eta)} | f(y_1, y_2) - f(x_2, x_1) | G_\eta(y_2 - x_1) G_\eta(y_1 - x_2)\, d y_1 \, d y_2|\\
&\leq |p\times\int_{B((x_1, x_2), \sqrt{2}\eta)} ||f||_{C_b^1(\R^{2d})}\times \sqrt{2}\eta\times G_\eta(y_1 - x_1) G_\eta(y_2 - x_2)\, d y_1 \, d y_2 \\
& \quad + (1-p)\times\int_{B((x_2, x_1), \sqrt{2}\eta)}||f||_{C_b^1(\R^{2d})}\times \sqrt{2}\eta\times G_\eta(y_2 - x_1) G_\eta(y_1 - x_2)\, d y_1 \, d y_2|\\
&\leq \sqrt{2}||f||_{C_b^1(\R^{2d})}\eta.
\end{align*}
\end{description}

\end{proof}


\begin{lemma}\label{lem:MassControlOne2One}
If the $\ell$-th reaction is a reaction of the form $S_i \rightarrow S_j$, then
\begin{equation}
\int_{0}^{t} \int_{\tilde{\mathbb{X}}^{(\ell)}}     \frac{1}{\vec{\alpha}^{(\ell)}!}   K_\ell\left(\vec{x}\right) \left(\int_{\mathbb{Y}^{(\ell)}}    \left( \sum_{r = 1}^{\beta_{\ell j}} f_m(y_r^{(j)}) \right) m_\ell\left(\vec{y} \, |\, \vec{x} \right)\, d \vec{y}\right)\,\lambda^{(\ell)}[\mu^{\vec{\zeta}}_{s-}](d\vec{x}) \, ds \leq C(K) \int_{0}^{t} \la f_m(x), \mu^{{\vec{\zeta}}, i}_{s-}(dx) \ra \, ds
\end{equation}
\end{lemma}

\begin{proof}
By plugging in the specific form of the reaction rate and placement density as in Assumption \ref{Assume:measureOne2One}, we have
\begin{align}
&\int_{0}^{t} \int_{\tilde{\mathbb{X}}^{(\ell)}}     \frac{1}{\vec{\alpha}^{(\ell)}!}   K_\ell\left(\vec{x}\right) \left(\int_{\mathbb{Y}^{(\ell)}}    \left( \sum_{r = 1}^{\beta_{\ell j}} f_m(y_r^{(j)}) \right) m_\ell\left(\vec{y} \, |\, \vec{x} \right)\, d \vec{y}\right)\,\lambda^{(\ell)}[\mu^{\vec{\zeta}}_{s-}](d\vec{x}) \, ds \nonumber\\
&
= \int_{0}^{t} \int_{\R^d}   K_\ell(x) \left(\int_{\R^d}   f_m(y) \delta_x(y) \, dy  \right)\,\mu^{{\vec{\zeta}}, i}_{s-}(dx) \, ds \nonumber\\
&
= \int_{0}^{t} \int_{\R^d}   K_\ell(x)   f_m(x)\mu^{{\vec{\zeta}}, i}_{s-}(dx) \, ds \nonumber\\
&
\leq C(K) \int_{0}^{t} \la f_m(x), \mu^{{\vec{\zeta}}, i}_{s-}(dx) \ra \, ds. \nonumber
\end{align}
\end{proof}
\begin{lemma}\label{lem:MassControlOne2Two}
If the $\ell$-th reaction is a reaction of the form $S_i \rightarrow S_j + S_k$, where $i$ and $k$ could be $j$, then for the choice of $\epsilon \geq 0$ and $R \in \N$ in Assumption \ref{Assume:measureMrho}, we have for $m$ large enough,
\begin{align}
&\int_{0}^{t} \int_{\tilde{\mathbb{X}}^{(\ell)}}     \frac{1}{\vec{\alpha}^{(\ell)}!}   K_\ell\left(\vec{x}\right) \left(\int_{\mathbb{Y}^{(\ell)}}    \left( \sum_{r = 1}^{\beta_{\ell j}} f_m(y_r^{(j)}) \right) m_\ell\left(\vec{y} \, |\, \vec{x} \right)\, d \vec{y}\right)\,\lambda^{(\ell)}[\mu^{\vec{\zeta}}_{s-}](d\vec{x}) \, ds\nonumber\\
& \leq 2C(K) \int_{0}^{t}   \la f_{m-1-R}(x) , \mu_{s-}^{{\vec{\zeta}}, i}(dx) \ra  \, ds  + 2C(K)  ||f_m||_{C_b^2(\R^d)}C(\mu)t\epsilon.\nonumber
\end{align}
\end{lemma}

\begin{proof}
Let $\epsilon \geq 0$ and $R \in \N$ be such that Assumption \ref{Assume:measureMrho} is satisfied. By plugging in the specific form of the reaction rate and placement density as in Assumption \ref{Assume:measureOne2Two}, we have
\begin{align}
&\int_{0}^{t} \int_{\tilde{\mathbb{X}}^{(\ell)}}     \frac{1}{\vec{\alpha}^{(\ell)}!}   K_\ell\left(\vec{x}\right) \left(\int_{\mathbb{Y}^{(\ell)}}    \left( \sum_{r = 1}^{\beta_{\ell j}} f_m(y_r^{(j)}) \right) m_\ell\left(\vec{y} \, |\, \vec{x} \right)\, d \vec{y}\right)\,\lambda^{(\ell)}[\mu^{\vec{\zeta}}_{s-}](d\vec{x}) \, ds \nonumber\\
&
\leq \int_{0}^{t} \int_{\R^d}   K_\ell(x) \left(\int_{\R^{2d}}   \left(f_m(y) + f_m(z)\right) m_\ell(y, z \, | \, x) \, dy \, dz \right)\,\mu^{{\vec{\zeta}}, i}_{s-}(dx) \, ds \nonumber\\
&
= \int_{0}^{t} \int_{\R^d}   K_\ell(x) \left(\int_{\R^{2d}}   \left(f_m(y) + f_m(z)\right) \rho(|y-z|)\sum_{i = 1}^I p_i\delta(x-(\alpha_i y+ (1-\alpha_i)z)) \, dy \, dz \right)\,\mu^{{\vec{\zeta}}, i}_{s-}(dx) \, ds \nonumber\\
&
\leq C(K)  \int_{0}^{t} \la \left( \int_{|y-z| \leq R} \left(f_m(y) + f_m(z)\right) \rho(|y-z|)\sum_{i = 1}^I p_i\delta(x-(\alpha_i y+ (1-\alpha_i)z)) \, dy \, dz \right) , \mu_{s-}^{{\vec{\zeta}}, i}(dx) \ra  \, ds.\nonumber\\
&
\quad  + C(K)  \int_{0}^{t} \la \left( \int_{|y-z| > R}  \left(f_m(y) + f_m(z)\right) \rho(|y-z|)\sum_{i = 1}^I p_i\delta(x-(\alpha_i y+ (1-\alpha_i)z)) \, dy \, dz\right) , \mu_{s-}^{{\vec{\zeta}}, i}(dx) \ra  \, ds.\nonumber\\
&
\leq  C(K)  \int_{0}^{t}  \la \left( \int_{|y-z| \leq R} 2 f_{m-1-R}(x) \rho(|y-z|) \sum_{i = 1}^I p_i \delta(x-(\alpha_i y+ (1-\alpha_i)z)) dy \, dz \right) , \mu_{s-}^{{\vec{\zeta}}, i}(dx) \ra  \, ds.\nonumber\\
&
 \quad  + 2C(K)  ||f_m||_{C_b^2(\R^d)}\int_{0}^{t} \la \sum_{i = 1}^I p_i \left( \int_{|\eta| > R} \int_{z\in\R^d} \rho(|\eta|)\delta(x-\alpha_i\eta -z)  \, dz \, d\eta \right) , \mu_{s-}^{{\vec{\zeta}}, i}(dx) \ra  \, ds.\nonumber\\
&
\leq 2C(K)  \int_{0}^{t}  \la  f_{m-1-R}(x) \left( \int_{|y-z| \leq R} m_\ell(y, z \, | \, x) dy \, dz \right) , \mu_{s-}^{{\vec{\zeta}}, i}(dx) \ra  \, ds.\nonumber\\
&
 \quad  + C(K)  ||f_m||_{C_b^2(\R^d)}\int_{0}^{t} \la \sum_{i = 1}^I p_i \left( \int_{|\eta| > R} \rho(|\eta|) \, d\eta \right) , \mu_{s-}^{{\vec{\zeta}}, i}(dx) \ra  \, ds.\nonumber\\
&
\leq 2C(K) \int_{0}^{t}   \la f_{m-1-R}(x) , \mu_{s-}^{{\vec{\zeta}}, i}(dx) \ra  \, ds  + 2C(K)  ||f_m||_{C_b^2(\R^d)}C(\mu)t\epsilon.\nonumber
\end{align}
\end{proof}


\begin{lemma}\label{lem:MassControlTwo2One}
If the $\ell$-th reaction is a reaction of the form $S_i + S_k \rightarrow S_j$, then
\begin{align}
&\int_{0}^{t} \int_{\tilde{\mathbb{X}}^{(\ell)}}     \frac{1}{\vec{\alpha}^{(\ell)}!}   K_\ell\left(\vec{x}\right) \left(\int_{\mathbb{Y}^{(\ell)}}    \left( \sum_{r = 1}^{\beta_{\ell j}} f_m(y_r^{(j)}) \right) m_\ell\left(\vec{y} \, |\, \vec{x} \right)\, d \vec{y}\right)\,\lambda^{(\ell)}[\mu^{\vec{\zeta}}_{s-}](d\vec{x}) \, ds\nonumber\\
& \leq C(K)C(\mu)\int_{0}^{t}  \left(\la f_{m-1}(x), \mu_{s-}^{{\vec{\zeta}}, i}(dx) \ra + \la f_{m-1}(y), \mu_{s-}^{{\vec{\zeta}}, k}(dy) \ra \right)ds.\nonumber
\end{align}
\end{lemma}

\begin{proof}

By plugging in the specific form of the reaction rate and placement density as in Assumption \ref{Assume:measureTwo2One}, we have
\begin{align}
&\int_{0}^{t} \int_{\tilde{\mathbb{X}}^{(\ell)}}     \frac{1}{\vec{\alpha}^{(\ell)}!}   K_\ell\left(\vec{x}\right) \left(\int_{\mathbb{Y}^{(\ell)}}    \left( \sum_{r = 1}^{\beta_{\ell j}} f_m(y_r^{(j)}) \right) m_\ell\left(\vec{y} \, |\, \vec{x} \right)\, d \vec{y}\right)\,\lambda^{(\ell)}[\mu^{\vec{\zeta}}_{s-}](d\vec{x}) \, ds \nonumber\\
&
\leq    \int_{0}^{t} \la \la K_\ell(x, y)\left( \int_{\R^{d}} f_m(z) m_\ell(z \, | \, x, y) dz \right) , \mu_{s-}^{{\vec{\zeta}}, i}(dx) \ra, \mu_{s-}^{{\vec{\zeta}}, k}(dy) \ra \, ds.\nonumber\\
&
 \leq C(K)  \int_{0}^{t} \la \la \left( \int_{\R^{d}} f_m(z) \sum_{i = 1}^I p_i\delta_{\alpha_i x + (1-\alpha_i)y}(z) dz \right) , \mu_{s-}^{{\vec{\zeta}}, i}(dx) \ra, \mu_{s-}^{{\vec{\zeta}}, k}(dy) \ra \, ds.\nonumber\\
&
 \leq  C(K)  \int_{0}^{t}\la \la \sum_{i = 1}^I p_i(f_{m-1}(x) + f_{m-1}(y)) , \mu_{s-}^{{\vec{\zeta}}, i}(dx) \ra, \mu_{s-}^{{\vec{\zeta}}, k}(dy) \ra \, ds.\nonumber\\
&
 \leq C(K)C(\mu)\int_{0}^{t} \left( \la f_{m-1}(x), \mu_{s-}^{{\vec{\zeta}}, i}(dx) \ra + \la f_{m-1}(y), \mu_{s-}^{{\vec{\zeta}}, k}(dy) \ra \right)ds. \nonumber
\end{align}
\end{proof}


\begin{lemma}\label{lem:MassControlTwo2Two}
If the $\ell$-th reaction is a reaction of the form $S_i + S_k \rightarrow S_j + S_r$, where $i, k, r$ could be $j$, then
\begin{align}
&\int_{0}^{t} \int_{\tilde{\mathbb{X}}^{(\ell)}}     \frac{1}{\vec{\alpha}^{(\ell)}!}   K_\ell\left(\vec{x}\right) \left(\int_{\mathbb{Y}^{(\ell)}}    \left( \sum_{r = 1}^{\beta_{\ell j}} f_m(y_r^{(j)}) \right) m_\ell\left(\vec{y} \, |\, \vec{x} \right)\, d \vec{y}\right)\,\lambda^{(\ell)}[\mu^{\vec{\zeta}}_{s-}](d\vec{x}) \, ds\nonumber\\
& \leq C(K)C(\mu)\int_{0}^{t}  \left(\la f_{m}(x), \mu_{s-}^{{\vec{\zeta}}, i}(dx) \ra + \la f_{m}(y), \mu_{s-}^{{\vec{\zeta}}, k}(dy) \ra \right)ds.\nonumber
\end{align}
\end{lemma}

\begin{proof}

By plugging in the specific form of the reaction rate and placement density as in Assumption \ref{Assume:measureTwo2Two}, we have
\begin{align}
&\int_{0}^{t} \int_{\tilde{\mathbb{X}}^{(\ell)}}     \frac{1}{\vec{\alpha}^{(\ell)}!}   K_\ell\left(\vec{x}\right) \left(\int_{\mathbb{Y}^{(\ell)}}    \left( \sum_{r = 1}^{\beta_{\ell j}} f_m(y_r^{(j)}) \right) m_\ell\left(\vec{y} \, |\, \vec{x} \right)\, d \vec{y}\right)\,\lambda^{(\ell)}[\mu^{\vec{\zeta}}_{s-}](d\vec{x}) \, ds \nonumber\\
&
\leq    \int_{0}^{t} \la \la K_\ell(x, y)\left( \int_{\R^{2d}} \left( f_m(z) + f_m(w)\right) m_\ell(z, w \, | \, x, y) dz\, dw \right) , \mu_{s-}^{{\vec{\zeta}}, i}(dx) \ra, \mu_{s-}^{{\vec{\zeta}}, k}(dy) \ra \, ds \nonumber\\
&
 \leq C(K)  \int_{0}^{t} \la \la \left( \int_{\R^{2d}}  \left( f_m(z) + f_m(w)\right) \left( p\times\delta_{(x, y)}\left((z, w)\right)  + (1-p)\times\delta_{(x, y)}\left((w, z)\right) \right) dz\, dw \right) , \mu_{s-}^{{\vec{\zeta}}, i}(dx) \ra, \mu_{s-}^{{\vec{\zeta}}, k}(dy) \ra \, ds \nonumber\\
&
 \leq  C(K)  \int_{0}^{t}\la \la  \left( f_m(x) + f_m(y)\right), \mu_{s-}^{{\vec{\zeta}}, i}(dx) \ra, \mu_{s-}^{{\vec{\zeta}}, k}(dy) \ra \, ds.\nonumber\\
&
 \leq C(K)C(\mu)\int_{0}^{t} \left( \la f_{m}(x), \mu_{s-}^{{\vec{\zeta}}, i}(dx) \ra + \la f_{m}(y), \mu_{s-}^{{\vec{\zeta}}, k}(dy) \ra \right)ds. \nonumber
\end{align}
\end{proof}


Now we are in position to give the proof of Lemma \ref{lem:MassControlMollifier}.
\begin{proof}[Proof of Lemma \ref{lem:MassControlMollifier}]
We have that
\begin{align*}
\avg{&\sup_{t\in[0, T]}\int_{0}^{t} \int_{\tilde{\mathbb{X}}^{(\ell)}}     \frac{1}{\vec{\alpha}^{(\ell)}!}   K_\ell\left(\vec{x}\right) \left(\int_{\mathbb{Y}^{(\ell)}}    \left( \sum_{r = 1}^{\beta_{\ell j}} f_m(y_r^{(j)}) \right) m^\eta_\ell\left(\vec{y} \, |\, \vec{x} \right)\, d \vec{y}\right)\,\lambda^{(\ell)}[\mu^{\vec{\zeta}}_{s-}](d\vec{x}) \, ds }\nonumber\\
&
= \avg{\sup_{t\in[0, T]} \int_{0}^{t} \int_{\tilde{\mathbb{X}}^{(\ell)}}     \frac{1}{\vec{\alpha}^{(\ell)}!}   K_\ell\left(\vec{x}\right) \left(\int_{\mathbb{Y}^{(\ell)}}    \left( \sum_{r = 1}^{\beta_{\ell j}} f_m(y_r^{(j)}) \right) m_\ell\left(\vec{y} \, |\, \vec{x} \right)\, d \vec{y}\right)\,\lambda^{(\ell)}[\mu^{\vec{\zeta}}_{s-}](d\vec{x}) \, ds} \nonumber\\
&
\quad + \avg{\sup_{t\in[0, T]} \int_{0}^{t} \int_{\tilde{\mathbb{X}}^{(\ell)}}     \frac{1}{\vec{\alpha}^{(\ell)}!}   K_\ell\left(\vec{x}\right) \left(\int_{\mathbb{Y}^{(\ell)}}    \left( \sum_{r = 1}^{\beta_{\ell j}} f_m(y_r^{(j)}) \right) \left(m^\eta_\ell\left(\vec{y} \, |\, \vec{x} \right) - m_\ell\left(\vec{y} \, |\, \vec{x} \right) \right)\, d \vec{y}\right)\,\lambda^{(\ell)}[\mu^{\vec{\zeta}}_{s-}](d\vec{x}) \, ds} \nonumber\\
&
\leq \avg{\sup_{t\in[0, T]}\int_{0}^{t} \int_{\tilde{\mathbb{X}}^{(\ell)}}     \frac{1}{\vec{\alpha}^{(\ell)}!}   K_\ell\left(\vec{x}\right) \left(\int_{\mathbb{Y}^{(\ell)}}    \left( \sum_{r = 1}^{\beta_{\ell j}} f_m(y_r^{(j)}) \right) m_\ell\left(\vec{y} \, |\, \vec{x} \right)\, d \vec{y}\right)\,\lambda^{(\ell)}[\mu^{\vec{\zeta}}_{s-}](d\vec{x}) \, ds} \nonumber\\
&
\quad + 2C(K)C(\mu) ||f_m||_{C_b^2(\R^d)} C\eta T \qquad \text{(by Lemma \ref{lem:etaConv})}\nonumber\\
&
\leq  2C(K)(C(\mu)\vee 1) \sup_{1\leq i\leq J}\avg{ \sup_{t\in[0, T]}\int_{0}^{t}   \la f_{m-1-R}(x) , \mu_{s-}^{{\vec{\zeta}}, i}(dx) \ra  \, ds} \nonumber\\
&
\quad  + 2C(K)  ||f_m||_{C_b^2(\R^d)}C(\mu)T\epsilon+ 2C(K)C(\mu) ||f_m||_{C_b^2(\R^d)} C\eta T  \qquad \text{(by Lemma \ref{lem:MassControlOne2One} - \ref{lem:MassControlTwo2Two})}\nonumber\\
&
=  2C(K)(C(\mu)\vee 1) \sup_{1\leq i\leq J} \avg{\sup_{t\in[0, T]}\int_{0}^{t}  \la f_{m-1-R}(x) , \mu_{s-}^{{\vec{\zeta}}, i}(dx) \ra  \, ds }  +2C(K)C(\mu) ||f_m||_{C_b^2(\R^d)}T\left(C\eta + \epsilon\right),   \nonumber
\end{align*}
concluding the proof of the lemma.
\end{proof}


\section{Proofs of Propositions \ref{P:EquivalentGenerator} and \ref{P:ParticleSpatialDensity}.}\label{A:ProofsEquivalentDerivations}

Let us recall the forward equation (\ref{eq:multipartABtoCEqs}).  For both proofs of Propositions \ref{P:EquivalentGenerator} and \ref{P:ParticleSpatialDensity} we need to define an appropriate $L^2$ space. In particular, define an appropriate $L^2$ (Fock) space, $F$,  with inner product for two functions, $\vG_1 = \{\gabc_{1}(\vqa,\vqb, \vqc)\}_{a,b,c=0}^{\infty}$ and $\vG_2 = \{\gabc_{2}(\vqa,\vqb,\vqc)\}_{a,b,c=0}^{\infty}$, as
\begin{equation}\label{eq:FockInnerProduct}
    \paren{\vG_{1}, \vG_{2}}_F = \sum_{a=0}^{\infty} \sum_{b=0}^{\infty}  \sum_{c=0}^{\infty} \frac{1}{a! \, b! \, c!} \int_{\R^{(a+b+c)d}} \gabc_{1}(\vqa,\vqb, \vqc) \gabc_{2}(\vqa,\vqb,\vqc) \, d\vqa \, d\vqb \, d\vqc,
\end{equation}
we can interpret $\mathcal{T}^* = \diffop + \Rp + \Rm$ as the adjoint of the generator, $\mathcal{T}$, for the process $$(\vQ^{A(t)}(t), \vQ^{B(t)}(t), \vQ^{C(t)}(t), A(t), B(t), C(t)).$$
Formally, we find
\begin{align}
    \paren{\mathcal{T} \vG}_{a,b,c}(\vqa,\vqb,\vqc) ={}&  \paren{\diffop \vG}_{a,b,c}(\vq^a,\vq^b,\vq^c)  \nonumber\\
    & \quad +\sum_{l=1}^a \sum_{m=1}^b \paren{ \int_{\R^d} m^\eta_1\paren{\vz \vert \vqa_{l}, \vqb_{m} } K_1^{\gamma}(\vqa_{l}, \vqb_{m} )g^{(a-1, b-1, c+1)}(\vqa  \setminus\vqa_{l}, \vqb  \setminus \vqb_{m} , \vqc \cup \vz) d\vz \right.\nonumber\\
    &\quad \quad \quad  - \left. K_1^{\gamma} \paren{\vqa_{l}, \vqb_{m}} \gabc(\vqa,\vqb,\vqc) }\nonumber\\
     & \quad + \sum_{n=1}^{c} \paren{ \int_{\R^{2d}} m^\eta_2\paren{\vx, \vy \vert \vqc_{n} }K_2^{\gamma}(\vqc_{n})g^{(a+1, b+1, c-1)}(\vqa \cup\vx, \vqb  \cup \vy, \vqc \setminus \vqc_n, t) d\vx d\vy \right.\nonumber\\
    &\quad \quad \quad  - \left.  K_2^{\gamma}(\vqc_n) \gabc(\vqa, \vqb, \vqc)}\label{eq:operT}
\end{align}
Note that here $\diffop^* = \diffop$.

Next we present the proofs of the two propositions.

\begin{proof}[Proof of Proposition \ref{P:EquivalentGenerator}]
For simplicity of notation, without loss of generality, we will show the equivalence for the evolution of  $\avg{\varphi\left(\la f, \nu^{\vec{\zeta}}_t \ra_{\hat{P}}\right)}$. The same procedure follows for the more general multi-dimensional case $\avg{\varphi\left(\la f_1, \nu^{\vec{\zeta}}_t \ra_{\hat{P}}, \la f_2, \nu^{\vec{\zeta}}_t \ra_{\hat{P}}, \cdots, \la f_M, \nu^{\vec{\zeta}}_t \ra_{\hat{P}} \right)}$. By the definition of $\nu^{\vec{\zeta}}_t$ and adopting the notation of this section,

\begin{align}\label{Eq:mean_phi_inner_product}
\avg{\varphi\left(\la f, \nu^{\vec{\zeta}}_t \ra_{\hat{P}} \right)} &= \avg{\varphi\left( \sum_{i = 1}^{A(t)} f(\vQ^{A(t)}_i(t), S_1) + \sum_{j = 1}^{B(t)} f(\vQ^{B(t)}_j(t), S_2) + \sum_{k = 1}^{C(t)} f(\vQ^{C(t)}_k(t), S_3)\right)} \nonumber\\
&
=\sum_{a=0}^{\infty} \sum_{b=0}^{\infty}  \sum_{c=0}^{\infty} \frac{1}{a! \, b! \, c!} \int_{\R^{(a+b+c)d}} \varphi \left(\sum_{i = 1}^{a} f(\vqa_i, S_1) + \sum_{j = 1}^{b} f(\vqb_j, S_2) + \sum_{k = 1}^{c} f(\vqc_k, S_3)  \right) \nonumber\\
&\qquad
\times\pabc(\vqa,\vqb, \vqc, t) \, d\vqa \, d\vqb \, d\vqc, \nonumber\\
&
=    \paren{\vG, \, \vP(t)}_F,
\end{align}
where we define $\gabc (\vqa,\vqb,\vqc)  = \varphi\left( \sum_{i = 1}^a f(\vqa_i, S_1) + \sum_{j = 1}^b f(\vqb_j, S_2) +\sum_{k = 1}^c f(\vqc_k, S_3)\right)$.
For such a form of  $\gabc (\vqa,\vqb,\vqc)$, plugging into Eq (\ref{eq:operT}), we have
\begin{align}
    &\paren{\mathcal{T} \vG}_{a,b,c} (\vqa,\vqb,\vqc) =  \paren{D_1 \sum_{l=1}^{a} \lap_{\vqa_l}
    + D_2 \sum_{m=1}^{b} \lap_{\vqb_m}
    + D_3 \sum_{n=1}^{c} \lap_{\vqc_n}} \varphi\paren{\sum_{i = 1}^a f(\vqa_i, S_1) + \sum_{j = 1}^b f(\vqb_j, S_2) +\sum_{k = 1}^c f(\vqc_k, S_3)} \nonumber\\
    &
     \quad +\sum_{l=1}^a \sum_{m=1}^b \paren{ \int_{\R^d} \left[g^{(a-1, b-1, c+1)}(\vqa  \setminus\vqa_{l}, \vqb  \setminus \vqb_{m} , \vqc \cup \vz) - \gabc(\vqa,\vqb,\vqc) \right]\right.\nonumber\\
    &
    \quad \quad \quad \left. \times m^\eta_1\paren{\vz \vert \vqa_{l}, \vqb_{m} } \times K_1^{\gamma}(\vqa_{l}, \vqb_{m} )d\vz  }\nonumber\\
     &
      \quad + \sum_{n=1}^{c} \paren{ \int_{\R^{2d}} \left[g^{(a+1, b+1, c-1)}(\vqa \cup\vx, \vqb  \cup \vy, \vqc \setminus \vqc_n, t) -  \gabc(\vqa, \vqb, \vqc) \right] \right.\nonumber\\
    &
    \quad \quad \left.\times m^\eta_2\paren{\vx, \vy \vert \vqc_{n} }\times K_2^{\gamma}(\vqc_{n}) d\vx d\vy}\nonumber\\
    &
    = D_1  \sum_{i = 1}^a  \varphi''\paren{\sum_{i = 1}^a f(\vqa_i, S_1) + \sum_{j = 1}^b f(\vqb_j, S_2) +\sum_{k = 1}^c f(\vqc_k, S_3)}|\nabla_{\vqa_i} f_1(\vqa_i)|^2 \nonumber\\
    &
    \quad+ D_1  \sum_{i = 1}^a  \varphi'\paren{\sum_{i = 1}^a f(\vqa_i, S_1) + \sum_{j = 1}^b f(\vqb_j, S_2) +\sum_{k = 1}^c f(\vqc_k, S_3)}\lap_{\vqa_i} f_1(\vqa_i) \nonumber\\
    &
    \quad +  D_2  \sum_{j = 1}^b  \varphi''\paren{\sum_{i= 1}^a f(\vqa_i, S_1) + \sum_{j = 1}^b f(\vqb_j, S_2) +\sum_{k = 1}^c f(\vqc_k, S_3)}|\nabla_{\vqb_j} f_2(\vqb_j)|^2 \nonumber\\
    &
    \quad+ D_2  \sum_{j = 1}^b  \varphi'\paren{\sum_{i = 1}^a f(\vqa_i, S_1) + \sum_{j = 1}^b f(\vqb_j, S_2) +\sum_{k = 1}^c f(\vqc_k, S_3)}\lap_{\vqb_j} f_2(\vqb_j) \nonumber\\
        &
    \quad +  D_3  \sum_{k = 1}^c  \varphi''\paren{\sum_{i= 1}^a f(\vqa_i, S_1) + \sum_{j = 1}^b f(\vqb_j, S_2) +\sum_{k = 1}^c f(\vqc_k, S_3)}|\nabla_{\vqc_k} f_3(\vqc_k)|^2 \nonumber\\
    &
    \quad+ D_3  \sum_{k= 1}^c  \varphi'\paren{\sum_{i = 1}^a f(\vqa_i, S_1) + \sum_{j = 1}^b f(\vqb_j, S_2) +\sum_{k = 1}^c f(\vqc_k, S_3)}\lap_{\vqc_k} f_3(\vqc_k) \nonumber\\
    &
     \quad +\sum_{l=1}^a \sum_{m=1}^b \paren{ \int_{\R^d} \left[\varphi \paren{\sum_{i = 1}^a f(\vqa_i, S_1) + \sum_{j = 1}^b f(\vqb_j, S_2) +\sum_{k = 1}^c f(\vqc_k, S_3) - f(\vqa_{l} , S_1) - f(\vqb_{m}, S_2) + f(\vz, S_3)}  \right. \right.\nonumber\\
    &
    \quad \quad \quad \left.  \left.  - \varphi\paren{\sum_{i = 1}^a f(\vqa_i, S_1) + \sum_{j = 1}^b f(\vqb_j, S_2) +\sum_{k = 1}^c f(\vqc_k, S_3)}\right]   \times m^\eta_1\paren{\vz \vert \vqa_{l}, \vqb_{m} } \times K_1^{\gamma}(\vqa_{l}, \vqb_{m} )d\vz  }\nonumber\\
     &
      \quad + \sum_{n=1}^{c} \paren{ \int_{\R^{2d}} \left[\varphi \paren{\sum_{i = 1}^a f(\vqa_i, S_1) + \sum_{j = 1}^b f(\vqb_j, S_2) +\sum_{k = 1}^c f(\vqc_k, S_3) + f(\vx, S_1) + f(\vy, S_2) - f(\vqc_n, S_3)}\right. \right.\nonumber\\
    &
    \quad \quad \left. \left. -   \varphi\paren{\sum_{i = 1}^a f(\vqa_i, S_1) + \sum_{j = 1}^b f(\vqb_j, S_2) +\sum_{k = 1}^c f(\vqc_k, S_3)}\right]  \times m^\eta_2\paren{\vx, \vy \vert \vqc_{n} }\times K_2^{\gamma}(\vqc_{n}) d\vx d\vy}\nonumber\\\label{eq:operatorT}
\end{align}

Rewriting $\avg{\varphi\paren{\la f, \nu^{\vec{\zeta}}_t \ra_{\hat{P}}}}$ in Eq (\ref{Eq:mean_phi_inner_product}) as an integral equation, we obtain
\begin{align}\label{Eq:generator_forward_eqABC}
&\avg{\varphi\paren{\la f, \nu^{\vec{\zeta}}_t \ra_{\hat{P}}}}  =   \paren{\vG, \, \vP(t)}_F
= \paren{\vG, \, \vP(0)}_F + \int_0^t \partial_s\paren{\vG, \, \vP(s)}_F\, ds \nonumber\\
&
= \avg{\varphi\paren{\la f, \nu^{\vec{\zeta}}_0 \ra_{\hat{P}}}} + \int_0^t \paren{\vG, \, \partial_s\vP(s)}_F\, ds \nonumber\\
&
= \avg{\varphi\paren{\la f, \nu^{\vec{\zeta}}_0 \ra_{\hat{P}}}}  +  \int_0^t \paren{\vG, \, \mathcal{T}^*\vP(s)}_F\, ds \quad \text{(by Eq (\ref{eq:multipartABtoCEqs}))} \nonumber\\
&
=\avg{\varphi\paren{\la f, \nu^{\vec{\zeta}}_0 \ra_{\hat{P}}}}  +  \int_0^t \paren{\mathcal{T}\vG, \, \vP(s)}_F\, ds \nonumber\\
&
    = \avg{\varphi\paren{\la f, \nu^{\vec{\zeta}}_0 \ra_{\hat{P}}}}  \nonumber\\
    &
    \quad + \avg{ \int_0^t \varphi'\paren{\la f, \nu^{\vec{\zeta}}_s \ra_{\hat{P}}}\times\left[ D_1  \sum_{i = 1}^{A(s)}  \lap_{\vQ_i} f(\vQ^{A(s)}_i(s), S_1) +D_2 \sum_{j = 1}^{B(s)} \lap_{\vQ_j} f(\vQ^{B(s)}_j(s), S_2)\right. \nonumber\\
    &
    \qquad + \left.D_3\sum_{k = 1}^{C(s)} \lap_{\vQ_k} f(\vQ^{C(s)}_k(s), S_3)\right]  + \varphi''\paren{\la f, \nu^{\vec{\zeta}}_s \ra_{\hat{P}}}\times \left[ D_1 \sum_{i = 1}^{A(s)}  |\nabla_{\vQ_i} f(\vQ^{A(s)}_i(s), S_1)|^2  \right.\nonumber\\
    &
     \qquad+ \left. D_2 \sum_{j = 1}^{B(s)} |\nabla_{\vQ_j} f(\vQ^{B(s)}_j(s), S_2)|^2  + D_3\sum_{k = 1}^{C(s)} |\nabla_{\vQ_k} f(\vQ^{C(s)}_k(s), S_3)|^2 \right] \, ds}\nonumber\\
    &
     \quad +\avg{ \int_0^t\sum_{l=1}^{A(s)} \sum_{m=1}^{B(s)} \paren{ \int_{\R^d} K_1^{\gamma}(\vQ^{A(s)}_{l}(s), \vQ^{B(s)}_{m}(s) )m^\eta_1\paren{\vz \vert \vQ^{A(s)}_{l}(s), \vQ^{B(s)}_{m}(s) } \right.\nonumber\\
&
\qquad\left.\times \left(\varphi\paren{\la f, \nu^{\vec{\zeta}}_s \ra_{\hat{P}}-f(\vQ^{A(s)}_{l}(s), S_1) - f(\vQ^{B(s)}_{m}(s) , S_2) + f(\vz, S_3)}  - \varphi\paren{\la f, \nu^{\vec{\zeta}}_s \ra_{\hat{P}}} \right)d\vz } \, ds}\nonumber\\
     &
      \quad + \avg{ \int_0^t\sum_{n=1}^{C(s)} \paren{ \int_{\R^{2d}} K_2^{\gamma}( \vQ^{C(s)}_{n} (s))m^\eta_2\paren{\vx, \vy \vert \vQ^{C(s)}_{n} (s) } \right.\nonumber\\
 &
\qquad \left.\times\left( \varphi\paren{\la f, \nu^{\vec{\zeta}}_s \ra_{\hat{P}} + f(\vx, S_1 ) + f(\vy, S_2) - f( \vQ^{C(s)}_{n} (s), S_3)}  - \varphi\paren{\la f, \nu^{\vec{\zeta}}_s \ra_{\hat{P}}}\right)d\vx d\vy } \, ds}. \end{align}

From the measure-valued formulation in Eq (\ref{Eq:EM_formula_ABC}), we obtain the following integral equation for $\avg{\varphi\paren{\la f, \nu^{\vec{\zeta}}}_t \ra_{\hat{P}}} $ by applying It\^{o}'s formula on Eq (\ref{Eq:EM_formula_ABC}),
\begin{align}\label{Eq:generator_measure_valuedABC}
&\avg{\varphi\paren{\la f, \nu^{\vec{\zeta}}_t \ra_{\hat{P}}}}  =   \avg{\varphi\paren{\la f, \nu^{\vec{\zeta}}_0 \ra_{\hat{P}}}}  + \avg{\int_0^t  \varphi'\paren{\la f, \nu^{\vec{\zeta}}_{s} \ra_{\hat{P}}} \sum_{i = 1}^{\la 1, \nu^{\vec{\zeta}}_{s-}\ra_{\hat{P}}} D^i \frac{\partial^2 f}{\partial Q^2}\left(H^i(\nu^{\vec{\zeta}}_{s-}) \right) \, ds} \nonumber\\
&
 \quad + \frac{1}{2} \avg{\int_0^t  \varphi''\paren{\la f, \nu^{\vec{\zeta}}_{s} \ra_{\hat{P}}} \sum_{i = 1}^{\la 1, \nu^{\vec{\zeta}}_{s-}\ra_{\hat{P}}} \left(\sqrt{2D^i} \frac{\partial f}{\partial Q}\left(H^i(\nu^{\vec{\zeta}}_{s-}) \right) \right)^2\, ds} \nonumber\\
    &
     \quad +\avg{ \int_0^t \int_{\left(\N\setminus\{0\}\right)^{2}} \int_{\R^d}\int_{\R_+^2} \varphi\paren{\la f, \nu^{\vec{\zeta}}_{s} \ra_{\hat{P}} + \paren{-f(H_Q^{i}(\nu^{{\vec{\zeta}},1}_{s-}), S_1) - f(H_Q^{j}(\nu^{{\vec{\zeta}},2}_{s-}), S_2) + f(z, S_3)} \right. \nonumber\\
     &
     \qquad\qquad \left. \times 1_{\{i\leq \la 1,\nu^{{\vec{\zeta}},1}_{s-}\ra\}}\times 1_{\{j\leq \la 1,\nu^{{\vec{\zeta}},2}_{s-}\ra\}}\times 1_{\{ \theta_1 \leq K_{1}^{\gamma}(H_Q^{i}(\nu^{{\vec{\zeta}},1}_{s-}), \, H_Q^{j}(\nu^{{\vec{\zeta}},2}_{s-})) \}} \times 1_{\{ \theta_2 \leq  m^\eta_{1}(z|H_Q^{i}(\nu^{{\vec{\zeta}},1}_{s-}),H_Q^{j}(\nu^{{\vec{\zeta}},2}_{s-}))\}} }  \nonumber\\
     &
     \qquad- \varphi\paren{\la f, \nu^{\vec{\zeta}}_{s} \ra_{\hat{P}}} dN_{1}(s,i,j,z,\theta_1, \theta_2) }\nonumber\\
     &
     \quad +\avg{ \int_0^t \int_{\N\setminus\{0\}} \int_{\R^{2d}}\int_{\R_+^2} \varphi\paren{\la f, \nu^{\vec{\zeta}}_{s} \ra_{\hat{P}} + \paren{f(x, S_1) + f(y, S_2) - f(H_Q^{k}(\nu^{{\vec{\zeta}},3}_{s-}), S_3)} \right. \nonumber\\
     &
     \qquad\qquad \left.\times 1_{\{k\leq \la 1,\nu^{{\vec{\zeta}},3}_{s-}\ra\}}\times  1_{\{ \theta_1 \leq K_{2}^{\gamma}(H_Q^{k}(\nu^{{\vec{\zeta}},3}_{s-}))\}}\times  1_{\{ \theta_2 \leq  m^\eta_{2}(x,y|H_Q^{k}(\nu^{{\vec{\zeta}},3}_{s-}))\}} }  \nonumber\\
     &
     \qquad- \varphi\paren{\la f, \nu^{\vec{\zeta}}_{s} \ra_{\hat{P}}} dN_{2}(s,k,x,y,\theta_1, \theta_2) }\nonumber\\
     &
      =   \avg{\varphi\paren{\la f, \nu^{\vec{\zeta}}_0 \ra_{\hat{P}}}}  + \avg{\int_0^t  \varphi'\paren{\la f, \nu^{\vec{\zeta}}_{s} \ra_{\hat{P}}} \sum_{i = 1}^{\la 1, \nu^{\vec{\zeta}}_{s-}\ra_{\hat{P}}} D^i \frac{\partial^2 f}{\partial Q^2}\left(H^i(\nu^{\vec{\zeta}}_{s-}) \right) \, ds} \nonumber\\
&
 \quad + \avg{\int_0^t  \varphi''\paren{\la f, \nu^{\vec{\zeta}}_{s} \ra_{\hat{P}}} \sum_{i = 1}^{\la 1, \nu^{\vec{\zeta}}_{s-}\ra_{\hat{P}}} D^i\left( \frac{\partial f}{\partial Q}\left(H^i(\nu^{\vec{\zeta}}_{s-}) \right) \right)^2\, ds} \nonumber\\
    &
     \quad +\avg{ \int_0^t \sum_{ i = 1}^{\la 1,\nu^{{\vec{\zeta}},1}_{s-}\ra} \sum_{j = 1}^{\la 1,\nu^{{\vec{\zeta}},2}_{s-}\ra} \int_{\R^d} K_{1}^{\gamma}(H_Q^{i}(\nu^{{\vec{\zeta}},1}_{s-}), \, H_Q^{j}(\nu^{{\vec{\zeta}},2}_{s-}))\times m^\eta_{1}(z|H_Q^{i}(\nu^{{\vec{\zeta}},1}_{s-}),H_Q^{j}(\nu^{{\vec{\zeta}},2}_{s-})) \nonumber\\
     &\qquad \times\left[ \varphi\paren{\la f, \nu^{\vec{\zeta}}_{s} \ra_{\hat{P}} + \paren{-f(H_Q^{i}(\nu^{{\vec{\zeta}},1}_{s-}), S_1) - f(H_Q^{j}(\nu^{{\vec{\zeta}},2}_{s-}), S_2) + f(z, S_3)} } -   \varphi\paren{\la f, \nu^{\vec{\zeta}}_{s} \ra_{\hat{P}}} \right]\, dz\, ds }\nonumber\\
     &
     \quad +\avg{ \int_0^t \sum_{k = 1}^{\la 1,\nu^{{\vec{\zeta}},3}_{s-}\ra} \int_{\R^{2d}} K_{2}^{\gamma}(H_Q^{k}(\nu^{{\vec{\zeta}},3}_{s-}))\times m^\eta_{2}(x,y|H_Q^{k}(\nu^{{\vec{\zeta}},3}_{s-}))\nonumber\\
     &
     \qquad\times\left[ \varphi\paren{\la f, \nu^{\vec{\zeta}}_{s} \ra_{\hat{P}} + \paren{f(x, S_1) + f(y, S_2) - f(H_Q^{k}(\nu^{{\vec{\zeta}},3}_{s-}), S_3)} } - \varphi\paren{\la f, \nu^{\vec{\zeta}}_{s} \ra_{\hat{P}}}\right] \, dx\, dy\, ds }.
\end{align}

We observe that the integral equation for $\avg{\varphi\paren{\la f, \nu^{\vec{\zeta}}}_t \ra_{\hat{P}}} $  is the same when derived from the forward equation (\ref{Eq:generator_forward_eqABC}) and from the measure-valued formulation (\ref{Eq:generator_measure_valuedABC}).
\end{proof}


\begin{proof}[Proof of Proposition \ref{P:ParticleSpatialDensity}]

We use the forward Kolmogorov equation to prove the proposition. We are interested in finding an equation for the average concentration field for \textrm{A}, \textrm{B}  and \textrm{C} molecules, i.e. $\avg{A(\vx,t)}$, $\avg{B(\vy,t)}$ and $\avg{C(\vz,t)}$ from forward equation. This is defined by
\begin{align*}
    \avg{A(\vx,t)} &= \avg{ \sum_{i=1}^{A(t)} \delta\paren{\vQ^{A(t)}_i(t)-\vx} } \\
    &= \sum_{a=0}^{\infty} \sum_{b=0}^{\infty} \sum_{c=0}^{\infty}  \frac{1}{a! \, b! \, c!} \sum_{i=1}^a \int_{\R^{(a+b+c)d}} \delta(\vqa_i - \vx) \, \pabc(\vqa,\vqb,\vqc, t) \, d\vqa \, d\vqb \, d\vqc \\
    &= \sum_{a=0}^{\infty} \sum_{b=0}^{\infty}  \sum_{c=0}^{\infty}  \frac{1}{(a-1)! \, b!\, c!} \int_{\R^{(a-1+b + c)d}} \pabc(\vq^{a-1}\cup\vx,\vqb,\vqc,t) \, d\vq^{a-1} \, d\vqb \, d\vqc.
\end{align*}
Similarly, for molecule B and C, we have
\begin{align*}
    \avg{B(\vy,t)} &= \sum_{a=0}^{\infty} \sum_{b=0}^{\infty}  \sum_{c=0}^{\infty}  \frac{1}{a! \, (b-1)!\, c!} \int_{\R^{(a+b-1 + c)d}} \pabc(\vqa, \vq^{b-1}\cup\vy,\vqc,t) \, d\vqa \, d\vq^{b-1} \, d\vqc.\\
        \avg{C(\vz,t)} &= \sum_{a=0}^{\infty} \sum_{b=0}^{\infty}  \sum_{c=0}^{\infty}  \frac{1}{a! \, b!\, (c-1)!} \int_{\R^{(a+b + c-1)d}} \pabc(\vqa, \vqb , \vq^{c-1}\cup\vz,t) \, d\vqa \, d\vqb \, d\vq^{c-1}.
\end{align*}
In deriving this equation we will need to use the correlation in the \textrm{A} and \textrm{B} fields, given by
\begin{align*}
    &\avg{A(\vx,t) B(\vy,t)} = \avg{ \sum_{i=1}^{A(t)} \sum_{j=1}^{B(t)} \delta\paren{\vQ^{A(t)}_i(t)-\vx} \delta\paren{\vQ^{B(t)}_j(t)-\vy} } \\
    &= \sum_{a=0}^{\infty} \sum_{b=0}^{\infty} \sum_{c=0}^{\infty}  \frac{1}{a! \, b! c!} \sum_{i=1}^a \sum_{j=1}^b \int_{\R^{(a+b)d}} \delta(\vqa_i - \vx) \delta(\vqb_j - \vy) \, \pabc(\vqa,\vqb,\vqc,t) \, d\vqa \, d\vqb \, d\vqc \\
    &= \sum_{a=0}^{\infty} \sum_{b=0}^{\infty} \sum_{c=0}^{\infty}  \frac{1}{(a-1)! \, (b-1)! c!} \int_{\R^{(a+b +c-2)d}} \pabc(\vq^{a-1}\cup\vx,\vq^{b-1}\cup\vy,\vqc,t) \, d\vq^{a-1} \, d\vq^{b-1} \, d\vqc.
\end{align*}
Using these definitions, and assuming the probability densities vanish at infinity, we find that
\begin{align*}
    \partial_t \avg{A(\vx,t)} &= \sum_{a=0}^{\infty} \sum_{b=0}^{\infty}  \sum_{c=0}^{\infty}  \frac{1}{(a-1)! \, b!\, c!} \int_{\R^{(a-1+b + c)d}} \paren{(\diffop + \Rp + \Rm)\vP}_{a,b,c}(\vq^{a-1}\cup\vx,\vqb,\vqc,t) \, d\vq^{a-1} \, d\vqb \, d\vqc.\\
    &= (\RN{1})+ (\RN{2})+ (\RN{3}),
 \end{align*}
where
\begin{align*}
    (\RN{1})&= \sum_{a=0}^{\infty} \sum_{b=0}^{\infty}  \sum_{c=0}^{\infty}  \frac{1}{(a-1)! \, b!\, c!} \int_{\R^{(a-1+b + c)d}} \paren{\diffop \vP}_{a,b,c}(\vq^{a-1}\cup\vx,\vqb,\vqc,t) \, d\vq^{a-1} \, d\vqb \, d\vqc\\
    &= \sum_{a=0}^{\infty} \sum_{b=0}^{\infty}  \sum_{c=0}^{\infty}  \frac{1}{(a-1)! \, b!\, c!} \int_{\R^{(a-1+b + c)d}} \Bigg[ \paren{D_1 \sum_{l=1}^{a-1} \lap_{\vqa_l}
    + D_2 \sum_{m=1}^{b} \lap_{\vqb_m}
    + D_3 \sum_{n=1}^{c} \lap_{\vqc_n}}\pabc(\vq^{a-1}\cup\vx,\vqb,\vqc,t) \\
   &\quad + D_1 \lap_{\vx}\pabc(\vq^{a-1}\cup\vx,\vqb,\vqc,t) \Bigg] \, d\vq^{a-1} \, d\vqb \, d\vqc\\
   &  = D_1 \lap_{\vx} \avg{A(\vx,t)}
\end{align*}
where on the second to last line, the first term becomes zero due to integration by parts and the fact that probability density vanishes at infinity (recall that by Theorem~\ref{thm:regularityABC} $p\in \mathcal{C}([0,\infty);H^2(X))$). Similarly,
\begin{align*}
    &(\RN{2})= \sum_{a=0}^{\infty} \sum_{b=0}^{\infty}  \sum_{c=0}^{\infty}  \frac{1}{(a-1)! \, b!\, c!} \int_{\R^{(a-1+b + c)d}}  \paren{\Rp \vP}_{a,b,c}(\vq^{a-1}\cup\vx,\vqb,\vqc,t) \, d\vq^{a-1} \, d\vqb \, d\vqc.\\
    &= \sum_{a=0}^{\infty} \sum_{b=0}^{\infty}  \sum_{c=0}^{\infty}  \frac{1}{(a-1)! \, b!\, c!} \int_{\R^{(a-1+b + c)d}}  -\sum_{m=1}^b K_1^{\gamma} \paren{\vx, \vqb_{m}}\pabc(\vq^{a-1}\cup \vx,\vqb,\vqc,t) \, d\vq^{a-1} \, d\vqb \, d\vqc.\\
    &\quad+\sum_{a=0}^{\infty} \sum_{b=0}^{\infty}  \sum_{c=0}^{\infty}  \frac{1}{(a-1)! \, b!\, c!} \int_{\R^{(a-1+b + c)d}} \left\{ -\sum_{l=1}^{a-1} \sum_{m=1}^b K_1^{\gamma} \paren{\vqa_{l}, \vqb_{m}} \pabc(\vq^{a-1}\cup \vx,\vqb,\vqc,t)\right. \\
    &\quad \quad+ \left. \sum_{n=1}^{c} \!\brac{\int_{\R^{2d}} m_1(\vqc_n \vert \vxi, \veta) K_1^{\gamma}( \vxi, \veta)
      p^{(a+1,b+1,c-1)}(\vq^{a-1}\cup \vx \cup \vxi, \vqb \cup \veta, \vqc \setminus \vqc_n, t) d\vxi d\veta}\!\! \right\} \, d\vq^{a-1} \, d\vqb \, d\vqc.\\
    &= -\sum_{a=0}^{\infty} \sum_{b=0}^{\infty}  \sum_{c=0}^{\infty}  \frac{1}{(a-1)! \, b!\, c!} \int_{\R^{(a+b + c-2)d}}  \brac{b\times \int_{\R^{d}} K_1^{\gamma} \paren{\vx,\vy}\pabc(\vq^{a-1}\cup \vx,\vq^{b-1}\cup\vy,\vqc,t) \, d\vy}\,d\vq^{a-1} \, d\vq^{b-1} \, d\vqc.\\
    &\quad  +\sum_{a=0}^{\infty} \sum_{b=0}^{\infty}  \sum_{c=0}^{\infty}  \frac{1}{(a-1)! \, b!\, c!} \left\{ - \int_{\R^{(a-1+b + c)d}}  \sum_{l=1}^{a-1} \sum_{m=1}^b K_1^{\gamma} \paren{\vqa_{l}, \vqb_{m}} \pabc(\vq^{a-1}\cup \vx,\vqb,\vqc,t)d\vq^{a-1} \, d\vqb \, d\vqc \right. \\
    &\quad\quad +  \left. \int_{\R^{(a +b+ c)d}}\int_{\R^d}\! \frac{c}{a(b+1)} \sum_{l=1}^{a} \sum_{m=1}^{b+1} m_1(\vz \vert \vqa_{l}, \vqb_{m}) K_1^{\gamma}( \vqa_{l}, \vqb_{m})\times\right.\nonumber\\
    &\left. \hspace{7cm}  p^{(a+1,b+1,c-1)}(\vqa\cup \vx,  \vq^{b+1}, \vq^{c-1}, t) \, d\vz \, d\vq^{a} \, d\vq^{b+1} \, d\vq^{c-1} \right\}\\
   &= - \int_{\R^{d}} K_1^{\gamma} \paren{\vx,\vy}\avg{A(\vx,t) B(\vy,t)} \, d\vy\\
    &\quad  - \sum_{a=0}^{\infty} \sum_{b=0}^{\infty}  \sum_{c=0}^{\infty}  \frac{1}{(a-1)! \, b!\, c!} \int_{\R^{(a-1+b + c)d}}  \sum_{l=1}^{a-1} \sum_{m=1}^b K_1^{\gamma}\paren{\vqa_{l}, \vqb_{m}} \pabc(\vq^{a-1}\cup \vx,\vqb,\vqc,t)d\vq^{a-1} \, d\vqb \, d\vqc \\
    &\quad + \sum_{a=0}^{\infty} \sum_{b=0}^{\infty}  \sum_{c=0}^{\infty}  \frac{1}{a! \, (b+1)!\, (c-1)!} \int_{\R^{(a +b+ c)d}}\! \sum_{l=1}^{a} \sum_{m=1}^{b+1} K_1^{\gamma}(\vqa_{l}, \vqb_{m})
      p^{(a+1,b+1,c-1)}(\vqa\cup \vx,  \vq^{b+1}, \vq^{c-1}, t) \, d\vq^{a} \, d\vq^{b+1} \, d\vq^{c-1} \\
  &= - \int_{\R^{d}} K_1^{\gamma} \paren{\vx,\vy}\avg{A(\vx,t) B(\vy,t)} \, d\vy.
\end{align*}
In the third equality of $(\RN{2})$, we exchanged the orders of integrals and sums using that
\begin{align*}
&\int_{\R^{bd}}\sum_{m = 1}^b K_1^{\gamma} \paren{\vx, \vqb_{m}} \pabc(\vq^{a-1}\cup \vx,\vqb,\vqc,t) \, d\vqb\\
& = \sum_{m = 1}^b \int_{\R^{bd}} K_1^{\gamma} \paren{\vx, \vqb_{m}} \pabc(\vq^{a-1}\cup \vx,(\vqb\setminus\vqb_{m})\cup\vqb_{m} ,\vqc,t)\, d\vqb_{m} \, d(\vqb\setminus\vqb_{m})\\
& = \sum_{m = 1}^b \int_{\R^{bd}} K_1^{\gamma} \paren{\vx, \vy} \pabc(\vq^{a-1}\cup \vx,\vq^{b-1}\cup\vy ,\vqc,t)\, d\vy \, d\vq^{b-1}\\
&\quad \text{(Here we replace $\vqb_m$ by $\vy$ and $\vqb\setminus\vqb_{m}$ by $\vq^{b-1}$ due to  that particles of the same type are indistinguishable)}\\
& = b\times \int_{\R^{bd}} K_1^{\gamma} \paren{\vx, \vy} \pabc(\vq^{a-1}\cup \vx,\vq^{b-1}\cup\vy ,\vqc,t)\, d\vy \, d\vq^{b-1}.
\end{align*}
Similar ideas also apply to the third term and to deriving $(\RN{3})$. In the second to last equality of $(\RN{2})$,  we used that $\int_{\R^d} m_1(\vz \vert \vx , \vy) d\vz = 1$ and the second and third term cancel by shifting indexes. Similarly
\begin{align*}
    (\RN{3})&= \sum_{a=0}^{\infty} \sum_{b=0}^{\infty}  \sum_{c=0}^{\infty}  \frac{1}{(a-1)! \, b!\, c!} \int_{\R^{(a-1+b + c)d}}  \paren{\Rm\vP}_{a,b,c}(\vq^{a-1}\cup\vx,\vqb,\vqc,t) \, d\vq^{a-1} \, d\vqb \, d\vqc\\
    &= \sum_{a=0}^{\infty} \sum_{b=0}^{\infty}  \sum_{c=0}^{\infty}  \frac{1}{(a-1)! \, b!\, c!} \int_{\R^{(a-1+b + c)d}} \left\{ - \sum_{n=1}^{c} K_2^{\gamma}(\vqc_n) \pabc(\vq^{a-1}\cup\vx, \vqb, \vqc, t)\right. \\
    &\quad + \sum_{m=1}^b  \brac{\int_{\R^d} m_2\paren{\vx,\vqb_m\vert\vz}K_2^{\gamma}(\vz)
    p^{(a-1,b-1,c+1)}\paren{\vq^{a-1}\setminus \vq^{a-1}_l\cup\vx, \vqb\setminus \vqb_m, \vqc \cup \vz, t} d\vz}\\
  &\quad +\left. \sum_{m=1}^b  \sum_{l=1}^{a-1} \brac{\int_{\R^d} m_2\paren{\vq^{a-1}_l,\vqb_m\vert\vz}K_2^{\gamma}(\vz)
    p^{(a-1,b-1,c+1)}\paren{\vq^{a-1}\setminus \vq^{a-1}_l\cup\vx, \vqb\setminus \vqb_m, \vqc \cup \vz, t} d\vz}   \right\} \, d\vq^{a-1} \, d\vqb \, d\vqc\\
   &= -\sum_{a=0}^{\infty} \sum_{b=0}^{\infty}  \sum_{c=0}^{\infty}  \frac{1}{(a-1)! \, b!\, c!} \int_{\R^{(a+b + c-2)d}} c\int_{\R^d}K_2^{\gamma}(\vz) \pabc(\vq^{a-1}\cup\vx, \vqb, \vq^{c-1}\cup\vz, t)  \,d\vz\, d\vq^{a-1} \, d\vqb \, d\vq^{c-1}\\
  &\quad+\sum_{a=0}^{\infty} \sum_{b=0}^{\infty}  \sum_{c=0}^{\infty}  \frac{1}{(a-1)! \, b!\, c!}\left\{ \int_{\R^{(a+b + c-2)d}}b\brac{\int_{\R^{2d}} m_2\paren{\vx,\vy\vert\vz}K_2^{\gamma}(\vz)
    p^{(a-1,b-1,c+1)}\paren{\vq^{a-1}, \vq^{b-1}, \vqc \cup \vz, t} d\vy d\vz} \right. \\
    &\qquad\qquad d\vq^{a-1} \, d\vq^{b-1} \, d\vqc + \int_{\R^{(a+b + c-3)d}}(a-1)b\\
     &\qquad\qquad \left. \times\brac{\int_{\R^{3d}} m_2\paren{\vxi,\veta\vert\vz}K_2^{\gamma}(\vz)
    p^{(a-1,b-1,c+1)}\paren{\vq^{a-2}\cup\vx, \vq^{b-1}, \vqc \cup \vz, t} d\vxi\, d\veta\, d\vz}  \, d\vq^{a-2} \, d\vq^{b-1} \, d\vqc\right\}\\
    &= -\sum_{a=0}^{\infty} \sum_{b=0}^{\infty}  \sum_{c=0}^{\infty}  \frac{1}{(a-1)! \, b!\, (c-1)!} \int_{\R^{(a+b + c-2)d}} \int_{\R^d}K_2^{\gamma}(\vz) \pabc(\vq^{a-1}\cup\vx, \vqb, \vq^{c-1}\cup\vz, t)  \,d\vz\, d\vq^{a-1} \, d\vqb \, d\vq^{c-1}\\
  &\quad+\sum_{a=0}^{\infty} \sum_{b=0}^{\infty}  \sum_{c=0}^{\infty}  \frac{1}{a! \, b!\, c!} \int_{\R^{(a+b + c)d}}\brac{\int_{\R^{2d}} m_2\paren{\vx,\vy\vert\vz}K_2^{\gamma}(\vz)
    p^{(a,b,c+1)}\paren{\vqa, \vqb \vqc \cup \vz, t} d\vy d\vz} d\vqa \, d\vqb \, d\vqc \\
     &\quad+\sum_{a=0}^{\infty} \sum_{b=0}^{\infty}  \sum_{c=0}^{\infty}  \frac{1}{(a-2)! \, (b-1)!\, c!}  \int_{\R^{(a+b + c-3)d}}\brac{\int_{\R^{d}} K_2^{\gamma}\paren{\vz}
    p^{(a-1,b-1,c+1)}\paren{\vq^{a-2}\cup\vx, \vq^{b-1}, \vqc \cup \vz, t} \, d\vz}  \\
    &\qquad\qquad \, d\vq^{a-2} \, d\vq^{b-1} \, d\vqc\\
    & = \int_{\R^{d}} \brac{\int_{\R^d}m_2\paren{\vx,\vy\vert\vz}K_2^{\gamma}(\vz) d\vy } \avg{C(\vz,t)} d\vz .
\end{align*}
In $(\RN{3})$, we used that $\int_{\R^{2d}} m_2(\vx , \vy \vert \vz ) d\vx \, d\vy = 1$. In the second to last line, the first and third term cancel by shifting indexes.

In summary, the average concentration of species A satisfies
\begin{align*}
    \partial_t \avg{A(\vx,t)} &= D_1 \lap_{\vx} \avg{A(\vx,t)} - \int_{\R^{d}} K_1^{\vec{\zeta}}\paren{\vx,\vy}\avg{A(\vx,t) B(\vy,t)} \, d\vy \\
    &\qquad+ \int_{\R^{d}} \brac{\int_{\R^d}m_2\paren{\vx,\vy\vert\vz}K_2^{\vec{\zeta}}(\vz) d\vy } \avg{C(\vz,t)} d\vz.
\end{align*}
Following similar arguments, one can derive equations for the average concentration of each species, given by~\eqref{Eq:Mean_Field_Model_From_Foward_Eqs0} as claimed.
This concludes the proof of the Proposition.

\end{proof}


\section{Proof of Theorem \ref{thm:regularityABC}}\label{A:ProofPDE}
\begin{proof}[Proof of Theorem \ref{thm:regularityABC}]

Due to the linearity of the equation, the proof of existence and uniqueness is standard here, so we only present a sketch of the argument for completeness.

Notice that the operator $\diffop$ defined in (\ref{eq:multipartDiffop}) generates a contractive analytic semigroup on $F$, denoted by $\{e^{t\mathcal{L}}\}_{t\geq 0}$. Now, since by Lemma \ref{lem:opBddLip}, $\Rp,\Rm$ are Lipschitz continuous, existence of a unique local mild solution to (\ref{eq:multipartABtoCEqs}), $\vP \in\mathcal{C}([0,t_{0});F)$ follows by the standard Picard-Lindel\"{o}f theorem for equations with values in Banach spaces if the initial condition satisfies $\vP_{0}\in F$.

Next we establish global existence of a unique mild solution. The boundedness of the linear operators $\Rp,\Rm$ by Lemma~\ref{lem:opBddLip}, together with the contraction property of the semigroup $t\mapsto e^{t\diffop}$, implies that
\begin{align}
\|\vP(t)\|_{F}&\leq \|\vP_{0}\|_{F}+C\int_{0}^{t}\|\vP(s)\|_{F}\,ds
\end{align}
and a subsequent Gronwall lemma yields the bound $\|\vP(t)\|_{F}\leq \|\vP_{0}\|_{F} e^{Ct}$. This bound allows us, by choosing $t_0$ small enough, to extend the solution from the interval $[0,t_0)$ to the interval $[0,\infty)$. Hence, a unique global mild solution $\vP\in\mathcal{C}([0,\infty);F)$ exists.

We actually have stronger regularity for the solution, $\vP(t)$. This is a direct consequence of the contraction and regularization properties of the semigroup $\{e^{t\diffop}\}_{t\geq 0}$. Indeed, since $\vP(t)$ is in $L^{2}(X)$, Lemma~\ref{lem:opBddLip} gives that $\Rp(\vP)$ and $\Rm(\vP)$ are both in $L^{2}(X)$. If, in addition, the initial condition $\vP_{0}\in H^1(X)$, then the mild form of the solution together with standard parabolic estimates (see estimates 3.1 in Chapter I.V, Section 3 of \cite{Ladyzenskaja}),  gives that $\vP\in\mathcal{C}([0,\infty);H^2(X))$.

Note, $\vP(t)$ can be viewed as a probability density. Indeed, we have that if $\vP_0\geq 0$,  $\vP(t)$ is, by Lemma~\ref{lem:soluNonNegative}, always non-negative for all $t\geq0$. Second, if $\vP_0$ satisfies the normalization condition,
\begin{equation*}
  \sum_{a=0}^{\infty} \sum_{b=0}^{\infty} \sum_{c=0}^{\infty} \brac{ \frac{1}{a!\, b!\, c!}
  \int_{\R^{da}} \int_{\R^{db}} \int_{\R^{dc}} \pabc_0\paren{\vqa,\vqb,\vqc} \, d\vqc \, d\vqb \, d\vqa
  } = 1,
\end{equation*}
then the same normalization condition holds for $\vP(t), t\geq0$,
\begin{equation*}
  \sum_{a=0}^{\infty} \sum_{b=0}^{\infty} \sum_{c=0}^{\infty} \brac{ \frac{1}{a!\, b!\, c!}
  \int_{\R^{da}} \int_{\R^{db}} \int_{\R^{dc}} \pabc\paren{\vqa,\vqb,\vqc,t} \, d\vqc \, d\vqb \, d\vqa
  } = 1,
\end{equation*}
by Lemma \ref{lem:Normalization}. This concludes the proof of the theorem.
\end{proof}

\begin{lemma}\label{lem:opBddLip}
We have that the operators $\Rp$ and $\Rm$ are bounded linear, Lipschitz continuous operators on $F$. Namely, for $\vG = \{g^{(a,b,c)}(\vqa,\vqb,\vqc)\}_{a,b,c=0}^{\infty} \in F$,
\begin{align}
\|\Rp (\vG)\|_F&\leq C \|\vG\|_{F},\nonumber\\
\|\Rm (\vG)\|_F&\leq C \|\vG\|_{F},\nonumber\\
\|\Rp (\vG_1)-\Rp (\vG_2)\|_F&\leq C \|\vG_1-\vG_2\|_{F},\nonumber\\
\|\Rm (\vG_1)-\Rm (\vG_2)\|_F&\leq C \|\vG_1-\vG_2\|_{F}.
\end{align}
\end{lemma}

\begin{proof}
We'll only show the first two estimates hold. The Lipschitz conditions on $\Rp$ and $\Rm$ follow directly from $\Rp$ and $\Rm$ being bounded and linear. Assume that the initial number of particles are $a_0, b_0$ and $c_0$ for species A, B and C respectively. We then have that the following upper bounds hold for all times $0\leq a\leq a_0 + c_0 : = a_{max}$,  $0\leq b\leq b_0 + c_0: = b_{max}$ and  $0\leq c\leq a_0\wedge b_0 + c_0: = c_{max}$.

By definition of the norm on Fock Space and the definition of $\Rp$, we have that
\begin{align}\label{eq:RpBdd1}
\|\Rp (\vG)\|_F^2 & =  \sum_{a=0}^{\infty} \sum_{b=0}^{\infty}  \sum_{c=0}^{\infty} \frac{1}{a! \, b! \, c!} \int_{\R^{(a+b+c)d}} \left(\paren{\Rp\vP}_{a,b,c}(\vqa,\vqb, \vqc) \right)^2 \, d\vqa \, d\vqb \, d\vqc,\nonumber\\
&
 =  \sum_{a=0}^{\infty} \sum_{b=0}^{\infty}  \sum_{c=0}^{\infty} \frac{1}{a! \, b! \, c!} \int_{\R^{(a+b+c)d}} \left(-\paren{\sum_{l=1}^a \sum_{m=1}^bK_1^{\gamma} \paren{\vqa_{l}, \vqb_{m}}} \gabc(\vqa,\vqb,\vqc,t)\right. \nonumber\\
&\quad
+ \left.\sum_{n=1}^{c} \!\brac{\int_{\R^{2d}} m_1(\vqc_n \vert \vx, \vy) K_1^{\gamma} \paren{\vx, \vy}
      g^{(a+1,b+1,c-1)}(\vqa \cup \vx, \vqb \cup \vy, \vqc \setminus \vqc_n, t) d\vx d\vy} \right)^2 \, d\vqa \, d\vqb \, d\vqc,\nonumber\\
      &
\leq  \sum_{a=0}^{\infty} \sum_{b=0}^{\infty}  \sum_{c=0}^{\infty} \frac{1}{a! \, b! \, c!} \int_{\R^{(a+b+c)d}} \left( a_{max}b_{max} C(K) \gabc(\vqa,\vqb,\vqc,t)\right)^2  \, d\vqa \, d\vqb \, d\vqc \nonumber\\
&\quad
+ \sum_{a=0}^{\infty} \sum_{b=0}^{\infty}  \sum_{c=0}^{\infty} \frac{1}{a! \, b! \, c!} \int_{\R^{(a+b+c)d}} \left( \sum_{n=1}^{c} \!\brac{\int_{\R^{2d}} m_1(\vqc_n \vert \vx, \vy) K_1^{\gamma} \paren{\vx, \vy}\right.\right.\nonumber\\
 &
 \qquad  \qquad \times \left.\left.    g^{(a+1,b+1,c-1)}(\vqa \cup \vx, \vqb \cup \vy, \vqc \setminus \vqc_n, t) d\vx d\vy} \right)^2 \, d\vqa \, d\vqb \, d\vqc.
\end{align}

Without loss of generality, let us assume $0<\alpha_i\leq 1$ in Assumption \ref{Assume:measureTwo2One}. Now denote $C_1 = \left( a_{max}b_{max} C(K)\right)^2<\infty$, and substitute the specific form of $$m_1(\vqc_n \vert \vx, \vy)= \sum_{i=1}^{I}p_i\times \delta\left(\vqc_n-(\alpha_i \vx +(1-\alpha_i)\vy)\right)$$ into (\ref{eq:RpBdd1}). We obtain
\begin{align}\label{eq:RpBdd2}
&\|\Rp (\vG)\|_F^2  \leq  C_1\|\vG\|_F^2 +  \sum_{a=0}^{\infty} \sum_{b=0}^{\infty}  \sum_{c=0}^{\infty} \frac{1}{a! \, b! \, c!} \int_{\R^{(a+b+c)d}} \left( \sum_{n=1}^{c} \!\brac{\int_{\R^{2d}} \sum_{i=1}^{I}p_i \times K_1^{\gamma}\paren{\vx, \vy}\right.\right.\nonumber\\
 &
 \qquad \times \left.\left.   \delta\left(\vqc_n-(\alpha_i \vx +(1-\alpha_i)\vy)\right)\times g^{(a+1,b+1,c-1)}(\vqa \cup \vx, \vqb \cup \vy, \vqc \setminus \vqc_n, t) d\vx d\vy} \right)^2 \, d\vqa \, d\vqb \, d\vqc.\nonumber\\
 &
 \leq  C_1\|\vG\|_F^2 +  \sum_{a=0}^{\infty} \sum_{b=0}^{\infty}  \sum_{c=0}^{\infty} \frac{1}{a! \, b! \, c!} \int_{\R^{(a+b+c)d}} \left( \sum_{n=1}^{c} \!\brac{\int_{\R^{d}} \sum_{i=1}^{I}p_i \times K_1^{\gamma}\paren{\frac{1}{\alpha_i}\left(\vqc_n - (1-\alpha_i)\vy\right), \vy}\right.\right.\nonumber\\
 &
 \qquad \times \left.\left.   g^{(a+1,b+1,c-1)}(\vqa \cup \frac{1}{\alpha_i}\left(\vqc_n - (1-\alpha_i)\vy\right) , \vqb \cup \vy, \vqc \setminus \vqc_n, t) d\vy} \right)^2 \, d\vqa \, d\vqb \, d\vqc.\nonumber\\
 &
 \leq  C_1\|\vG\|_F^2 +  \sum_{a=0}^{\infty} \sum_{b=0}^{\infty}  \sum_{c=0}^{\infty} \frac{1}{a! \, b! \, c!} \int_{\R^{(a+b+c-1)d}}  c\sum_{i=1}^{I}p_i^2\sum_{n=1}^{c} \int_{\R^d}\left( \!\brac{\int_{\R^{d}}  K_1^{\gamma}\paren{\frac{1}{\alpha_i}\left(\vqc_n - (1-\alpha_i)\vy\right), \vy}\right.\right.\nonumber\\
 &
 \qquad \times \left.\left.   g^{(a+1,b+1,c-1)}(\vqa \cup \frac{1}{\alpha_i}\left(\vqc_n - (1-\alpha_i)\vy\right) , \vqb \cup \vy, \vqc \setminus \vqc_n, t) d\vy}^2 \, d\vqc_n  \right) \, d\vqa \, d\vqb \, d\vqc\setminus \vqc_n.\nonumber\\
 &
 \leq  C_1\|\vG\|_F^2 +  \sum_{a=0}^{\infty} \sum_{b=0}^{\infty}  \sum_{c=0}^{\infty} \frac{1}{a! \, b! \, (c-1)!} \int_{\R^{(a+b+c-1)d}}  \sum_{i=1}^{I}p_i^2 \sum_{n=1}^{c} \int_{\R^d} \!\left( \int_{\R^{d}} \tilde{K}_1^{\gamma}\paren{|\frac{1}{\alpha_i}\left(\vqc_n - (1-\alpha_i)\vy\right)- \vy|}^2 \, d\vy \right) \nonumber\\
 &
 \qquad \times \left( \int_{\R^{d}} \left|g^{(a+1,b+1,c-1)}\left(\vqa \cup \frac{1}{\alpha_i}\left(\vqc_n - (1-\alpha_i)\vy\right) , \vqb \cup \vy, \vqc \setminus \vqc_n, t\right)\right|^2 d\vy \right)  \, d\vqc_n \, d\vqa \, d\vqb \, d\vqc\setminus \vqc_n.\nonumber\\
 &
 \leq  C_1\|\vG\|_F^2 +  \sum_{a=0}^{\infty} \sum_{b=0}^{\infty}  \sum_{c=0}^{\infty} \frac{1}{a! \, b! \, (c-1)!} \int_{\R^{(a+b+c-1)d}}  \sum_{i=1}^{I}p_i^2 \left( \int_{\R^{d}}\alpha_i \tilde{K}_1^{\gamma}\paren{|\vz|}^2 \, d\vz \right) \nonumber\\
 &
 \qquad \times  \sum_{n=1}^{c} \int_{\R^d} \left( \int_{\R^{d}} \left|g^{(a+1,b+1,c-1)}\left(\vqa \cup \frac{1}{\alpha_i}\left(\vqc_n - (1-\alpha_i)\vy\right) , \vqb \cup \vy, \vqc \setminus \vqc_n, t\right)\right|^2 d\vy \right)  \, d\vqc_n \, d\vqa \, d\vqb \, d\vqc\setminus \vqc_n.\nonumber\\
 &
 \leq  C_1\|\vG\|_F^2 +  \sum_{a=0}^{\infty} \sum_{b=0}^{\infty}  \sum_{c=0}^{\infty} \frac{1}{a! \, b! \, (c-1)!} \int_{\R^{(a+b+c-1)d}}  \sum_{i=1}^{I}p_i^2 \left( \int_{\R^{d}}\alpha_i \tilde{K}_1^{\gamma}\paren{|\vz|}^2 \, d\vz \right) \nonumber\\
 &
 \qquad \times  \sum_{n=1}^{c}  \int_{\R^d} \left( \int_{\R^{d}}  \left|g^{(a+1,b+1,c-1)}\left(\vqa \cup \frac{1}{\alpha_i}\left(\vqc_n - (1-\alpha_i)\vy\right) , \vqb \cup \vy, \vqc \setminus \vqc_n, t\right)\right|^2 \, d\vqc_n \right)\,d\vy   \, d\vqa \, d\vqb \, d\vqc\setminus \vqc_n.\nonumber\\
  &
 \leq  C_1\|\vG\|_F^2 +  \sum_{a=0}^{\infty} \sum_{b=0}^{\infty}  \sum_{c=0}^{\infty} \frac{1}{a! \, b! \, (c-1)!} \int_{\R^{(a+b+c-1)d}}  \sum_{i=1}^{I}p_i^2 \left( \int_{\R^{d}}\alpha_i \tilde{K}_1^{\gamma}\paren{|\vz|}^2 \, d\vz \right) \nonumber\\
 &
 \qquad \times  \sum_{n=1}^{c}  \alpha_i\int_{\R^d} \left( \int_{\R^{d}}  \left|g^{(a+1,b+1,c-1)}\left(\vqa \cup \vx , \vqb \cup \vy, \vqc \setminus \vqc_n, t\right)\right|^2 \, d\vx \right)\,d\vy   \, d\vqa \, d\vqb \, d\vqc\setminus \vqc_n.\nonumber\\
 &
 \leq  C_1\|\vG\|_F^2 +  \sum_{a=0}^{\infty} \sum_{b=0}^{\infty}  \sum_{c=0}^{\infty} \frac{Iabc}{(a+1)! \, (b+1)! \, (c-1)!}\left( \int_{\R^{d}}\tilde{K}_1^{\gamma}\paren{|\vz|}^2 \, d\vz \right) \nonumber\\
 &
 \qquad  \times \int_{\R^{(a+b+c+1)d}}  \left( \left|g^{(a+1,b+1,c-1)}\left(\vq^{a+1} , \vq^{b+1}, \vq^{c-1}, t \right)\right|^2\right)  d\vq^{a+1}\, d\vq^{b+1} \, d\vq^{c-1}.\nonumber\\
 &
 \leq  C_1\|\vG\|_F^2 +  C_2\|\vG\|_F^2 = (C_1 + C_2 )\|\vG\|_F^2,\nonumber\\
\end{align}
where $C_2 = I\times a_{max}b_{max}c_{max} \left( \int_{\R^{d}}\tilde{K}_1^{\gamma}\paren{|\vz|}^2 \, d\vz \right)<\infty$ by Assumption \ref{assumption:twoToOneRate}.

Similarly, by definition of the norm on Fock Space and the definition of $\Rm$, we have that
\begin{align}\label{eq:RpBdd3}
\|\Rm (\vG)\|_F^2 & =  \sum_{a=0}^{\infty} \sum_{b=0}^{\infty}  \sum_{c=0}^{\infty} \frac{1}{a! \, b! \, c!} \int_{\R^{(a+b+c)d}} \left(\paren{\Rm\vG}_{a,b,c}(\vqa,\vqb, \vqc) \right)^2 \, d\vqa \, d\vqb \, d\vqc,\nonumber\\
&
 =  \sum_{a=0}^{\infty} \sum_{b=0}^{\infty}  \sum_{c=0}^{\infty} \frac{1}{a! \, b! \, c!} \int_{\R^{(a+b+c)d}} \left( - \paren{ \sum_{n=1}^{c} K_2^{\gamma}(\vqc_n)} \gabc(\vqa, \vqb, \vqc, t)\right. \nonumber\\
&\quad
+ \left.\sum_{l=1}^a \sum_{m=1}^b \brac{\int_{\R^d} m_2\paren{\vqa_l,\vqb_m\vert\vz}K_2^{\gamma}(\vz)
    g^{(a-1,b-1,c+1)}\paren{\vqa\setminus \vqa_l, \vqb\setminus \vqb_m, \vqc \cup \vz, t} d\vz}\right)^2 \, d\vqa \, d\vqb \, d\vqc,\nonumber\\
      &
\leq  \sum_{a=0}^{\infty} \sum_{b=0}^{\infty}  \sum_{c=0}^{\infty} \frac{1}{a! \, b! \, c!} \int_{\R^{(a+b+c)d}} \left( c_{max} C(K) \gabc(\vqa,\vqb,\vqc,t)\right)^2  \, d\vqa \, d\vqb \, d\vqc \nonumber\\
&\quad
+ \sum_{a=0}^{\infty} \sum_{b=0}^{\infty}  \sum_{c=0}^{\infty} \frac{1}{a! \, b! \, c!} \int_{\R^{(a+b+c)d}} \left(\sum_{l=1}^a \sum_{m=1}^b \brac{\int_{\R^d} m_2\paren{\vqa_l,\vqb_m\vert\vz}K_2^{\gamma}(\vz) \right.\right.\nonumber\\
 &
 \qquad  \qquad \times \left.\left.   g^{(a-1,b-1,c+1)}\paren{\vqa\setminus \vqa_l, \vqb\setminus \vqb_m, \vqc \cup \vz, t} d\vz}\right)^2  \, d\vqa \, d\vqb \, d\vqc.
\end{align}
Now denote $C_3 = \left( c_{max}C(K)\right)^2<\infty$, and substitute the specific form of
\[
m_2\paren{\vqa_l,\vqb_m\vert\vz} = \rho(|\vqa_l-\vqb_m|) \sum_{i=1}^{I}p_i\times\delta\left(\vz-(\alpha_i \vqa_i +(1-\alpha_i)\vqb_m)\right),
\]
into (\ref{eq:RpBdd3}). We obtain
\begin{align}\label{eq:RpBdd4}
&\|\Rm (\vG)\|_F^2
\leq C_3\|\vG\|_F^2  +  \sum_{a=0}^{\infty} \sum_{b=0}^{\infty}  \sum_{c=0}^{\infty} \frac{1}{a! \, b! \, c!} \int_{\R^{(a+b+c)d}} \left(\sum_{l=1}^a \sum_{m=1}^b \brac{\int_{\R^d} m_2\paren{\vqa_l,\vqb_m\vert\vz}K_2^{\gamma}(\vz) \right.\right.\nonumber\\
 &
 \qquad  \qquad \times \left.\left.   g^{(a-1,b-1,c+1)}\paren{\vqa\setminus \vqa_l, \vqb\setminus \vqb_m, \vqc \cup \vz, t} d\vz}\right)^2  \, d\vqa \, d\vqb \, d\vqc.\nonumber\\
 &
\leq C_3\|\vG\|_F^2  +  \sum_{a=0}^{\infty} \sum_{b=0}^{\infty}  \sum_{c=0}^{\infty} \frac{1}{a! \, b! \, c!} \int_{\R^{(a+b+c)d}} \left(ab\sum_{l=1}^a \sum_{m=1}^b \brac{\int_{\R^d} m_2\paren{\vqa_l,\vqb_m\vert\vz}K_2^{\gamma}(\vz) \right.\right.\nonumber\\
 &
 \qquad  \qquad \times \left.\left.   g^{(a-1,b-1,c+1)}\paren{\vqa\setminus \vqa_l, \vqb\setminus \vqb_m, \vqc \cup \vz, t} d\vz}^2\right)  \, d\vqa \, d\vqb \, d\vqc.\nonumber\\
 &
\leq C_3\|\vG\|_F^2  +  \sum_{a=0}^{\infty} \sum_{b=0}^{\infty}  \sum_{c=0}^{\infty} \frac{1}{(a-1)! \, (b-1)! \, c!} \int_{\R^{(a+b+c)d}} \left(\sum_{l=1}^a \sum_{m=1}^b C(K)^2 \brac{\int_{\R^d}\sum_{i=1}^{I}  p_i \times \rho(|\vqa_l-\vqb_m|) \right.\right.\nonumber\\
 &
 \qquad  \qquad \times\delta\left(\vz-(\alpha_i \vqa_i +(1-\alpha_i)\vqb_m)\right) \left.\left.   \times g^{(a-1,b-1,c+1)}\paren{\vqa\setminus \vqa_l, \vqb\setminus \vqb_m, \vqc \cup \vz, t} d\vz}^2\right)  \, d\vqa \, d\vqb \, d\vqc.\nonumber\\
 &
\leq C_3\|\vG\|_F^2  +  \sum_{a=0}^{\infty} \sum_{b=0}^{\infty}  \sum_{c=0}^{\infty} \frac{C(K)^2 }{(a-1)! \, (b-1)! \, c!} \int_{\R^{(a+b+c)d}} \left(\sum_{l=1}^a \sum_{m=1}^b \sum_{i=1}^{I}  p_i^2 \times \rho(|\vqa_l-\vqb_m|)^2 \right.\nonumber\\
 &
 \qquad  \qquad  \left. \times \left|g^{(a-1,b-1,c+1)}\paren{\vqa\setminus \vqa_l, \vqb\setminus \vqb_m, \vqc \cup (\alpha_i \vqa_i +(1-\alpha_i)\vqb_m), t}\right|^2\right)\, d\vqa_l \, d \vqb_m   \, d\vqa\setminus \vqa_l \, d\vqb\setminus \vqb_m \, d\vqc.\nonumber\\
 &
\leq C_3\|\vG\|_F^2  +  \sum_{a=0}^{\infty} \sum_{b=0}^{\infty}  \sum_{c=0}^{\infty} \frac{C(K)^2 ab}{(a-1)! \, (b-1)! \, c!}  \sum_{i=1}^{I}  p_i^2 \times \int_{\R^{(a+b+c)d}} \left(\rho(|\vx - \vy|)^2 \right.\nonumber\\
 &
 \qquad  \qquad  \left. \times \left|g^{(a-1,b-1,c+1)}\paren{\vq^{a-1}, \vq^{b-1}, \vqc \cup (\alpha_i \vx +(1-\alpha_i)\vy), t}\right|^2\right)\, d\vx \, d \vy   \, d\vq^{a-1} \, d\vq^{b-1} \, d\vqc.\nonumber\\
 &
\leq C_3\|\vG\|_F^2  +  \sum_{a=0}^{\infty} \sum_{b=0}^{\infty}  \sum_{c=0}^{\infty} \frac{C(K)^2 ab}{(a-1)! \, (b-1)! \, c!}  \sum_{i=1}^{I}  p_i^2 \times \int_{\R^{(a+b+c)d}} \left(\rho(|\vw|)^2 \right.\nonumber\\
 &
 \qquad  \qquad  \left. \times \left|g^{(a-1,b-1,c+1)}\paren{\vq^{a-1}, \vq^{b-1}, \vqc \cup (\alpha_i \vw +\vy), t}\right|^2\right)\, d\vw \, d \vy   \, d\vq^{a-1} \, d\vq^{b-1} \, d\vqc.\nonumber\\
 &
\leq C_3\|\vG\|_F^2  +  \sum_{a=0}^{\infty} \sum_{b=0}^{\infty}  \sum_{c=0}^{\infty} \frac{C(K)^2 ab(c+1)}{(a-1)! \, (b-1)! \, (c+1)!}  \sum_{i=1}^{I}  p_i^2 \times \left(\int_{\R^d}\rho(|\vw|)^2 \, d\vw\right) \nonumber\\
 &
 \qquad  \qquad  \left. \times\int_{\R^{(a+b+c-1)d}}   \left|g^{(a-1,b-1,c+1)}\paren{\vq^{a-1}, \vq^{b-1}, \vqc \cup\vz, t}\right|^2\right)\, \, d \vz   \, d\vq^{a-1} \, d\vq^{b-1} \, d\vqc.\nonumber\\
 &\qquad\qquad  \text{(where $\vz = \alpha_i \vw +\vy$)}\nonumber\\
 &
\leq C_3\|\vG\|_F^2  +  C_4\|\vG\|_F^2\nonumber = (C_3+C_4)\|\vG\|_F^2,\\
\end{align}
where $C_4 = I\times C(K)^2 a_{max}b_{max}(c_{max}+1)\left(\int_{\R^d}\rho(|\vw|)^2 \, d\vw\right) < \infty$ by Assumption \ref{assumption:oneToTwoRate}.
\end{proof}

\begin{lemma}\label{lem:Normalization}
Assume that the solution $\vP\in\mathcal{C}([0,\infty);H^2(X))$ to (\ref{eq:multipartABtoCEqs}) exists and is unique. If the normalization condition holds for the initial condition $\vP_0$, i.e. that
\begin{equation*}
  \sum_{a=0}^{\infty} \sum_{b=0}^{\infty} \sum_{c=0}^{\infty} \brac{ \frac{1}{a!\, b!\, c!}
  \int_{\R^{da}} \int_{\R^{db}} \int_{\R^{dc}} \pabc_0\paren{\vqa,\vqb,\vqc} \, d\vqc \, d\vqb \, d\vqa
  } = 1,
\end{equation*}
then we have that the normalization condition holds for $\vP(t)$ for all $t \geq 0$, i.e. that
\begin{equation}\label{eq:normalization}
  \sum_{a=0}^{\infty} \sum_{b=0}^{\infty} \sum_{c=0}^{\infty} \brac{ \frac{1}{a!\, b!\, c!}
  \int_{\R^{da}} \int_{\R^{db}} \int_{\R^{dc}} \pabc\paren{\vqa,\vqb,\vqc,t} \, d\vqc \, d\vqb \, d\vqa
  } = 1.
\end{equation}
\end{lemma}
\begin{proof}
By assumption, we have that the normalization condition holds for the initial condition. Furthermore, we have
\begin{align*}
&\frac{\partial}{\partial t}\sum_{a=0}^{\infty} \sum_{b=0}^{\infty} \sum_{c=0}^{\infty} \brac{ \frac{1}{a!\, b!\, c!}
  \int_{\R^{da}} \int_{\R^{db}} \int_{\R^{dc}} \pabc\paren{\vqa,\vqb,\vqc,t} \, d\vqc \, d\vqb \, d\vqa
  } \\
&
= \sum_{a=0}^{\infty} \sum_{b=0}^{\infty} \sum_{c=0}^{\infty} \brac{ \frac{1}{a!\, b!\, c!}
  \int_{\R^{da}} \int_{\R^{db}} \int_{\R^{dc}}  \frac{\partial}{\partial t}\pabc\paren{\vqa,\vqb,\vqc,t} \, d\vqc \, d\vqb \, d\vqa
  } \\
  &
  = \sum_{a=0}^{\infty} \sum_{b=0}^{\infty} \sum_{c=0}^{\infty} \brac{ \frac{1}{a!\, b!\, c!}
  \int_{\R^{da}} \int_{\R^{db}} \int_{\R^{dc}}  \paren{\diffop \vP}_{a,b,c}\paren{\vqa,\vqb,\vqc,t} \, d\vqc \, d\vqb \, d\vqa
  } \\
  & \quad
  + \sum_{a=0}^{\infty} \sum_{b=0}^{\infty} \sum_{c=0}^{\infty} \brac{ \frac{1}{a!\, b!\, c!}
  \int_{\R^{da}} \int_{\R^{db}} \int_{\R^{dc}}   \paren{\Rp  \vP}_{a,b,c}\paren{\vqa,\vqb,\vqc,t} \, d\vqc \, d\vqb \, d\vqa
  } \\
    & \quad
    + \sum_{a=0}^{\infty} \sum_{b=0}^{\infty} \sum_{c=0}^{\infty} \brac{ \frac{1}{a!\, b!\, c!}
  \int_{\R^{da}} \int_{\R^{db}} \int_{\R^{dc}}   \paren{\Rm \vP}_{a,b,c}\pabc\paren{\vqa,\vqb,\vqc,t} \, d\vqc \, d\vqb \, d\vqa
  }.
\end{align*}
Note, in the last equality the first term is zero using the divergence theorem and that $\vP\in\mathcal{C}([0,\infty);H^2(X))$. The second term is
\begin{align*}
& \sum_{a=0}^{\infty} \sum_{b=0}^{\infty} \sum_{c=0}^{\infty} \brac{ \frac{1}{a!\, b!\, c!}
  \int_{\R^{da}} \int_{\R^{db}} \int_{\R^{dc}}   \paren{\Rp \vP}_{a,b,cv}\paren{\vqa,\vqb,\vqc,t} \, d\vqc \, d\vqb \, d\vqa
  } \\
    &
    = \sum_{a=0}^{\infty} \sum_{b=0}^{\infty} \sum_{c=0}^{\infty} \frac{1}{a!\, b!\, c!}\brac{
  \int_{\R^{da}} \int_{\R^{db}} \int_{\R^{dc}}   -\paren{\sum_{l=1}^a \sum_{m=1}^bK_1^{\gamma} \paren{\vqa_{l}, \vqb_{m}}} \pabc(\vqa,\vqb,\vqc,t) \, d\vqc \, d\vqb \, d\vqa \right. \\
    &\quad +\left.\int_{\R^{da}} \int_{\R^{db}} \int_{\R^{dc}}  \sum_{n=1}^{c} \!\brac{\int_{\R^{2d}} m_1(\vqc_n \vert \vx, \vy) K_1^{\gamma} \paren{\vx, \vy}
      p^{(a+1,b+1,c-1)}(\vqa \cup \vx, \vqb \cup \vy, \vqc \setminus \vqc_n, t) d\vx d\vy}\!\! \, d\vqc \, d\vqb \, d\vqa
  } \\
  &
    = \sum_{a=0}^{\infty} \sum_{b=0}^{\infty} \sum_{c=0}^{\infty}\frac{1}{a!\, b!\, c!} \brac{
  \int_{\R^{d(a+b+c)}}   -\paren{\sum_{l=1}^a \sum_{m=1}^bK_1^{\gamma} \paren{\vqa_{l}, \vqb_{m}}} \pabc(\vqa,\vqb,\vqc,t) \,d\vqc \, d\vqb \, d\vqa \right. \\
    &\quad + \left. \sum_{n=1}^{c}   \int_{\R^{d(a+b+c+1)}}   \!\brac{\int_{\R^{d}} m_1(\vqc_n \vert \vx, \vy) d\vqc_n} K_1^{\gamma} \paren{\vx, \vy}
      p^{(a+1,b+1,c-1)}(\vqa \cup \vx, \vqb \cup \vy, \vqc \setminus \vqc_n, t)  \, d\vqc \setminus d\vqc_n \, d\vqb\cup d\vy, \, d\vqa \cup d\vx
  } \\
  &
    = \sum_{a=0}^{\infty} \sum_{b=0}^{\infty} \sum_{c=0}^{\infty} \frac{1}{a!\, b!\, c!} \brac{
  \int_{\R^{d(a+b+c)}}   -\paren{\sum_{l=1}^a \sum_{m=1}^bK_1^{\gamma} \paren{\vqa_{l}, \vqb_{m}}} \pabc(\vqa,\vqb,\vqc,t) \,d\vqc \, d\vqb \, d\vqa \right. \\
    &+\left.  \int_{\R^{d(a+b+c+1)}}  \frac{c}{(a+1)(b+1)}\sum_{l=1}^{a+1} \sum_{m=1}^{b+1} K_1^{\gamma} \paren{\vq^{a+1}_l, \vq^{b+1}_m}
      p^{(a+1,b+1,c-1)}(\vq^{a+1}, \vq^{b+1}, \vq^{c-1}, t)  \, d\vq^{c-1}  \, d\vq^{b+1}, \, d\vq^{a+1}  } \\
        &
    = \sum_{a=0}^{\infty} \sum_{b=0}^{\infty} \sum_{c=0}^{\infty} \brac{ \frac{1}{a!\, b!\, c!}
  \int_{\R^{d(a+b+c)}}   -\paren{\sum_{l=1}^a \sum_{m=1}^bK_1^{\gamma} \paren{\vqa_{l}, \vqb_{m}}} \pabc(\vqa,\vqb,\vqc,t) \,d\vqc \, d\vqb \, d\vqa \right. \\
    &+\left. \frac{1}{(a+1)!\, (b+1)!\, (c-1)!}  \int_{\R^{d(a+b+c+1)}}\sum_{l=1}^{a+1} \sum_{m=1}^{b+1} K_1^{\gamma} \paren{\vq^{a+1}_l, \vq^{b+1}_m}
      p^{(a+1,b+1,c-1)}(\vq^{a+1}, \vq^{b+1}, \vq^{c-1}, t)  \, d\vq^{c-1}  \, d\vq^{b+1}, \, d\vq^{a+1}  } \\
      & = 0,
\end{align*}
by using the symmetry of $p^{(a,b,c)}$ with respect to permutations of the ordering of particle positions for particles of the \emph{same} type. A similar calculation shows that the third term is zero. Thus, the time derivative of the left-hand-side (\ref{eq:normalization}) is always zero. Combining with the normalization condition for the initial condition, we have that (\ref{eq:normalization}) is true for all $t\geq0$.
\end{proof}

\begin{lemma}\label{lem:soluNonNegative}
Assuming there exists a unique solution $\vP\in\mathcal{C}([0,\infty);H^2(X))$ and $\vP_0\geq 0$, the solution components $\pabc\paren{\vqa,\vqb,\vqc,t}$ are always non-negative for all $t\geq0$.
\end{lemma}

\begin{proof}
First, we consider $\gabc\paren{\vqa,\vqb,\vqc,t}$ satisfying the following decoupled linear PDEs with initial condition $\gabc\paren{\vqa,\vqb,\vqc,0} = p_0^{(a,b,c)}\paren{\vqa,\vqb,\vqc}\geq 0$,
\begin{equation}\label{eq:truncatedSolu}
   \PD{}{t} \gabc(\vq^a,\vq^b,\vq^c, t) = (\diffop + \Rp_1 + \Rm_1) \gabc(\vq^a,\vq^b,\vq^c. t),
\end{equation}
where we define
\begin{equation}
      \Rp_1 \gabc(\vq^a,\vq^b,\vq^c. t) =
   -\paren{\sum_{l=1}^a \sum_{m=1}^bK_1^{\gamma} \paren{\vqa_{l}, \vqb_{m}}} \gabc(\vqa,\vqb,\vqc,t)\nonumber
   \end{equation}
and
\begin{equation} 
  \Rm_1 \gabc(\vq^a,\vq^b,\vq^c. t)= - \paren{ \sum_{n=1}^{c} K_2^{\gamma}(\vqc_n)} \gabc(\vqa, \vqb, \vqc, t).\nonumber
\end{equation}
Since $0\leq K_1^\gamma\leq C(K)$ and  $0\leq K_2^\gamma\leq C(K)$, we can then obtain via a comparison argument for semilinear equations that the solution to Eq (\ref{eq:truncatedSolu}), $\gabc\paren{\vqa,\vqb,\vqc,t} \geq 0$.

Let us further define
\begin{equation} 
      \Rp_2 \gabc(\vq^a,\vq^b,\vq^c. t) =
         \sum_{n=1}^{c} \!\brac{\int_{\R^{2d}} m_1(\vqc_n \vert \vx, \vy) K_1^{\gamma} \paren{\vx, \vy}
      g^{(a+1,b+1,c-1)}(\vqa \cup \vx, \vqb \cup \vy, \vqc \setminus \vqc_n, t) d\vx d\vy}\!\!,
  \nonumber
\end{equation}
and
\begin{equation} 
  \Rm_2 \gabc(\vq^a,\vq^b,\vq^c. t)= \sum_{l=1}^a \sum_{m=1}^b \brac{\int_{\R^d} m_2\paren{\vqa_l,\vqb_m\vert\vz}K_2^{\gamma}(\vz)
    g^{(a-1,b-1,c+1)}\paren{\vqa\setminus \vqa_l, \vqb\setminus \vqb_m, \vqc \cup \vz, t} d\vz}.\nonumber
\end{equation}
We then set $\Rp = \Rp_1 + \Rp_2$ and $\Rm = \Rm_1 + \Rm_2$. Due to the positive mapping property of the operators $\Rp_2$ and  $\Rm_2$, we shall have for the function $\gabc\paren{\vqa,\vqb,\vqc,t}$ that
\begin{equation}
   \PD{}{t} \gabc(\vq^a,\vq^b,\vq^c, t) \leq (\diffop + \Rp + \Rm) \gabc(\vq^a,\vq^b,\vq^c. t).\nonumber
\end{equation}
Hence, again utilizing the comparison principle for semilinear PDEs, we obtain that
\[
0\leq  \gabc(\vq^a,\vq^b,\vq^c,t)\leq  \pabc(\vq^a,\vq^b,\vq^c, t),
\]
i.e. the non-negativity of our solution, concluding the proof.
\end{proof}

\section{Placement density integrals in~\eqref{Eq:density_formula}}\label{A:PlacementInts}

In this appendix we expand out the formal notation used for the inner integrand in the integrals
\begin{equation*}
I = \int_{\tilde{\vec{x}} \in\mathbb{X}^{(\ell)}}  K_\ell(\tilde{\vec{x}}) \left( \int_{\vy \in \mathbb{Y}^{(\ell)}}   \delta_{x}(y_r^{(j)}) m_\ell(\vec{y}\, | \,\tilde{\vec{x}}) \,d \vec{y} \right)
 \left( \Pi_{k = 1}^J \Pi_{s = 1}^{\alpha_{\ell k}}  \rho_{k}(\tilde{x}_s^{(k)}, t)\right) \, d\tilde{\vec{x}}
\end{equation*}
for several choices of the placement density, $m_{\ell}$.

For a first order reaction of the form $S_{i} \to S_{j}$, by Assumption~\ref{Assume:measureOne2One} we have
\begin{equation*}
m_{\ell}(y | x) = \delta(x-y),
\end{equation*}
so that in the equation for $\rho_{j}(x,t)$, $I$ becomes
\begin{align*}
I &= \int_{\R^{d}}  K_\ell(\tilde{x}) \left( \int_{\R^{d}}   \delta(x-y) \delta(y-\tilde{x}) \,dy \right) \rho_{i}(\tilde{x},t) \, d\tilde{x} \\
&= \int_{\R^{d}}  K_\ell(\tilde{x})  \delta(x-\tilde{x})  \rho_{i}(\tilde{x},t) \, d\tilde{x} \\
&= K_{\ell}(x) \rho_{i}(x,t).
\end{align*}
The first order reaction term that appears in~\eqref{Eq:density_formula_birthDeath} is of this form.

For a second order bimolecular reaction of the form $S_{i} + S_{k} \to S_{j}$, by Assumption~\ref{Assume:measureTwo2One} we have
\begin{equation*}
m_{\ell}(z\,|\,x,y) = \sum_{n=1}^{N}p_n \times \delta\left(z-(\alpha_n x +(1-\alpha_n)y)\right),
\end{equation*}
so that in the equation for $\rho_{j}(x,t)$, $I$ becomes
\begin{align*}
I &= \int_{\R^{2d}} K_{\ell}(\tilde{x}_{1},\tilde{x}_{2}) \left( \int_{\R^{d}} \delta(x-y) \paren{\sum_{n=1}^{N}p_n \times \delta\left(y-(\alpha_n \tilde{x}_{1} +(1-\alpha_n)\tilde{x}_{2})\right)} \, dy \right) \rho_{i}(\tilde{x}_{1},t) \rho_{k}(\tilde{x}_{2},t) \, d\tilde{x}_{1} \, d\tilde{x}_{2} \\
&= \int_{\R^{2d}} K_{\ell}(\tilde{x}_{1},\tilde{x}_{2}) \paren{\sum_{n=1}^{N}p_n \times \delta\left(x-(\alpha_n \tilde{x}_{1} +(1-\alpha_n)\tilde{x}_{2})\right)} \rho_{i}(\tilde{x}_{1},t) \rho_{k}(\tilde{x}_{2},t) \, d\tilde{x}_{1} \, d\tilde{x}_{2} \\
&= \int_{\R^{2d}} K_{\ell}(\tilde{x}_{1},\tilde{x}_{2}) \, m_{\ell}(x | \tilde{x}_{1}, \tilde{x}_{2}) \rho_{i}(\tilde{x}_{1},t) \rho_{k}(\tilde{x}_{2},t) \, d\tilde{x}_{1} \, d\tilde{x}_{2}.
\end{align*}

The second order reaction term that appears in the equation for $\rho_{3}(x,t)$ in~\eqref{Eq:density_formula_reversible} is of this form. Note, $I$ can be further simplified to eliminate any $\delta$-function terms, giving
\begin{align*}
I &= \sum_{n=1}^{N} \frac{p_{n}}{\alpha_{n}^{d}} \int_{\R^{d}} K_{\ell}\paren{\frac{1}{\alpha_{n}}(x - (1-\alpha_{n})\tilde{x}_{2}),\tilde{x}_{2}} \rho_{i}\paren{\frac{1}{\alpha_{n}}(x - (1-\alpha_{n})\tilde{x}_{2}),t} \rho_{k}(\tilde{x}_{2},t) \, d\tilde{x}_{2}.
\end{align*}

Finally, for a two-product reaction of the form $S_{i} \to S_{j} + S_{k}$, by Assumption~\ref{Assume:measureOne2Two} we have
\begin{equation*}
m_{\ell}(x,y\,|\,z) = \rho(|x-y|) \sum_{n=1}^{N}p_n\times\delta\left(z-(\alpha_n x +(1-\alpha_n)y)\right),
\end{equation*}
so that in the equation for $S_{j}$, $I$ becomes
\begin{align*}
I &= \int_{\R^{d}} K_{\ell}(\tilde{x}) \paren{ \int_{\R^{2d}} \delta(x - y_{1})
\paren{\rho(|y_{1}-y_{2}|) \sum_{n=1}^{N}p_n\times\delta\left(\tilde{x}-(\alpha_n y_{1} +(1-\alpha_n)y_{2})\right)} \, dy_{1} dy_{2}}
\rho_{i}(\tilde{x},t) \, d\tilde{x} \\
&= \int_{\R^{d}} K_{\ell}(\tilde{x}) \paren{\int_{\R^{d}}\paren{\rho(|x-y_{2}|) \sum_{n=1}^{N}p_n\times\delta\left(\tilde{x}-(\alpha_n x +(1-\alpha_n)y_{2})\right)}dy_{2}} \rho_{i}(\tilde{x},t) \, d\tilde{x} \\
&= \int_{\R^{d}} K_{\ell}(\tilde{x}) \paren{\int_{\R^{d}} m(x,y_{2} | \tilde{x}) dy_{2}} \rho_{i}(\tilde{x},t) \, d\tilde{x}.
\end{align*}

The first order reaction term that appears in the equation for $\rho_{1}(x,t)$ in~\eqref{Eq:density_formula_reversible}
 is of this form. Note, $I$ can be further simplified to eliminate any $\delta$-function terms, giving
\begin{align*}
I = \sum_{n=1}^{N} \frac{p_{n}}{(1 - \alpha_{n})^{d}} \int_{\R^{d}} K_{\ell}(\tilde{x}) \, \rho\!\paren{\frac{x - \tilde{x}}{1 - \alpha_{n}}} \rho_{i}(\tilde{x},t) \, d\tilde{x}.
\end{align*}

\end{appendices}

\end{document}